\documentclass[reqno]{amsart}
\usepackage{amssymb,amsfonts,amstext,amsmath,amsthm}  
\usepackage{enumerate}
\usepackage{subfigure}
\usepackage{calligra}

\newtheorem{thm}{Theorem}[section]
\newtheorem{cor}[thm]{Corollary}
\newtheorem{lem}[thm]{Lemma}
\newtheorem{prop}[thm]{Proposition}
\newtheorem{rem}[thm]{Remark}

\numberwithin{equation}{section}
\newcommand{\Om}{{\Omega}}

\newcommand{\R}{{\rm I}\!{\rm R}}

\def\n{\mathbf{n}}

\begin{document} 

%\begin{frontmatter}

\title[An indefinite concave-convex equation under a Neumann boundary condition]{An indefinite concave-convex equation under a Neumann boundary condition I} 

\vspace{1cm}

%%%%%%%%%%%%%%%%%%%%%%%%%%%%%%%%%%%%%%%%%%%%%%%%%%%%%%%%%%%%%%%%%%%%%%%%

\author{Humberto Ramos Quoirin}
\address{H. Ramos Quoirin \newline Universidad de Santiago de Chile, Casilla 307, Correo 2, Santiago, Chile}
\email{\tt humberto.ramos@usach.cl}

\author{Kenichiro Umezu}
\address{K. Umezu \newline Department of Mathematics, Faculty of Education, Ibaraki University, Mito 310-8512, Japan}
\email{\tt kenichiro.umezu.math@vc.ibaraki.ac.jp}


%\thanks{The first author was supported by the FONDECYT grant 11121567}

%%%%%%%%%%%%%%%%%%%%%%%%%%%%%%%%%%%%%%%%%%%%%%%%%%%%%%%%%%%%%%%%%%%%%%%% 

\begin{abstract}
We investigate the problem 
$$-\Delta u = \lambda b(x)|u|^{q-2}u +a(x)|u|^{p-2}u \mbox{ in } \Omega, \quad \frac{\partial u}{\partial \n} = 0  \mbox{ on } \partial \Omega, \leqno{(P_\lambda)}  $$
where $\Omega$ is a bounded smooth domain in $\R^N$ ($N \geq2$), $1<q<2<p$, $\lambda \in \R$, and $a,b \in C^\alpha(\overline{\Omega})$ with $0<\alpha<1$. Under some indefinite type conditions on $a$ and $b$ we prove the existence of two nontrivial non-negative solutions for $|\lambda|$ small. We characterize then the asymptotic profiles of these solutions as $\lambda \to 0$, which implies in some cases the positivity and ordering of these solutions. In addition, this asymptotic analysis suggests the existence of a loop type subcontinuum in the non-negative solutions set. We prove in some cases the existence of such subcontinuum via a bifurcation and topological analysis of a regularized version of $(P_\lambda)$. 
\end{abstract}

%\end{frontmatter}

\maketitle

%%%%%%%%%%%%%%%%%%%%%%%%%%%%%%%%%%%%%%%%%%%%%%%%%%%%%%%%%%%%%%% 
%\vspace{-0.5in}

%$\,$\\

\section{Introduction and statements of main results}  
Let $\Omega$ be a bounded domain of $\R^N$ ($N\geq 2$) with smooth boundary $\partial \Omega$. This article is concerned with existence, non-existence, and multiplicity of non-negative solutions for the problem
$$ 
\begin{cases}
-\Delta u = \lambda b(x)|u|^{q-2}u+a(x)|u|^{p-2}u & \mbox{in $\Omega$}, \\
\frac{\partial u}{\partial \n} = 0 & \mbox{on $\partial \Omega$},
\end{cases} \leqno{(P_\lambda)}  
$$
where 
\begin{itemize}
  \item $\Delta = \sum_{j=1}^N \frac{\partial^2}{\partial x_j^2}$ is the usual Laplacian in $\R^N$, 
  \item $\lambda \in \R$, 
  \item $1<q<2<p<\infty$, 
%     and $p< 2^* = \frac{2N}{N-2}$ if $N>2$, 
  \item $a, b \in C^\alpha(\overline{\Omega})$ with $\alpha \in (0,1)$,
  \item $\mathbf{n}$ is the unit outer normal to the boundary $\partial \Omega$.\end{itemize}

By a solution of $(P_\lambda)$ we mean a classical solution of $(P_\lambda)$. A solution $u$ of $(P_\lambda)$ is said to be {\it nontrivial and non-negative} if it satisfies $u\geq 0$ on $\overline{\Omega}$ and $u\not\equiv 0$, whereas it is said to be  {\it positive} if it satisfies $u>0$ on $\overline{\Omega}$.

If $\lambda>0$ and $a,b$ are positive on some non-empty open subset of $\Omega$ then $f_\lambda(x,s)=\lambda b(x) |s|^{q-2}s+a(x)|s|^{p-2}s$ belongs to the class of {\it concave-convex} type nonlinearities.
Since the work of Ambrosetti, Brezis and Cerami \cite{ABC}, this class of problems has been widely investigated, mostly for Dirichlet boundary conditions. In \cite{ABC} the authors proved the existence of $\Lambda>0$ such that the problem
\begin{align} \label{abc}
\begin{cases}
-\Delta u = \lambda |u|^{q-2}u + |u|^{p-2} u & \mbox{in $\Omega$}, \\
u=0 & \mbox{on $\partial \Omega$},
\end{cases}
\end{align}
has a minimal positive solution $u_\lambda$ for $0<\lambda<\Lambda$, at least one positive weak solution for $\lambda=\Lambda$, and no positive solution for $\lambda>\Lambda$ \cite[Theorem 2.1]{ABC}. Moreover, if $p \leq \frac{2N}{N-2}$ when $N \geq 3$ then \eqref{abc} has a second positive solution $v_\lambda>u_\lambda$ for $\lambda<\Lambda$ \cite[Theorem 2.3]{ABC}. It was also proved that $u_{\lambda}$ is the only positive solution of \eqref{abc} which converges to $0$ in $C(\overline{\Omega})$ as $\lambda \to 0^+$ \cite[Theorem 2.2]{ABC}. Most of the previous results were extended by De Figueiredo, Gossez, and Ubilla \cite{DGU2} to a larger class of concave-convex type problems, whose prototype is the analogue of $(P_\lambda)$ for Dirichlet boundary conditions, i.e.
\begin{align} \label{ccd}
\begin{cases}
-\Delta u = \lambda b(x)|u|^{q-2}u + a(x)|u|^{p-2} u & \mbox{in $\Omega$}, \\
u=0 & \mbox{on $\partial \Omega$}.
\end{cases}
\end{align}
Here $b \geq 0$, $b \not \equiv 0$ and $a$ may change sign.  For other works dealing with non-negative solutions of  indefinite concave-convex problems under Dirichlet boundary conditions we refer to \cite{DS2,dP,W06}. 

Several differences between $(P_\lambda)$ and \eqref{ccd} may be observed. 
The most evident one arises in the definite case $a,b \geq 0$, with $a,b \not \equiv 0$. It is known from \cite{DGU1,DGU2} that in this case \eqref{ccd} has a nontrivial non-negative solution for some $\lambda>0$. This result no longer holds for $(P_\lambda)$. As a matter of fact, if $u$ is a non-negative solution of $(P_\lambda)$ then a simple integration provides
$$\int_\Omega \left(\lambda  b(x)u^{q-1}+a(x)u^{p-1} \right)=0,$$
so that $u \equiv 0$ if $\lambda >0$.

The first purpose of this work is to obtain conditions on $a$ and $b$ which guarantee the existence of a nontrivial non-negative solution of $(P_\lambda)$ for some $\lambda>0$. In particular, we shall obtain two nontrivial non-negative solutions $u_{1,\lambda}, u_{2,\lambda}$ for $\lambda>0$ sufficiently small. At this point further differences between $(P_\lambda)$ and \eqref{ccd} may be pointed out. Unlike \cite[Theorem 2.2]{ABC}, we shall see that in some cases we have $u_{1,\lambda}, u_{2,\lambda} \to 0$ in $C(\overline{\Omega})$ as $\lambda \to 0^+$ (see Theorem \ref{t1}). Furthermore, in contrast with \cite{ABC,DGU2}, the second solution $u_{2,\lambda}$ may be obtained  without the condition $p \leq \frac{2N}{N-2}$ when $N \geq 3$ (see Remark \ref{rem:thm01}). 

To the best of our knowledge, very few works have been devoted to  concave-convex problems under Neumann boundary conditions. Tarfulea \cite{Ta98} considered $(P_\lambda)$ in the case $b \equiv 1$, proving that $\int_\Omega a<0$ is a necessary and sufficient condition for the existence of a positive solution. Making use of the sub-supersolutions method, the author proved the existence of $\Lambda > 0$ such that problem $(P_\lambda)$ has at least one positive solution for $\lambda < \Lambda$ which converges to $0$ in $L^\infty (\Omega)$ as $\lambda \to 0^+$, and no positive solution for $\lambda > \Lambda$. 

Garcia-Azorero, Peral, and Rossi \cite{G-APR04} dealt with  the problem
\begin{align} \label{prob:vexcave}
\begin{cases}
-\Delta u +u = |u|^{p-2}u & \mbox{in $\Omega$}, \\
\frac{\partial u}{\partial \n} = \lambda |u|^{q-2}u & \mbox{on $\partial \Omega$}.
\end{cases}
\end{align}
By means of a variational approach, they proved  that if $1<q<2$ and $p=\frac{2N}{N-2}$ when $N>2$ then there exists $\Lambda_1 > 0$ such that  \eqref{prob:vexcave} has at least two positive solutions for $\lambda < \Lambda_1$, at least one positive solution for $\lambda = \Lambda_1$, and no positive solution for $\lambda > \Lambda_1$. 

In \cite{Al}, Alama investigated the problem 
\begin{align} \label{alprob}
\begin{cases}
-\Delta u =\mu u +b(x) u^{q-1} + \gamma u^{p-1} & \mbox{in $\Omega$}, \\
\frac{\partial u}{\partial \n} = 0 & \mbox{on $\partial \Omega$},
\end{cases}
\end{align}
where $\mu \in \R$ and $\gamma>0$. Note that when $\mu=0$ this problem can be reduced to $(P_\lambda)$ by a suitable rescaling. A special difficulty in this problem is the possible existence of {\it dead core} solutions when $b$ changes sign. Using variational, bifurcation, and sub-supersolutions techniques, the author proved existence, non-existence and multiplicity results for non-negative solutions in accordance with $\gamma$ and $\mu$. Moreover, these solutions are shown to be positive in the set where $b>0$. However, the author did not discuss the structure of the non-negative solutions set when $\mu=0$. 

The second and main purpose of this article is to investigate the existence of a subcontinuum of non-negative solutions of $(P_\lambda)$. Some works have  been devoted to this issue in the context of concave-convex nonlinearities. In \cite{K1}, Korman proved that if $\Omega$ is a ball in $\R^N$ then there exists $\lambda_0>0$ such that the problem 
\begin{align} 
\begin{cases}
-\Delta u = \lambda \left(|u|^{p-2}u +|u|^{q-2}u\right)& \mbox{in $\Omega$}, \\
u = 0 & \mbox{on $\partial \Omega$},
\end{cases}
\end{align}
has exactly two positive solutions for $\lambda<\lambda_0$, one positive solution for $\lambda=\lambda_0$ and no positive solution for $\lambda>\lambda_0$. In addition, he proved that for $\lambda\leq \lambda_0$ the positive solutions lie on a single smooth solution curve and described the behavior of this curve with respect to $\lambda$. In \cite{K2} he extended these results to a problem with a non-autonomous concave-convex nonlinearity. Delgado and Su\'arez \cite{DS2} considered the problem
\begin{align} 
\begin{cases}
\mathcal{L} u = \lambda |u|^{q-2}u +a(x)|u|^{p-2}u & \mbox{in $\Omega$}, \\
u = 0 & \mbox{on $\partial \Omega$},
\end{cases}
\end{align}
where $\mathcal{L}$ is a second order uniformly elliptic operator not necessarily self-adjoint and $a$ changes sign. They proved the existence of a unbounded subcontinuum of non-negative solutions emanating supercritically from $(\lambda,u)=(0,0)$. To the best of our knowledge, no results on the existence of a subcontinuum of non-negative solutions for $(P_\lambda)$ are known when $b$ changes sign. 

Based on the asymptotic analysis of $u_{1,\lambda}, u_{2,\lambda}$ as $\lambda \to 0^+$, we shall prove in some cases the existence of a loop type subcontinuum (see Theorem \ref{thm:ep0:exist}) in the non-negative solutions set of $(P_\lambda)$. This kind of continuum has been investigated by L\'opez-G\'omez and Molina-Meyer in \cite{LGMM} and Brown in \cite{B07} for problems involving nonlinearities that are $C^1$ at $u=0$, which is not the case for $(P_\lambda)$. For that same reason, the standard global bifurcation theory proposed by Rabinowitz \cite{Ra71} (see also L\'opez-G\'omez \cite{LG01}) does not apply to $(P_\lambda)$ in a straightforward way. We shall overcome this difficulty using a regularization procedure that will be described later. Several works have made a direct use of the global bifurcation theory. We refer to Hess and Kato \cite{HK80} for a problem with a non self-adjoint operator, to Blat and Brown \cite{BB84} for a class of nonlinear elliptic systems, to L\'opez-G\'omez and Molina-Meyer \cite{LGMM} for a study of isolas or compact solution components, to Cantrell and Cosner \cite{CC89} for diffusive logistic equations from Mathematical Biology, and to Umezu \cite{Um10} and Cano-Casanova \cite{Ca14} for nonlinear boundary conditions.

Note that if $b \geq 0$ then, by the strong maximum principle and the boundary point lemma, nontrivial non-negative solutions of $(P_\lambda)$ are positive solutions. On the other hand, it is known that if $b^- \not \equiv 0$ then {\it dead core} solutions may arise \cite{BPT}, which makes delicate the study of the non-negative solutions set of $(P_\lambda)$, as shown in \cite{Al}. For instance, when $b$ changes sign the existence of a minimal non-negative solution for $\lambda>0$ small is still unknown. Furthermore, when $a \geq 0$ and $b$ changes sign, it is not known whether the condition $\int_\Omega b \geq 0$ provides non-existence of nontrivial non-negative solutions of $(P_\lambda)$ for $\lambda >0$. 

In our existence results we shall also be concerned with stability properties of positive solutions of $(P_\lambda)$. Let us recall that a positive solution $u$ of $(P_\lambda)$ is said to be
{\it asymptotically stable} (respect. {\it unstable}) if $\gamma_1(\lambda,u)>0$ (respect. $<0$), where
$\gamma_1(\lambda, u)$ is the first eigenvalue of the linearized problem at $u$, namely, 
\begin{align}  \label{eigenvp:ulam}
\begin{cases}
-\Delta \phi = (p-1)a(x)u^{p-2}\phi + \lambda
(q-1) b(x)u^{q-2}\phi + \gamma \phi &
\mbox{in $\Omega$}, \\
\frac{\partial \phi}{\partial \mathbf{n}} = 0 & \mbox{on
$\partial \Omega$}. 
\end{cases}
\end{align}
In addition, $u$ is said to be {\it weakly stable} if $\gamma_1(\lambda,u)\geq 0$.

Throughout this article, we consider the following sets:
$$
\Omega^a_{\pm} = \{ x \in \Omega : a(x)\gtrless 0 \}, \quad \Omega^a_0 = \{ x \in \Omega : a(x) = 0 \}, \quad \Omega^b_{\pm} = \{ x \in \Omega : b(x) \gtrless 0 \}.
$$
%\Omega^a_0 = \{ x \in \Omega : a = 0 \}$$ 
%and $$,  \quad \Omega^b_0 = \{ x \in \Omega : b = 0 \}.$$
Our main existence results for $\lambda>0$ shall be obtained under the condition 
\begin{equation}
\label{ab<0} \int_{\Omega} a <0 \quad \text{or} \quad \int_{\Omega} b <0.
\end{equation}
If either $ \int_{\Omega} a<0 \leq \int_{\Omega} b$ or $\int_{\Omega} a >0\geq \int_{\Omega} b $ then we set \begin{align}\label{u1:profile}
c^* = \left( \frac{-\int_{\Omega} b }{\int_\Omega a}\right)^{\frac{1}{p-q}}.
\end{align}
We are now in position to state out main results. 

First we follow a variational approach to show that $(P_\lambda)$ has two nontrivial non-negative solutions for $\lambda>0$ small if \eqref{ab<0} holds and $\Omega_+^a, \Omega_+^b \neq \emptyset$. This approach also provides us with the asymptotic profiles of these solutions as $\lambda \to 0^+$: 

\medskip
\begin{thm}
\label{t1}
 Assume \eqref{ab<0} and $p<\frac{2N}{N-2}$ if $N\geq 3$. Then there exists $\lambda_0>0$ such that:
\begin{enumerate}
\item If $\Omega^b_+ \neq \emptyset$ then $(P_\lambda)$ has a nontrivial non-negative solution $u_{1,\lambda}$ for $0<\lambda<\lambda_0$. Moreover there holds $u_{1,\lambda} \to 0$ in $C^2(\overline{\Omega})$ as $\lambda \to 0^+$.  More precisely:
\begin{enumerate}
\item If, in addition, $\int_\Omega b <0$ and $\lambda_n \to 0^+$ then, up to a subsequence, $\lambda_n^{-\frac{1}{2-q}}u_{1,\lambda_n} \to w_0$ in $C^2(\overline{\Omega})$ as $n \to \infty$, where $w_0$ is a nontrivial non-negative ground state solution of 
\begin{equation}
\label{plb}
\begin{cases}
-\Delta w = b(x)|w|^{q-2}w & \mbox{in $\Omega$}, \\
\frac{\partial w}{\partial \n} = 0 & \mbox{on $\partial \Omega$}.
\end{cases} 
\end{equation}

\item If, in addition, $\int_\Omega a<0\leq \int_\Omega b$ then $\lambda^{-\frac{1}{p-q}}u_{1,\lambda} \to c^*$ in $C^2(\overline{\Omega})$ as $\lambda \to 0^+$. In particular, if $\int_\Omega a<0< \int_\Omega b$ then $u_{1,\lambda}$ is an asymptotically stable positive solution of $(P_\lambda)$ for $\lambda>0$ sufficiently small. \\
\end{enumerate}
\item If $\Omega^a_+ \neq \emptyset$ then $(P_\lambda)$ has a nontrivial non-negative solution $u_{2,\lambda}$ for $0<\lambda<\lambda_0$. Moreover there holds:
\begin{enumerate}
\item If, in addition, $\int_\Omega a>0>\int_\Omega b$ then $u_{2,\lambda} \to 0$ and $\lambda^{-\frac{1}{p-q}}u_{2,\lambda} \to c^*$ in $C^2(\overline{\Omega})$ as $\lambda \to 0^+$. In particular, $u_{2,\lambda}$ is a unstable positive solution of $(P_\lambda)$ for $\lambda>0$ sufficiently small.
\item If, in addition, $\int_\Omega a = 0>\int_\Omega b$ then $u_{2,\lambda} \to 0$ in $C^2(\overline{\Omega})$ as $\lambda \to 0^+$.
\item If, in addition, $\int_\Omega a <0$ and $\lambda_n \to 0^+$ then, up to a subsequence, $u_{2,\lambda_n} \to u_{2,0}$ in $C^2(\overline{\Omega})$ as $n\to \infty$, where $u_{2,0}$ is a positive ground state solution of 
\begin{align} 
\label{pla}
\begin{cases}
-\Delta u = a(x)|u|^{p-2}u & \mbox{in $\Omega$}, \\
\frac{\partial u}{\partial \n} = 0 & \mbox{on $\partial \Omega$}. 
\end{cases} 
\end{align} 
In particular, $u_{2,\lambda}$ is a unstable positive solution of $(P_\lambda)$ for $\lambda>0$ sufficiently small.\\ 
\end{enumerate}
\end{enumerate}
\end{thm}

\begin{rem}
\label{rem:thm01}
\strut
{\rm
\begin{enumerate}
\item  Except for (1)(a), 1(b) with $\int_\Omega b = 0$, 2(b) and (2)(c), Theorem \ref{t1} remains true without the condition $p<\frac{2N}{N-2}$ if $N\geq 3$. In the case $\Omega^b_+ \neq \emptyset$ and $\int_\Omega b<0$ a solution having similar features as $u_{1,\lambda}$ may be obtained by the sub-supersolutions method. Note that in contrast with the case of Dirichlet boundary conditions, obtaining a strict supersolution for $(P_\lambda)$ is not an easy task. We shall use the asymptotic profile of $u_{1,\lambda}$ provided by Theorem \ref{t1} (1)(a) to obtain such a supersolution, cf. Proposition \ref{psubsup}.
In the case $\int_\Omega a<0< \int_\Omega b$ we shall use
the Lyapunov-Schmidt reduction method to obtain a positive solution $u_\lambda$ such that $\lambda^{-\frac{1}{p-q}}u_{\lambda} \to c^*$ in $C^2(\overline{\Omega})$ as $\lambda \to 0^+$. The same procedure can be applied in the case $\int_\Omega a>0> \int_\Omega b$, cf. Remark \ref{rbif}. \\
\item When $\Omega^b_+$ is a non-empty subdomain of $\Omega$ and $\int_\Omega b<0$, we deduce from Theorem \ref{t1} (1)(a) that for any subset $D$ satisfying $\overline{D} \subset \Omega^b_+$ there exist $\overline{\lambda}, \overline{c} > 0$ such that $\inf_D u_{1, \lambda} \geq \overline{c}$ for $\lambda \in (0, \overline{\lambda})$. This result comes from the fact that $w_0 > 0$ in $\Omega^b_+$. \\ 
\end{enumerate}
}
\end{rem}

From Theorem \ref{t1} we infer in particular some positivity and ordering properties for $u_{1,\lambda}$ and $u_{2,\lambda}$ (cf. Remark \ref{rasyu2} (1)):

\begin{cor}
Assume $p<\frac{2N}{N-2}$ if $N\geq 3$. Let $u_{1,\lambda}$ and $u_{2,\lambda}$ be provided by Theorem \ref{t1}.
\begin{enumerate}
\item If $\Omega_+^a \neq \emptyset$ and $\int_\Omega a<0< \int_\Omega b$ then there exists $\lambda^*>0$ such that $u_{2,\lambda}>u_{1,\lambda}>0$ on $\overline{\Omega}$ for $0<\lambda<\lambda^*$.
\item If $\Omega_+^b \neq \emptyset$ and $\int_\Omega a>0>\int_\Omega b$ then there exists $\lambda^*>0$ such that $u_{2,\lambda}>u_{1,\lambda} \geq 0$  on $\overline{\Omega}$ for $0<\lambda<\lambda^*$.
\end{enumerate}

\end{cor}

As for non-existence of nontrivial non-negative solutions of $(P_\lambda)$, we have the following result:

\begin{thm}
\label{t2}
\strut
\begin{enumerate}
\item Let $\lambda > 0$. Then the following two assertions hold: 

\begin{enumerate}

  \item Assume  $b \geq 0$ and $\int_\Omega a \geq 0$. Then $(P_\lambda)$ has no nontrivial non-negative solution.

  \item Assume that $b$ changes sign, $\Omega_+^b$ is a subdomain of $\Omega$, and $\Omega^b_- = \Omega \setminus \overline{\Omega^b_+}$. If $a \geq 0$ and $\int_\Omega b \geq 0$ then $(P_\lambda)$ has no non-negative solution taking positive values somewhere in $\Omega_+^b$.

\end{enumerate}

\item Assume $\Omega_+^a \cap \Omega_+^b \neq \emptyset$. Then there exists $\overline{\lambda}>0$ such that $(P_\lambda)$ has no nontrivial non-negative solution for $\lambda>\overline{\lambda}$.

\end{enumerate}
\end{thm}

\begin{rem} \label{rem:t2}
{\rm 
Theorem \ref{t2} holds true for $\lambda < 0$ with $b$ replaced by $-b$. Indeed, it suffices to look at the equation in $(P_\lambda)$ as $-\Delta u = (-\lambda)(-b(x))|u|^{q-2}u+ a(x)|u|^{p-2}u$. 

}\end{rem}

We consider then structure of the non-negative solutions set of $(P_\lambda)$. Under the condition 
\begin{equation}
\label{abl}
\Omega^a_+ \neq \emptyset, \quad \Omega^b_+ \neq \emptyset, \quad \int_\Omega b \leq 0 \quad \text{ and } \quad \int_\Omega a<0,
\end{equation} 
Theorem \ref{t1} asserts that $u_{1,\lambda} \to 0$ in $C^2(\overline{\Omega})$ as $\lambda \to 0^+$, and if $\lambda_n \to 0^+$ then, up to a subsequence, $u_{2, \lambda_n} \to u_{2,0}$ in $C^2(\overline{\Omega})$, where $u_{2,0}$ is a positive solution of \eqref{pla}. In addition,  this result does not depend on the sign of $\int_\Omega b$. As a consequence we may also infer the existence of two nontrivial non-negative solutions $v_{1,\lambda}$ and $v_{2,\lambda}$ for $\lambda<0$ sufficiently small. These solutions satisfy $v_{1,\lambda} \to 0$ in $C^2(\overline{\Omega})$ as $\lambda \to 0^+$, and if $\lambda_n \to 0^+$ then, up to a subsequence, 
$v_{2, \lambda_n} \to v_{2,0}$ in $C^2(\overline{\Omega})$, where $v_{2,0}$ is a positive solution of \eqref{pla}. One may then ask if these solutions lie on a loop type subcontinuum of non-negative solutions of $(P_\lambda)$. 

%We consider two subcases of \eqref{abl} in accordance with $\int_\Omega b$:

We shall investigate this question  by considering a regularized version of $(P_\lambda)$, namely,
$$
\begin{cases}
-\Delta u = a(x)|u|^{p-2}u + \lambda (b(x) - \epsilon) |u+\epsilon|^{q-2}u & \mbox{in} \ \Omega, \\
\frac{\partial u}{\partial \mathbf{n}} = 0 & \mbox{on} \ \partial \Omega, 
\end{cases} \leqno{(P_{\lambda,\epsilon})}  
$$
where $\lambda \in \R$ and $\epsilon > 0$. We may then look at $(P_\lambda)$ as the limit problem of $(P_{\lambda,\epsilon})$ when $\epsilon \to 0^+$. This procedure has been already used in \cite{RQU3}, where a regularized version of a nonlinear boundary condition is studied. Note that the mapping $t \mapsto |t+\epsilon|^{q-2}t$ is analytic at $t=0$ and any nontrivial non-negative solution of $(P_{\lambda,\epsilon})$ is positive on $\overline{\Omega}$. 
The unilateral global bifurcation theorem by Rabinowitz \cite[Theorem 1.27]{Ra71} (see also L\'opez-G\'omez \cite[Theorem 6.4.3]{LG01}) may then be applied to $(P_{\lambda,\epsilon})$. To this end we consider its linearized problem at $u=0$: 
\begin{align} \label{eprob151011}
\begin{cases}
-\Delta \varphi = \lambda (b-\epsilon) \epsilon^{q-2} \varphi & \mbox{in} \ \Omega, \\
\frac{\partial \varphi}{\partial \mathbf{n}} = 0 & \mbox{on} \ 
\partial \Omega. 
\end{cases}
\end{align}
Under the condition 
\begin{align} \label{a:b} 
\Omega^{b - \epsilon}_+ \neq \emptyset \quad\mbox{and} \quad \int_\Omega b \leq 0, 
\end{align}
this problem has exactly two principal eigenvalues $ \lambda=0$ and $\lambda=\lambda_\epsilon>0$, which are both simple.
% 
%Since $\lambda_\epsilon = \lambda_1 \epsilon^{2-q}$, we note that
%
%\begin{align} \label{151011lamepto0}
%\lambda_\epsilon \searrow 0 \quad \mbox{as $\epsilon \to 0^+$}. 
%\end{align}
%
We use the unilateral global bifurcation theory to obtain two subcontinua $\mathcal{C}_0=\mathcal{C}_0(\epsilon)$, $\mathcal{C}_1=\mathcal{C}_1 (\epsilon)$ of positive solutions of $(P_{\lambda,\epsilon})$ bifurcating from $(0,0)$ and $(\lambda_\epsilon,0)$, respectively. Moreover, we analyse the local nature of these subcontinua near the bifurcation points (Theorem \ref{thm:rprob:r}).
We turn then to the study of the global nature of $\mathcal{C}_0 (\epsilon), \mathcal{C}_1 (\epsilon)$ and their limiting nature  as $\epsilon \to 0^+$. First we show that positive solutions $(\lambda,u)$ of $(P_{\lambda, \epsilon})$ are {\it a priori} bounded in $\R \times C(\overline{\Omega})$ if the following conditions are 
assumed, where we assume $(H_2)$ following Amann and L\'opez-G\'omez \cite{ALG}: \\

\begin{itemize}

\item[$(H_0)$] There exist balls $B_1$, $B_2$ such that $\overline{B_1}, \overline{B_2} \subset \Omega$, and 
\begin{align*} 
\begin{cases}
a\geq 0, \ \ a\not\equiv 0 \ \ \mbox{and $b>0$ on $\overline{B_1}$}, & \bigskip \\
a\geq 0, \ \ a\not\equiv 0 \ \ \mbox{and $b<0$ on $\overline{B_2}$}, & 
\end{cases}
\end{align*}\\

\item [$(H_1)$] $\Omega^a_{\pm}$ are subdomains of $\Omega$ with smooth boundary and satisfy $\overline{\Omega^a_+} \subset \Omega$, $\overline{\Omega^a_+} \cup \Omega^a_- = \Omega$. \\
%
%and $\Omega^a_- \subset \Omega^b_0$.  

\item [$(H_2)$] Under $(H_1)$ there exist a function $\alpha^+$ which is continuous, positive, and bounded away from zero in a tubular neighborhood of $\partial \Omega^a_+$ in $\Omega^a_+$ and $\gamma>0$ such that $$a^+(x) = \alpha^+ (x) {\rm dist} (x, \partial \Omega^a_+)^{\gamma},$$
where ${\rm dist}\, (x, A)$ denotes the distance function to a set $A$. Moreover, we assume that $$
2 < p < \min \left\{ \frac{2N}{N-2}, \frac{2N+\gamma}{N-1} \right\} \quad 
\mbox{if} \ \ N>2. 
$$\\
\end{itemize}

%	Meanwhile, in the case $\int_\Omega b = 0$, if in addition to conditions $(H_2)$ we assume that 
%
%	\begin{itemize}
%
%
%	\item[$(H_0')$] There exist a ball $B$ such that $\overline{B} \subset \Omega$, and 
%	\begin{align} 
%	a\geq 0, \ \ a\not\equiv 0 \ \ \mbox{and $b>0$ in $\overline{B}$},  
%	%& \bigskip \\
%	%a\geq 0, \ \ a\not\equiv 0 \ \ \mbox{and $b<0$ in $\overline{B_2}$}, & %
%	\end{align}
%
%
%	\item [$(H_1')$] $\Omega^a_{\pm}$ are subdomains of $\Omega$ with smooth boundary and satisfy $\overline{\Omega^a_+} \subset \Omega$, $\overline{\Omega^a_+} \cup \Omega^a_- = \Omega$, and $\Omega^a_- \subset \Omega^{b-b_0}_+$ for some $b_0 > 0$,    
%	\end{itemize}
%	then we show that positive solutions $(\lambda,u)$ of $(P_{\lambda, \epsilon})$ are {\it a priori} bounded in $[0, \infty) \times C(\overline{\Omega})$. 
%

Based on these {\it a priori} bounds and the global properties of $\mathcal{C}_0$ and $\mathcal{C}_1$, we infer that these subcontinua are both bounded, and consequently must coincide, i.e. $\mathcal{C}_0 = \mathcal{C}_1=\mathcal{C}_*$ (Theorem \ref{t:ep:global}). Thus $(P_{\lambda, \epsilon})$ has a bounded subcontinuum of positive solutions going from $(0,0)$ to  $(\lambda_\epsilon, 0)$, see Figure \ref{fig15_111402}.  We consider then the limiting profiles of $\mathcal{C}_0$ and $\mathcal{C}_1$ as $\epsilon \to 0^+$ by means of Whyburn's topological method  \cite{Wh64}. Here {\it a priori} bounds from below for positive solutions of $(P_{\lambda, \epsilon})$ with $\lambda = 0$ (Lemma \ref{lem:lameq0}) and the fact that bifurcation from zero does not occur for $(P_\lambda)$ at any $\lambda\not= 0$ (Proposition \ref{prop:nobiforip}) play an important role. The latter fact is verified under the condition
\begin{itemize}
\item [$(H_3)$] $\Omega^b_+$ and $\Omega^b_-$ are both subdomains of $\Omega$.\\
\end{itemize}

Combining the previous results, we establish:

\begin{thm} \label{thm:ep0:exist}
Assume \eqref{abl}. If $(H_0)$, $(H_1)$, $(H_2)$ and $(H_3)$ are satisfied then $(P_\lambda)$ has a {\rm loop type} subcontinuum (non-empty, closed and connected component) $\mathcal{C}$ of nontrivial non-negative solutions bifurcating at $(0,0)$, which joins $(0,0)$ to itself. Moreover:
\begin{enumerate}
\item $\mathcal{C}$ is non-trivial, i.e. $\mathcal{C} \neq \{(0,0)\}$.
\item  The only trivial solution contained in $\mathcal{C}$ is $(\lambda, u)=(0,0)$, i.e. $\mathcal{C}$ does not contain any point $(\lambda,0)$ with $\lambda \neq 0$. 
\item There exists $\delta > 0$ such that $\mathcal{C}$ does not contain any positive solution $u $ of \eqref{pla} satisfying $\Vert u \Vert_{C(\overline{\Omega})} \leq \delta$. 
\end{enumerate} 
\end{thm}

Figure \ref{fig15_1114a} illustrates the subcontinuum provided by Theorems  \ref{thm:ep0:exist}.

\begin{rem}{\rm 
An example of $(a, b)$ satisfying conditions $(H_0)$, $(H_1)$ and $(H_3)$ can be constructed 
as in Figure \ref{fig15_1210}. 
}\end{rem}

Finally, let us mention that our regularization procedure described above can also be used to obtain subcontinua (non-necessarily of loop type) for a larger class of concave-convex type problems. We shall treat this issue in a forthcoming article.

%%%%%%%%%%%%%%%%%%%%%%
%	\begin{figure}[H] % 
	\begin{figure}[!htb]
		%WinTpicVersion4.28b
{\unitlength 0.1in
\begin{picture}( 28.5000, 17.9000)(  8.3000,-24.1000)
% ELLIPSE 1 0 3 0 Black White
% 4 2255 1515 3680 620 3680 620 3680 620
% 
\special{pn 13}%
\special{ar 2256 1516 1426 896  0.0000000  6.2831853}%
% ELLIPSE 2 0 3 0 Black White
% 4 1823 1563 2646 2161 2646 2161 2652 2161
% 
\special{pn 8}%
\special{ar 1824 1564 824 598  0.7817662  0.7853982}%
% ELLIPSE 1 0 3 0 Black White
% 4 2604 1488 3434 2040 3434 2040 3434 2040
% 
\special{pn 13}%
\special{ar 2604 1488 830 552  0.0000000  6.2831853}%
% STR 2 0 3 0 Black White
% 4 2930 1121 2930 1189 5 0 0 0
% $\Omega_a^+$
\put(29.3000,-11.8900){\makebox(0,0){$\Omega_a^+$}}%
% STR 2 0 3 0 Black White
% 4 2627 706 2627 774 5 0 0 0
% $\Omega_a^-$
\put(26.2700,-7.7400){\makebox(0,0){$\Omega_a^-$}}%
% STR 2 0 3 0 Black White
% 4 1657 1015 1657 1083 5 0 0 0
% $\Omega_b^\mp$
\put(16.5700,-10.8300){\makebox(0,0){$\Omega_b^\mp$}}%
% STR 2 0 3 0 Black White
% 4 2161 2188 2161 2256 5 0 0 0
% $\Omega_b^\pm$
\put(21.6100,-22.5600){\makebox(0,0){$\Omega_b^\pm$}}%
% ELLIPSE 2 2 3 0 Black White
% 4 1948 1469 2704 2094 2704 2094 2704 2094
% 
\special{pn 8}%
\special{pn 8}%
\special{pa 2704 1470}%
\special{pa 2704 1478}%
\special{fp}%
\special{pa 2702 1515}%
\special{pa 2702 1523}%
\special{fp}%
\special{pa 2697 1558}%
\special{pa 2696 1566}%
\special{fp}%
\special{pa 2688 1600}%
\special{pa 2686 1608}%
\special{fp}%
\special{pa 2675 1641}%
\special{pa 2673 1649}%
\special{fp}%
\special{pa 2659 1682}%
\special{pa 2656 1689}%
\special{fp}%
\special{pa 2641 1721}%
\special{pa 2636 1728}%
\special{fp}%
\special{pa 2619 1759}%
\special{pa 2614 1766}%
\special{fp}%
\special{pa 2594 1796}%
\special{pa 2590 1802}%
\special{fp}%
\special{pa 2566 1830}%
\special{pa 2561 1836}%
\special{fp}%
\special{pa 2537 1863}%
\special{pa 2530 1868}%
\special{fp}%
\special{pa 2504 1892}%
\special{pa 2499 1899}%
\special{fp}%
\special{pa 2471 1921}%
\special{pa 2465 1926}%
\special{fp}%
\special{pa 2435 1947}%
\special{pa 2429 1952}%
\special{fp}%
\special{pa 2398 1972}%
\special{pa 2391 1976}%
\special{fp}%
\special{pa 2360 1994}%
\special{pa 2353 1998}%
\special{fp}%
\special{pa 2321 2014}%
\special{pa 2314 2017}%
\special{fp}%
\special{pa 2280 2031}%
\special{pa 2273 2034}%
\special{fp}%
\special{pa 2240 2046}%
\special{pa 2232 2048}%
\special{fp}%
\special{pa 2198 2059}%
\special{pa 2191 2062}%
\special{fp}%
\special{pa 2156 2070}%
\special{pa 2149 2072}%
\special{fp}%
\special{pa 2113 2080}%
\special{pa 2105 2080}%
\special{fp}%
\special{pa 2070 2086}%
\special{pa 2062 2088}%
\special{fp}%
\special{pa 2026 2092}%
\special{pa 2018 2092}%
\special{fp}%
\special{pa 1981 2094}%
\special{pa 1973 2094}%
\special{fp}%
\special{pa 1936 2094}%
\special{pa 1928 2094}%
\special{fp}%
\special{pa 1892 2092}%
\special{pa 1884 2092}%
\special{fp}%
\special{pa 1848 2088}%
\special{pa 1840 2088}%
\special{fp}%
\special{pa 1804 2083}%
\special{pa 1796 2082}%
\special{fp}%
\special{pa 1761 2075}%
\special{pa 1753 2074}%
\special{fp}%
\special{pa 1719 2065}%
\special{pa 1711 2064}%
\special{fp}%
\special{pa 1677 2053}%
\special{pa 1670 2051}%
\special{fp}%
\special{pa 1636 2038}%
\special{pa 1628 2036}%
\special{fp}%
\special{pa 1596 2023}%
\special{pa 1589 2019}%
\special{fp}%
\special{pa 1556 2003}%
\special{pa 1549 2000}%
\special{fp}%
\special{pa 1517 1983}%
\special{pa 1510 1979}%
\special{fp}%
\special{pa 1479 1959}%
\special{pa 1473 1955}%
\special{fp}%
\special{pa 1442 1934}%
\special{pa 1436 1930}%
\special{fp}%
\special{pa 1408 1906}%
\special{pa 1402 1901}%
\special{fp}%
\special{pa 1375 1877}%
\special{pa 1369 1871}%
\special{fp}%
\special{pa 1344 1845}%
\special{pa 1339 1839}%
\special{fp}%
\special{pa 1315 1811}%
\special{pa 1311 1805}%
\special{fp}%
\special{pa 1289 1775}%
\special{pa 1285 1769}%
\special{fp}%
\special{pa 1265 1738}%
\special{pa 1262 1731}%
\special{fp}%
\special{pa 1245 1700}%
\special{pa 1242 1692}%
\special{fp}%
\special{pa 1229 1660}%
\special{pa 1226 1653}%
\special{fp}%
\special{pa 1215 1620}%
\special{pa 1212 1613}%
\special{fp}%
\special{pa 1204 1578}%
\special{pa 1202 1571}%
\special{fp}%
\special{pa 1196 1536}%
\special{pa 1196 1528}%
\special{fp}%
\special{pa 1192 1492}%
\special{pa 1192 1484}%
\special{fp}%
\special{pa 1192 1447}%
\special{pa 1194 1439}%
\special{fp}%
\special{pa 1196 1403}%
\special{pa 1198 1396}%
\special{fp}%
\special{pa 1204 1360}%
\special{pa 1206 1353}%
\special{fp}%
\special{pa 1215 1319}%
\special{pa 1217 1312}%
\special{fp}%
\special{pa 1229 1277}%
\special{pa 1232 1270}%
\special{fp}%
\special{pa 1246 1238}%
\special{pa 1250 1231}%
\special{fp}%
\special{pa 1266 1199}%
\special{pa 1271 1193}%
\special{fp}%
\special{pa 1290 1162}%
\special{pa 1294 1156}%
\special{fp}%
\special{pa 1316 1126}%
\special{pa 1321 1120}%
\special{fp}%
\special{pa 1345 1093}%
\special{pa 1351 1087}%
\special{fp}%
\special{pa 1375 1061}%
\special{pa 1382 1056}%
\special{fp}%
\special{pa 1409 1032}%
\special{pa 1415 1027}%
\special{fp}%
\special{pa 1444 1005}%
\special{pa 1450 1000}%
\special{fp}%
\special{pa 1480 978}%
\special{pa 1486 974}%
\special{fp}%
\special{pa 1518 956}%
\special{pa 1525 952}%
\special{fp}%
\special{pa 1557 935}%
\special{pa 1564 932}%
\special{fp}%
\special{pa 1597 917}%
\special{pa 1603 913}%
\special{fp}%
\special{pa 1637 900}%
\special{pa 1644 898}%
\special{fp}%
\special{pa 1678 886}%
\special{pa 1685 884}%
\special{fp}%
\special{pa 1720 874}%
\special{pa 1727 872}%
\special{fp}%
\special{pa 1762 864}%
\special{pa 1770 862}%
\special{fp}%
\special{pa 1805 856}%
\special{pa 1813 855}%
\special{fp}%
\special{pa 1849 850}%
\special{pa 1857 850}%
\special{fp}%
\special{pa 1893 846}%
\special{pa 1901 846}%
\special{fp}%
\special{pa 1937 844}%
\special{pa 1945 844}%
\special{fp}%
\special{pa 1982 846}%
\special{pa 1990 846}%
\special{fp}%
\special{pa 2027 848}%
\special{pa 2035 848}%
\special{fp}%
\special{pa 2071 852}%
\special{pa 2078 854}%
\special{fp}%
\special{pa 2114 860}%
\special{pa 2121 862}%
\special{fp}%
\special{pa 2157 868}%
\special{pa 2164 871}%
\special{fp}%
\special{pa 2199 880}%
\special{pa 2206 882}%
\special{fp}%
\special{pa 2240 892}%
\special{pa 2247 896}%
\special{fp}%
\special{pa 2281 908}%
\special{pa 2288 911}%
\special{fp}%
\special{pa 2321 926}%
\special{pa 2328 929}%
\special{fp}%
\special{pa 2360 946}%
\special{pa 2367 949}%
\special{fp}%
\special{pa 2399 966}%
\special{pa 2405 971}%
\special{fp}%
\special{pa 2436 991}%
\special{pa 2442 996}%
\special{fp}%
\special{pa 2471 1017}%
\special{pa 2477 1023}%
\special{fp}%
\special{pa 2505 1046}%
\special{pa 2511 1052}%
\special{fp}%
\special{pa 2537 1077}%
\special{pa 2542 1083}%
\special{fp}%
\special{pa 2567 1110}%
\special{pa 2572 1116}%
\special{fp}%
\special{pa 2594 1144}%
\special{pa 2599 1151}%
\special{fp}%
\special{pa 2619 1181}%
\special{pa 2624 1188}%
\special{fp}%
\special{pa 2641 1219}%
\special{pa 2644 1226}%
\special{fp}%
\special{pa 2660 1258}%
\special{pa 2662 1265}%
\special{fp}%
\special{pa 2676 1298}%
\special{pa 2678 1305}%
\special{fp}%
\special{pa 2688 1339}%
\special{pa 2690 1346}%
\special{fp}%
\special{pa 2697 1382}%
\special{pa 2698 1389}%
\special{fp}%
\special{pa 2702 1425}%
\special{pa 2704 1433}%
\special{fp}%
\end{picture}}%
	  \caption{An example of $(a, b)$ satisfying $(H_0)$, $(H_1)$ and $(H_3)$.}   			    
  	\label{fig15_1210}
	    \end{figure}
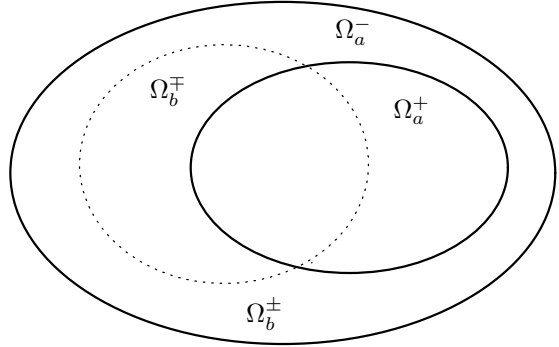 
%%%%%%%%%%%%%%%%%%%%%	   

%%%%%%%%%%%%%%%%%%%%%%%%%%%%%%%%%%%%%%%
%	\begin{figure}[H]
	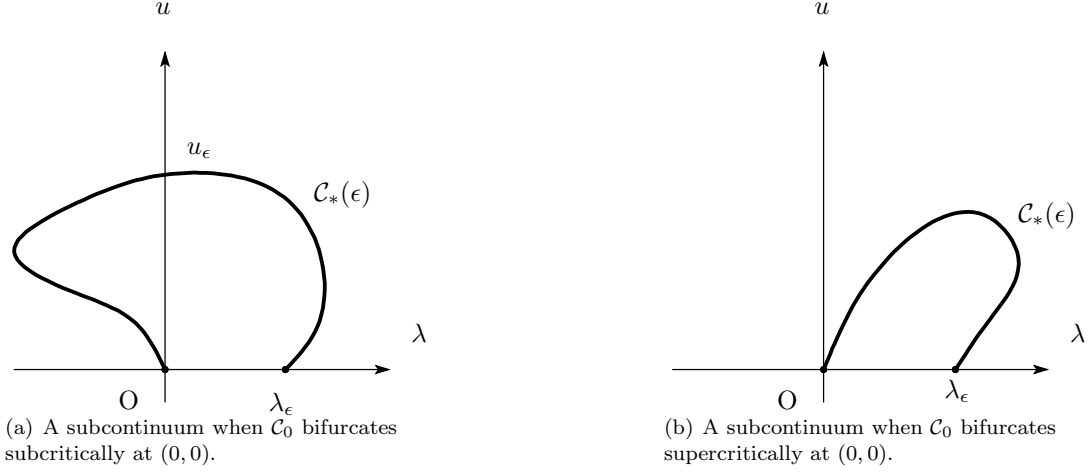
\begin{figure}[!htb] 
      \begin{center}
    % (a)
      \subfigure[A subcontinuum when $\mathcal{C}_0$ bifurcates subcritically  at $(0,0)$.]{
        %WinTpicVersion4.28b
{\unitlength 0.1in
\begin{picture}( 19.6000, 21.4000)( 21.3000,-24.9000)
% VECTOR 2 0 3 0 Black White
% 2 2130 2320 4090 2320
% 
\special{pn 8}%
\special{pa 2130 2320}%
\special{pa 4090 2320}%
\special{fp}%
\special{sh 1}%
\special{pa 4090 2320}%
\special{pa 4024 2300}%
\special{pa 4038 2320}%
\special{pa 4024 2340}%
\special{pa 4090 2320}%
\special{fp}%
% VECTOR 2 0 3 0 Black White
% 2 2920 2490 2920 660
% 
\special{pn 8}%
\special{pa 2920 2490}%
\special{pa 2920 660}%
\special{fp}%
\special{sh 1}%
\special{pa 2920 660}%
\special{pa 2900 728}%
\special{pa 2920 714}%
\special{pa 2940 728}%
\special{pa 2920 660}%
\special{fp}%
% STR 2 0 3 0 Black White
% 4 2730 2380 2730 2480 5 0 0 0
% O
\put(27.3000,-24.8000){\makebox(0,0){O}}%
% STR 2 0 3 0 Black White
% 4 4250 2030 4250 2130 5 0 0 0
% $\lambda$
\put(42.5000,-21.3000){\makebox(0,0){$\lambda$}}%
% STR 2 0 3 0 Black White
% 4 2910 330 2910 430 5 0 0 0
% $u$
\put(29.1000,-4.3000){\makebox(0,0){$u$}}%
% STR 2 0 3 0 Black White
% 4 3100 1070 3100 1170 5 0 0 0
% $u_\epsilon$
\put(31.0000,-11.7000){\makebox(0,0){$u_\epsilon$}}%
% SPLINE 0 0 3 0 Black White
% 10 2920 2320 2780 2090 2430 1910 2130 1700 2430 1460 3020 1290 3570 1440 3710 2110 3550 2320 3550 2320
% 
\special{pn 20}%
\special{pa 2920 2320}%
\special{pa 2876 2230}%
\special{pa 2860 2202}%
\special{pa 2826 2148}%
\special{pa 2808 2122}%
\special{pa 2768 2078}%
\special{pa 2746 2060}%
\special{pa 2722 2042}%
\special{pa 2698 2026}%
\special{pa 2672 2010}%
\special{pa 2644 1998}%
\special{pa 2616 1984}%
\special{pa 2586 1972}%
\special{pa 2556 1958}%
\special{pa 2524 1946}%
\special{pa 2490 1934}%
\special{pa 2456 1920}%
\special{pa 2384 1892}%
\special{pa 2348 1876}%
\special{pa 2312 1858}%
\special{pa 2278 1842}%
\special{pa 2244 1824}%
\special{pa 2214 1804}%
\special{pa 2188 1786}%
\special{pa 2166 1766}%
\special{pa 2148 1746}%
\special{pa 2136 1724}%
\special{pa 2130 1704}%
\special{pa 2132 1682}%
\special{pa 2140 1662}%
\special{pa 2152 1640}%
\special{pa 2194 1598}%
\special{pa 2222 1578}%
\special{pa 2252 1558}%
\special{pa 2284 1538}%
\special{pa 2318 1518}%
\special{pa 2354 1498}%
\special{pa 2390 1480}%
\special{pa 2460 1446}%
\special{pa 2494 1432}%
\special{pa 2526 1416}%
\special{pa 2558 1402}%
\special{pa 2590 1390}%
\special{pa 2652 1366}%
\special{pa 2742 1336}%
\special{pa 2770 1328}%
\special{pa 2800 1322}%
\special{pa 2830 1314}%
\special{pa 2858 1308}%
\special{pa 2948 1296}%
\special{pa 2978 1294}%
\special{pa 3074 1288}%
\special{pa 3106 1290}%
\special{pa 3138 1290}%
\special{pa 3172 1292}%
\special{pa 3204 1296}%
\special{pa 3238 1300}%
\special{pa 3270 1304}%
\special{pa 3304 1312}%
\special{pa 3336 1318}%
\special{pa 3368 1328}%
\special{pa 3398 1338}%
\special{pa 3428 1350}%
\special{pa 3458 1364}%
\special{pa 3488 1380}%
\special{pa 3514 1398}%
\special{pa 3542 1416}%
\special{pa 3566 1436}%
\special{pa 3590 1460}%
\special{pa 3612 1484}%
\special{pa 3634 1510}%
\special{pa 3652 1538}%
\special{pa 3670 1568}%
\special{pa 3702 1628}%
\special{pa 3714 1660}%
\special{pa 3726 1694}%
\special{pa 3742 1762}%
\special{pa 3748 1796}%
\special{pa 3752 1830}%
\special{pa 3756 1866}%
\special{pa 3756 1900}%
\special{pa 3754 1934}%
\special{pa 3750 1968}%
\special{pa 3746 2000}%
\special{pa 3730 2064}%
\special{pa 3718 2094}%
\special{pa 3704 2124}%
\special{pa 3688 2152}%
\special{pa 3672 2178}%
\special{pa 3652 2204}%
\special{pa 3634 2228}%
\special{pa 3612 2254}%
\special{pa 3568 2302}%
\special{pa 3550 2320}%
\special{fp}%
% STR 2 0 3 0 Black White
% 4 3530 2400 3530 2500 5 0 0 0
% $\lambda_\epsilon$
\put(35.3000,-25.0000){\makebox(0,0){$\lambda_\epsilon$}}%
% STR 2 0 3 0 Black White
% 4 3850 1300 3850 1400 5 0 0 0
% $\mathcal{C}_*(\epsilon)$
\put(38.5000,-14.0000){\makebox(0,0){$\mathcal{C}_*(\epsilon)$}}%
% DOT 0 0 3 0 Black White
% 2 2920 2320 2920 2320
% 
\special{pn 4}%
\special{sh 1}%
\special{ar 2920 2320 16 16 0  6.28318530717959E+0000}%
\special{sh 1}%
\special{ar 2920 2320 16 16 0  6.28318530717959E+0000}%
% DOT 0 0 3 0 Black White
% 2 3550 2320 3550 2320
% 
\special{pn 4}%
\special{sh 1}%
\special{ar 3550 2320 16 16 0  6.28318530717959E+0000}%
\special{sh 1}%
\special{ar 3550 2320 16 16 0  6.28318530717959E+0000}%
\end{picture}}%
          \label{fig15_1114d}
      }
      \hfill
      % (b)
%      \subfigure[A subcontinua when the supercritical bifurcation occurs for $\mathcal{C}_0$ at $(0,0)$. ]{
    %
    %
  %        \input fig15_1114cc   
	%
    %    \label{fig15_1114cc} 
     % }
%
      %\hfill
      % (b)
      \subfigure[A subcontinuum when $\mathcal{C}_0$ bifurcates supercritically  at $(0,0)$.]{
          %WinTpicVersion4.28b
{\unitlength 0.1in
\begin{picture}( 19.6000, 21.4000)( 21.3000,-24.9000)
% VECTOR 2 0 3 0 Black White
% 2 2130 2320 4090 2320
% 
\special{pn 8}%
\special{pa 2130 2320}%
\special{pa 4090 2320}%
\special{fp}%
\special{sh 1}%
\special{pa 4090 2320}%
\special{pa 4024 2300}%
\special{pa 4038 2320}%
\special{pa 4024 2340}%
\special{pa 4090 2320}%
\special{fp}%
% VECTOR 2 0 3 0 Black White
% 2 2920 2490 2920 660
% 
\special{pn 8}%
\special{pa 2920 2490}%
\special{pa 2920 660}%
\special{fp}%
\special{sh 1}%
\special{pa 2920 660}%
\special{pa 2900 728}%
\special{pa 2920 714}%
\special{pa 2940 728}%
\special{pa 2920 660}%
\special{fp}%
% STR 2 0 3 0 Black White
% 4 2730 2380 2730 2480 5 0 0 0
% O
\put(27.3000,-24.8000){\makebox(0,0){O}}%
% STR 2 0 3 0 Black White
% 4 4250 2030 4250 2130 5 0 0 0
% $\lambda$
\put(42.5000,-21.3000){\makebox(0,0){$\lambda$}}%
% STR 2 0 3 0 Black White
% 4 2910 330 2910 430 5 0 0 0
% $u$
\put(29.1000,-4.3000){\makebox(0,0){$u$}}%
% STR 2 0 3 0 Black White
% 4 4090 1410 4090 1510 5 0 0 0
% $\mathcal{C}_*(\epsilon)$
\put(40.9000,-15.1000){\makebox(0,0){$\mathcal{C}_*(\epsilon)$}}%
% STR 2 0 3 0 Black White
% 4 3620 2330 3620 2430 5 0 0 0
% $\lambda_\epsilon$
\put(36.2000,-24.3000){\makebox(0,0){$\lambda_\epsilon$}}%
% DOT 0 0 3 0 Black White
% 2 3610 2320 3610 2320
% 
\special{pn 4}%
\special{sh 1}%
\special{ar 3610 2320 16 16 0  6.28318530717959E+0000}%
\special{sh 1}%
\special{ar 3610 2320 16 16 0  6.28318530717959E+0000}%
% DOT 2 0 3 0 Black White
% 2 2920 2320 2920 2320
% 
\special{pn 4}%
\special{sh 1}%
\special{ar 2920 2320 8 8 0  6.28318530717959E+0000}%
\special{sh 1}%
\special{ar 2920 2320 8 8 0  6.28318530717959E+0000}%
% DOT 0 0 3 0 Black White
% 2 2920 2320 2920 2320
% 
\special{pn 4}%
\special{sh 1}%
\special{ar 2920 2320 16 16 0  6.28318530717959E+0000}%
\special{sh 1}%
\special{ar 2920 2320 16 16 0  6.28318530717959E+0000}%
% SPLINE 0 0 3 0 Black White
% 6 3610 2320 3780 2070 3940 1740 3780 1520 3210 1790 2920 2320
% 
\special{pn 20}%
\special{pa 3610 2320}%
\special{pa 3662 2242}%
\special{pa 3678 2216}%
\special{pa 3696 2190}%
\special{pa 3714 2162}%
\special{pa 3732 2136}%
\special{pa 3752 2110}%
\special{pa 3792 2056}%
\special{pa 3812 2028}%
\special{pa 3854 1972}%
\special{pa 3874 1944}%
\special{pa 3892 1916}%
\special{pa 3908 1886}%
\special{pa 3922 1858}%
\special{pa 3932 1828}%
\special{pa 3938 1798}%
\special{pa 3942 1768}%
\special{pa 3940 1738}%
\special{pa 3934 1706}%
\special{pa 3922 1676}%
\special{pa 3908 1646}%
\special{pa 3890 1618}%
\special{pa 3868 1590}%
\special{pa 3844 1566}%
\special{pa 3820 1544}%
\special{pa 3792 1526}%
\special{pa 3764 1512}%
\special{pa 3736 1502}%
\special{pa 3706 1496}%
\special{pa 3676 1494}%
\special{pa 3646 1496}%
\special{pa 3614 1500}%
\special{pa 3584 1508}%
\special{pa 3554 1518}%
\special{pa 3522 1532}%
\special{pa 3492 1546}%
\special{pa 3462 1564}%
\special{pa 3432 1584}%
\special{pa 3372 1628}%
\special{pa 3344 1652}%
\special{pa 3288 1704}%
\special{pa 3262 1732}%
\special{pa 3190 1816}%
\special{pa 3168 1844}%
\special{pa 3148 1870}%
\special{pa 3110 1926}%
\special{pa 3076 1982}%
\special{pa 3060 2012}%
\special{pa 3018 2096}%
\special{pa 2992 2152}%
\special{pa 2980 2180}%
\special{pa 2966 2208}%
\special{pa 2956 2236}%
\special{pa 2920 2320}%
\special{fp}%
\end{picture}}%   
        \label{fig15_1114ccc} 
      }
      \end{center}
      \caption{Bounded subcontinua of $(P_{\lambda, \epsilon})$ or $(P^{\quad \prime}_{\lambda, \epsilon})$.}
      \label{fig15_111402} 
    \end{figure} 
%%%%%%%%%%%%%%%%%%%%%%%%%%%%%%%%%%%%%%%%%%%%%%%

%%%%%%%%%%%%  
%	\begin{figure}[H] % 
	\begin{figure}[!htb]
		%WinTpicVersion4.28b
{\unitlength 0.1in
\begin{picture}( 19.6900, 21.4000)( 21.2100,-24.9000)
% VECTOR 2 0 3 0 Black White
% 2 2130 2320 4090 2320
% 
\special{pn 8}%
\special{pa 2130 2320}%
\special{pa 4090 2320}%
\special{fp}%
\special{sh 1}%
\special{pa 4090 2320}%
\special{pa 4024 2300}%
\special{pa 4038 2320}%
\special{pa 4024 2340}%
\special{pa 4090 2320}%
\special{fp}%
% VECTOR 2 0 3 0 Black White
% 2 2920 2490 2920 660
% 
\special{pn 8}%
\special{pa 2920 2490}%
\special{pa 2920 660}%
\special{fp}%
\special{sh 1}%
\special{pa 2920 660}%
\special{pa 2900 728}%
\special{pa 2920 714}%
\special{pa 2940 728}%
\special{pa 2920 660}%
\special{fp}%
% STR 2 0 3 0 Black White
% 4 2730 2380 2730 2480 5 0 0 0
% O
\put(27.3000,-24.8000){\makebox(0,0){O}}%
% STR 2 0 3 0 Black White
% 4 4250 2030 4250 2130 5 0 0 0
% $\lambda$
\put(42.5000,-21.3000){\makebox(0,0){$\lambda$}}%
% STR 2 0 3 0 Black White
% 4 2910 330 2910 430 5 0 0 0
% $u$
\put(29.1000,-4.3000){\makebox(0,0){$u$}}%
% STR 2 0 3 0 Black White
% 4 3040 1030 3040 1130 5 0 0 0
% $u_0$
\put(30.4000,-11.3000){\makebox(0,0){$u_0$}}%
% SPLINE 0 0 3 0 Black White
% 13 2920 2320 3040 2320 3400 2240 3740 1940 3800 1760 3550 1460 3060 1240 2750 1210 2290 1330 2130 1780 2580 2190 2850 2310 2920 2320
% 
\special{pn 20}%
\special{pa 2920 2320}%
\special{pa 2952 2322}%
\special{pa 3016 2322}%
\special{pa 3112 2316}%
\special{pa 3144 2312}%
\special{pa 3208 2302}%
\special{pa 3272 2286}%
\special{pa 3302 2278}%
\special{pa 3332 2268}%
\special{pa 3392 2244}%
\special{pa 3450 2216}%
\special{pa 3506 2184}%
\special{pa 3532 2166}%
\special{pa 3558 2146}%
\special{pa 3582 2126}%
\special{pa 3608 2104}%
\special{pa 3630 2082}%
\special{pa 3674 2034}%
\special{pa 3714 1982}%
\special{pa 3732 1954}%
\special{pa 3750 1924}%
\special{pa 3766 1894}%
\special{pa 3780 1864}%
\special{pa 3790 1832}%
\special{pa 3798 1802}%
\special{pa 3800 1772}%
\special{pa 3800 1742}%
\special{pa 3792 1714}%
\special{pa 3782 1686}%
\special{pa 3768 1658}%
\special{pa 3750 1632}%
\special{pa 3730 1606}%
\special{pa 3682 1558}%
\special{pa 3598 1492}%
\special{pa 3538 1452}%
\special{pa 3508 1434}%
\special{pa 3450 1398}%
\special{pa 3392 1366}%
\special{pa 3364 1350}%
\special{pa 3336 1336}%
\special{pa 3278 1310}%
\special{pa 3248 1298}%
\special{pa 3220 1286}%
\special{pa 3160 1266}%
\special{pa 3132 1258}%
\special{pa 3072 1242}%
\special{pa 3042 1236}%
\special{pa 3010 1230}%
\special{pa 2980 1226}%
\special{pa 2916 1218}%
\special{pa 2820 1212}%
\special{pa 2754 1210}%
\special{pa 2720 1210}%
\special{pa 2686 1212}%
\special{pa 2654 1212}%
\special{pa 2620 1214}%
\special{pa 2586 1218}%
\special{pa 2554 1222}%
\special{pa 2490 1234}%
\special{pa 2458 1242}%
\special{pa 2428 1252}%
\special{pa 2398 1264}%
\special{pa 2370 1278}%
\special{pa 2342 1294}%
\special{pa 2314 1312}%
\special{pa 2290 1332}%
\special{pa 2266 1354}%
\special{pa 2242 1378}%
\special{pa 2202 1434}%
\special{pa 2170 1494}%
\special{pa 2156 1526}%
\special{pa 2144 1558}%
\special{pa 2136 1592}%
\special{pa 2128 1624}%
\special{pa 2124 1658}%
\special{pa 2122 1690}%
\special{pa 2122 1722}%
\special{pa 2126 1754}%
\special{pa 2132 1784}%
\special{pa 2140 1814}%
\special{pa 2152 1842}%
\special{pa 2166 1868}%
\special{pa 2182 1894}%
\special{pa 2200 1918}%
\special{pa 2220 1942}%
\special{pa 2242 1966}%
\special{pa 2266 1988}%
\special{pa 2292 2008}%
\special{pa 2318 2030}%
\special{pa 2346 2050}%
\special{pa 2404 2088}%
\special{pa 2494 2142}%
\special{pa 2584 2192}%
\special{pa 2612 2208}%
\special{pa 2642 2226}%
\special{pa 2670 2240}%
\special{pa 2698 2256}%
\special{pa 2728 2270}%
\special{pa 2756 2282}%
\special{pa 2816 2302}%
\special{pa 2846 2310}%
\special{pa 2878 2316}%
\special{pa 2910 2320}%
\special{pa 2920 2320}%
\special{fp}%
% STR 2 0 3 0 Black White
% 4 3690 1220 3690 1320 5 0 0 0
% $\mathcal{C}$
\put(36.9000,-13.2000){\makebox(0,0){$\mathcal{C}$}}%
\end{picture}}%
	  \caption{A loop type subcontinuum of $(P_\lambda)$ when \eqref{abl} holds.} 
	\label{fig15_1114a}
	    \end{figure}
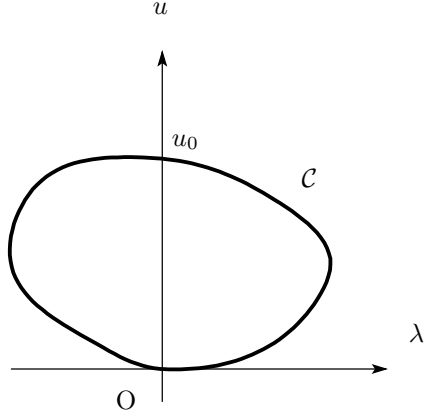 
%%%%%%%%%%%%%

The outline of this article is the following: in Section 2 we follow a variational approach based on the Nehari manifold method to prove Theorem \ref{t1}. In Section 3 we use the asymptotic profile of $u_{1,\lambda}$ to obtain a nontrivial non-negative solution of $(P_\lambda)$ for $\lambda>0$ small via the sub-supersolutions method. We also show that bifurcation from zero does not occur for $(P_\lambda)$ at any $\lambda > 0$. Section 4 is devoted to the proof of Theorem \ref{t2}. In Section 5 we carry out a bifurcation analysis for the regularized problem $(P_{\lambda,\epsilon})$. Finally, in Section 6 we prove Theorem \ref{thm:ep0:exist}.

\subsection{Notation}

Throughout this article we use the following notations and conventions: % \\
\begin{itemize}
\item The infimum of an empty set is assumed to be $\infty$.
\item Unless otherwise stated, for any $f \in L^1(\Om)$ the integral $\int_\Omega f$ is considered with respect to the Lebesgue measure, whereas for any $g \in L^1(\partial \Om)$ the integral $\int_{\partial \Om} g$  is considered with respect to the surface measure. 
\item For $r\geq 1$ the Lebesgue norm in $L^r (\Omega)$ will be denoted by $\| \cdot \|_r$ and the usual norm of $H^1(\Omega)$ by $\|\cdot \|$.% For $1<r<\infty$ the Holder conjugate of $r$ is denoted by $r'$.
\item The strong and weak convergence are denoted by $\rightarrow$ and $\rightharpoonup$, respectively. 
\item The positive
and negative parts of a function $u$ are defined by $u^{\pm} :=\max
\{\pm u,0\}$. 
\item If $U \subset \R^N$ then we denote the closure of $U$ by $\overline{U}$ and the interior of $U$ by $\text{int }U$.
\item The support of a measurable function $f$ is denoted by supp $f$.
\end{itemize}

\bigskip

\section{The variational approach}

\bigskip

Throughout this section we assume that $p<\frac{2N}{N-2}$.
We associate to $(P_\lambda)$ the $\mathcal{C}^1$ functional $I_\lambda$ defined on $X$ by
\begin{align*}
I_\lambda (u) := \frac{1}{2} E(u) - \frac{1}{p} A(u) - \frac{\lambda}{q}B(u), 
\end{align*}
where 
\begin{align*}
E(u) = \int_\Omega |\nabla u|^2, \quad 
A(u) = \int_\Omega a(x)|u|^{p}, \quad \text{and} \quad
B(u) = \int_{\Omega} b(x)|u|^q. 
\end{align*}
Let us recall that $X=H^1(\Omega)$ is equipped with the usual norm $\|u\|=\left[\int_\Omega \left(|\nabla u|^2 +u^2\right)\right]^{\frac{1}{2}}$. 
Critical points of $I_\lambda$ are weak solutions of $(P_\lambda)$, which are also classical solutions by standard regularity.
In the sequel we shall consider the following useful subsets of $X$:
\begin{align*}
& E^+ = \{ u \in X : E(u)>0 \}, \\
& A^{\pm} = \{ u \in X : A(u)\gtrless 0 \},  \quad A_0 = \{ u \in X : A(u)=0 \}, \quad A^{\pm}_0 = A^{\pm} \cup A_0.\\ 
& B^{\pm} = \{ u \in X : B(u) \gtrless 0 \}, \quad B_0 = \{ u \in X : B(u)=0 \}, \quad B^{\pm}_0 = B^{\pm} \cup B_0.
\end{align*}
The next result will be used repeatedly in this section:

\begin{lem} \label{lem:wc}
\strut
\begin{enumerate}
  \item If $(u_n)$ is a sequence such that $u_n \rightharpoonup u_0$ in $X$ and $\limsup_n E(u_n)\leq 0$ then $u_0$ is a constant and $u_n \rightarrow u_0$ in $X$. 
  \item Assume $\int_\Omega a<0$ (respect. $\int_\Omega b<0$). If $v\not \equiv 0$ and $v \in A_0^+$ (respect. $v \in B_0^+$) then $v$ is not a constant. 
\end{enumerate}
\end{lem}

\begin{proof} 
\strut 
\begin{enumerate}
  \item Since $u_n \rightharpoonup u_0$ in $X$ and $E$ is weakly lower semicontinuous, we have 
$$
0\leq E(u_0) \leq \liminf E(u_n) \leq 0. 
$$
Hence $E(u_0) = 0$, which implies that $u_0$ is a constant. Moreover $E(u_n) \to E(u_0)$. Since $u_n \to u_0$ in $L^2(\Omega)$ we deduce that $u_n \rightarrow u_0$ in $X$.\\
  \item If $v \in A_0^+$ is a non-zero constant then 
$0\leq A(v_0) = |v_0|^p \int_\Omega  a < 0$, which is a contradiction. 
\end{enumerate}
\end{proof}

The Nehari manifold associated to $I_\lambda$ is given by
$$
N_\lambda := \{ u \in X\setminus \{0\} : \langle I_\lambda'(u), u\rangle = 0 \} = \{ u \in X\setminus \{0\} : E(u) = A(u) + \lambda B(u) \}.
$$
We shall use the splitting $$N_\lambda= N_\lambda^+ \cup N_\lambda^- \cup N_\lambda^0,$$ where
\begin{align*}
N_\lambda^{\pm} := \left\{ u \in N_\lambda : \langle J_\lambda'(u), u \rangle \gtrless 0 \right\} &= \left\{ u \in N_\lambda : E(u) \lessgtr \lambda \frac{p-q}{p-2}B(u) \right\} \\ 
&= \left\{ u \in N_\lambda : E(u) \gtrless \frac{p-q}{2-q} A(u) \right\}, 
\end{align*}
and
$$N_\lambda^0= \{ u \in N_\lambda : \langle J_\lambda'(u), u \rangle = 0 \}.$$
Note that any nontrivial  solution of $(P_\lambda)$ belongs to $N_\lambda$. Furthermore, it follows from the implicit function theorem that $N_\lambda \setminus N_\lambda^0$ is a $C^1$ manifold and every critical point of the restriction of $I_\lambda$ to this manifold is a critical point of $I_\lambda$ (see for instance \cite[Theorem 2.3]{BZ03}), and therefore a solution of $(P_\lambda)$.

\begin{rem}
\label{runs}
{\rm
Note that any positive solution of $(P_\lambda)$ belonging to $N_\lambda^-$ is unstable. Indeed, if $u \in N_\lambda^-$ then 
$$E(u)-(p-1)A(u)-\lambda(q-1)B(u)=(2-q)E(u)-(p-q)A(u)<0.$$
It follows that $$\gamma_1(\lambda,u)=\inf\left\{ \int_\Omega \left(|\nabla \phi|^2 - (p-1)a(x)u^{p-2}\phi^2- \lambda(q-1)b(x)u^{q-2}\phi^2\right) :\ \|\phi\|_2=1\right\}<0.$$}
\end{rem}

To analyse the structure of $N_\lambda^\pm$, we consider the fibering maps corresponding to $I_\lambda$, which are set, for $u\not= 0$, as follows:
\begin{align*}
j_u(t) := I_\lambda (tu) = \frac{t^2}{2}E(u) - \frac{t^p}{p}A(u) - \lambda \frac{t^q}{q} B(u), \quad t>0. 
\end{align*}
It is easy to see that 
\begin{align*}
j_u'(1)=0 \lessgtr j_u''(1) \Longleftrightarrow u \in N_\lambda^{\pm}, 
\end{align*}
and more generally,
\begin{align*}
j_u'(t)=0 \lessgtr j_u''(t) \Longleftrightarrow tu \in N_\lambda^{\pm}. 
\end{align*}
Having this characterisation in mind, we look for conditions under which $j_u$ has a critical point. Set 
\begin{align*}
i_u(t) := t^{-q}j_u(t) = \frac{t^{2-q}}{2}E(u) - \frac{t^{p-q}}{p}A(u) - \lambda B(u), \quad t>0. 
\end{align*}
Let $u\in E^+ \cap A^+ \cap B^+$. Then $i_u$ has a global maximum $i_u(t^*)$ at some $t^* > 0$, and moreover, $t^*$ is unique. If $i_u(t^*) > 0$, then $j_u$ has a global maximum which is positive and a local minimum which is negative. Moreover, these are the only critical points of $j_u$. 
%
%
%%%%%%%%%%%%%%%%%%%%% Figures %%%%%%%%%%%%%%%%%%%%%%%%%%%%%%%
%
%	% \begin{figure}[H] 

%
%%%%%%%%%%%%%%%%%%%%%%%%%%%%%%%%%%%%%%%%%%%%%%%%%%%%%%%%%%%%%
%
We shall require a condition on $\lambda$ that provides $i_u(t^*)>0$. Note that 
$$
 i_u'(t) = \frac{2-q}{2}t^{1-q}E(u) - \frac{p-q}{p}t^{p-q-1}A(u) = 0$$
if and only if $$t = t^*:= \left( \frac{p(2-q)E(u)}{2(p-q)A(u)} \right)^{\frac{1}{p-2}}.$$
Moreover
$$ i_u(t^*) = \frac{p-2}{2(p-q)} \left( \frac{p(2-q)}{2(p-q)} \right)^{\frac{2-q}{p-2}} \frac{E(u)^{\frac{p-q}{p-2}}}{A(u)^{\frac{2-q}{p-2}}} - \frac{\lambda}{q}B(u) > 0$$ if and only if 
\begin{equation}0< \lambda < C_{pq} \frac{E(u)^{\frac{p-q}{p-2}}}{B(u)A(u)^{\frac{2-q}{p-2}}}, \label{iu*>0}
\end{equation}
where $C_{pq} = \left( \frac{q(p-2)}{2(p-q)} \right) \left( \frac{p(2-q)}{2(p-q)} \right)^{\frac{2-q}{p-2}}$. 
%= \frac{(q(p-2))^{\frac{p-2}{p-q}} (p(2-q))^{\frac{2-q}{p-q}}}{2(p-q)}. 
Note that $F(u) = \frac{E(u)^{\frac{p-q}{p-2}}}{B(u)A(u)^{\frac{2-q}{p-2}}}$ satisfies $F(tu) = F(u)$ for $t>0$, i.e. $F$ is homogeneous of order $0$. \\

We introduce now
\begin{align} \label{lam0}
\lambda_0 = \inf \{ E(u)^{\frac{p-q}{p-2}} : u \in E^+ \cap A^+ \cap B^+, \ C_{pq}^{-1} B(u)A(u)^{\frac{2-q}{p-2}} = 1 \}.
\end{align}
Note that if $E^+ \cap A^+ \cap B^+ =\emptyset$ then $\lambda_0=\infty$.
We deduce then the following result, which provides sufficient conditions for the existence of critical points of $j_u$:

\begin{prop} \label{prop:t1t2} 
Assume \eqref{ab<0}. Then $\lambda_0 > 0$ and, for $0<\lambda<\lambda_0$, there holds:
\begin{enumerate}
\item If either $u \in  A^+ \cap B^-$ or $u \in E^+ \cap A^+ \cap B_0$ then, $j_u$ has a positive global maximum at some $t_1 > 0$, i.e. $j_u'(t_1) = 0 > j_u''(t_1)$ and $j_u(t) < j_u(t_1)$ for $t\not= t_1$. Moreover, $t_1$ is the unique critical point of $j_u$. \\

\item If either $u \in  A^- \cap B^+$ or $E^+ \cap A_0 \cap B^+$ then $j_u$ has a negative global minimum at some $t_1 > 0$, i.e. $j_u'(t_1) = 0 < j_u''(t_1)$ and $j_u(t) > j_u(t_1)$ for $t\not= t_1$. Moreover, $t_1$ is the unique critical point of $j_u$. \\

\item  If $u\in E^+ \cap A^+ \cap B^+$ then $j_u$ has a negative local minimum at $t_1 > 0$ and a positive global maximum at $t_2 > t_1$. Furthermore  $t_1$ and $t_2$ are the only critical points of $j_u$.

\end{enumerate}
\end{prop}

\begin{proof}
First, we show that $\lambda_0 > 0$. Assume $\lambda_0 = 0$, so that we can choose $u_n \in E^+ \cap A^+ \cap B^+$ satisfying 
\begin{align*}
E(u_n) \rightarrow 0, \quad 
\mbox{and} \quad C_{pq}^{-1} B(u_n)A(u_n)^{\frac{2-q}{p-2}} = 1.\end{align*}
If $(u_n)$ is bounded in $X$ then we may assume that $u_n \rightharpoonup u_0$ for some $u_0 \in X$ and $u_n \to u_0$ in $L^p(\Omega)$ and  $L^q(\Omega)$. It follows from Lemma \ref{lem:wc}(1) that $u_0$ is a constant and $u_n \to u_0$ in $X$. From $u_n \in A^+ \cap B^+$ we deduce that $u_0 \in A_0^+ \cap B_0^+$. In addition, there holds
$$C_{pq}^{-1} B(u_0)A(u_0)^{\frac{2-q}{p-2}} = 1,$$
so that $u_0 \not \equiv 0$. From Lemma \ref{lem:wc} we get a contradiction.

Let us assume now that $\Vert u_n \Vert \to \infty$. Set $v_n = \frac{u_n}{\Vert u_n \Vert}$, so that $\Vert v_n \Vert = 1$. We may assume that $v_n \rightharpoonup v_0$ and $v_n \to v_0$ in $L^p(\Omega)$. Since $E(v_n) \to 0$ and $v_n \in A^+ \cap B^+$, we have $v_n \to v_0$ in $X$, $v_0$ is a constant, and $v_0 \in A_0^+ \cap B_0^+$. In particular, $\|v_0\|=1$, i.e. $v_0 \not \equiv 0$. Lemma \ref{lem:wc} provides again a contradiction.

Now, if $u \in  A^+ \cap B^-$ or $u \in E^+ \cap A^+ \cap B_0$ then it is clear that $j_u$ has a unique critical point, which is a global maximum point. In a similar way, if $u \in A^- \cap B^+$ then $j_u$ has a unique critical point, which is a global minimum point. Finally, if $u \in E^+ \cap A^+ \cap B^+$ then
\begin{align*}
\lambda_0 \leq C_{pq} \frac{E(u)^{\frac{p-q}{p-2}}}{B(u)A(u)^{\frac{2-q}{p-2}}}. 
\end{align*}
Thus, if $0<\lambda < \lambda_0$ then $i_u(t^*) > 0$ from \eqref{iu*>0}.
 This completes the proof. 
\end{proof} 

\begin{prop} \label{prop:Npm0}
Assume \eqref{ab<0}. Then, for $0< \lambda < \lambda_0$, we have: 
\begin{enumerate}
\item $N_\lambda^0$ is empty.
\item If $A^+ \cap B^+ \neq \emptyset$ then $N_\lambda^+$ and $N_\lambda^-$ are non-empty.  
\item If $A_0^- \cap B^+ \neq \emptyset$ then $N_\lambda^+$ is non-empty.  
\item If $A^+ \cap B_0^- \neq \emptyset$ then $N_\lambda^-$ is non-empty.
\end{enumerate}
\end{prop}

\begin{proof}
 From Proposition \ref{prop:t1t2} it follows that there is no $t>0$ such that $j_u'(t)=j_u''(t)=0$, i.e. $N_\lambda^0$ is empty. Let $u_0 \in A^+ \cap B^+$. Since $\int_\Omega a<0$ or $\int_\Omega b<0$, from Lemma \ref{lem:wc} it follows that $u_0 \in E^+ \cap A^+ \cap B^+$. By Proposition \ref{prop:t1t2} we infer that for $0< \lambda < \lambda_0$ there are $0<t_1<t_2$ such that $t_1 u \in N_\lambda^+$ and $t_2 u \in N_\lambda^-$. Assertions (3) and (4) are straightforward.
\end{proof}

The following result provides some properties of $N_\lambda^+$ and $N_\lambda^-$:

\begin{lem}  \label{lem:N+}
Assume \eqref{ab<0}.  Then we have:
\strut
\begin{enumerate}
 \item $N_\lambda^+ \subset B^+$ and $N_\lambda^- \subset A^+$.
  \item $N_\lambda^+$ is bounded in $X$ for $0<\lambda < \lambda_0$ . 
\end{enumerate}
\end{lem}

\begin{proof} 
\strut
\begin{enumerate}
\item Let $u \in N_\lambda^+$. Then $0\leq E(u) < \lambda \frac{p-q}{p-2}B(u)$, i.e. $u \in B^+$.
Now, if $u \in N_\lambda^-$ then $0\leq E(u) < \frac{p-q}{2-q} A(u)$, i.e. $u \in A^+$.\\
  \item Assume $(u_n) \subset N_\lambda^+$ and $\Vert u_n \Vert \to \infty$. Set $v_n = \frac{u_n}{\Vert u_n \Vert}$. It follows that $\Vert v_n \Vert = 1$, so we may assume that $v_n \rightharpoonup v_0$ , $B(v_n)$ is bounded, and $v_n \to v_0$ in $L^p(\Omega)$ (implying $A(v) \to A(v_0)$). Since $u_n \in N_\lambda^+$, we see that $$E(v_n) < \lambda \frac{p-q}{p-2} B(v_n) \Vert u_n \Vert^{q-2},$$ and thus $\limsup E(v_n) \leq 0$. Lemma \ref{lem:wc}(1) yields that $v_0$ is a constant and $v_n \to v_0$ in $X$. Consequently, $\Vert v_0 \Vert = 1$, and $v_0$ is a non-zero constant. On the other hand, since $u_n \in N_\lambda^+$, we have $v_n \in N_\lambda^+$, so $v_n \in B^+$. It follows that $v_0 \in B_0^+$. Finally, from 
  $$E(u_n)=\lambda B(u_n)+A(u_n)$$
  we deduce that $A(v_n) \to 0$, i.e. $v_0 \in A_0$. Therefore $v_0 \in A_0 \cap B_0^+$, which contradicts 
Lemma \ref{lem:wc}(2).
\end{enumerate}
\end{proof}

\begin{prop} 
\label{p1}
Assume \eqref{ab<0} and $\Omega^b_+ \neq \emptyset$.
Then there exists $u_{1,\lambda} \geq 0$ such that
$I_\lambda (u_{1,\lambda}) = \displaystyle \min_{N_\lambda^+} I_\lambda $ for $0<\lambda <\lambda_0$. In particular, $u_{1,\lambda}$ is a nontrivial non-negative solution  of $(P_\lambda)$ for $0<\lambda <\lambda_0$. 
\end{prop}

\begin{proof} Let $0<\lambda < \lambda_0$. We consider a minimizing sequence $(u_n) \subset N_\lambda^+$, i.e.
\begin{align*}
I_\lambda (u_n) \rightarrow \inf_{N_\lambda^+}I_\lambda < 0. 
\end{align*}
Since $(u_n)$ is bounded in $X$, we may assume that $u_n \rightharpoonup u_0$ in $X$ and $u_n \to u_0$ in $L^p(\Omega)$ and $L^q(\Omega)$. It follows that $$I_\lambda (u_0) \leq \liminf I_\lambda (u_n) = \inf_{N_\lambda^+}I_\lambda (u) < 0,$$ so that 
$u_0 \not \equiv 0$. Moreover, as $u_n \in B^+$ we have $u_0 \in B_0^+$. We claim that $u_0 \in B^+$. Indeed, if $u_0 \in B_0$ then, from $$E(u_n)<\lambda \frac{p-q}{p-2}B(u_n)$$ we obtain $E(u_0)=0$, i.e. $u_0$ is a constant. So $\int_\Omega b=0$, and consequently, by \eqref{ab<0}, we have $\int_\Omega a<0$. It follows that 
$j_{u_0}(t)= -\frac{1}{p}t^p |u_0|^p \int_\Omega a>0$ for every $t>0$, which contradicts $j_{u_0}(1)=I_\lambda(u_0)<0$. Thus $u_0 \in B^+$
and by Proposition \ref{prop:t1t2} we have  $t_1 u_0 \in N_\lambda^+$ for some $t_1 > 0$. Assume $u_n \not\to u_0$. If $1<t_1$ then we have 
\begin{equation}
I_\lambda (t_1u_0) = j_{u_0}(t_1) \leq j_{u_0}(1) < \liminf j_{u_n}(1) = \liminf I_\lambda (u_n) = \inf_{N_\lambda^+}I_\lambda,
\end{equation}
which is impossible.
If $t_1\leq 1$ then $j_{u_n}'(t_1)\leq 0$ for every $n$, so that $j_{u_0}'(t_1)<\liminf j_{u_n}'(t_1)\leq 0$, which is a contradiction. Therefore  $u_n \to u_0$.
Now, since $u_n \to u_0$ we have $j_{u_0}'(1)=0 \leq j_{u_0}''(1)$.
But $j_{u_0}''(1)=0$ is impossible by Proposition \ref{prop:Npm0}(1). 
Thus $u_0 \in N_\lambda^+$ and $I_\lambda (u_0)=\displaystyle \inf_{N_\lambda^+} I_\lambda$. We set $u_{1,\lambda}=u_0$.
\end{proof}

Next we obtain a second nontrivial non-negative solution of $(P_\lambda)$, which achieves  $\displaystyle \inf_{N_\lambda^-} I_\lambda$ for $\lambda \in (0, \lambda_0)$. The following result provides some properties of $N_\lambda^-$:

\begin{lem}  \label{lem:Nlma-}
Assume \eqref{ab<0}.  Then, for $0<\lambda < \lambda_0$, we have $I_\lambda (u) > 0$ for any $u\in N_\lambda^-$. 
\end{lem}

\begin{proof} 
Let $u \in N_\lambda^-$. By Lemma \ref{lem:N+} we know that $u \in A^+$. If $u \in B_0^+$ then $u$ is non-constant, i.e. $u \in E^+$. Hence either $u \in A^+ \cap B_0^+ \cap E^+$ or $u \in A^+ \cap B^-$. In both cases, by Proposition \ref{prop:t1t2} we have that $t=1$ is the global maximum point of $j_u$ and $I_\lambda(u)=j_u(1)>0$. 
\end{proof}

\begin{prop} 
\label{p2}
Assume \eqref{ab<0} and $\Omega^a_+ \neq \emptyset$.
Then there exists $u_{2,\lambda} \geq 0$ such that
$I_\lambda (u_{2,\lambda}) = \displaystyle \min_{N_\lambda^-} I_\lambda$ for $\lambda \in (0, \lambda_0)$. In particular, $u_{2,\lambda}$ is a nontrivial non-negative solution  of $(P_\lambda)$ for $\lambda \in (0, \lambda_0)$. 
\end{prop}

\begin{proof} 
Since $I_\lambda (u) > 0$ for $u \in N_\lambda^-$, we can choose $u_n \in N_\lambda^-$ such that 
\begin{align*}
I_\lambda (u_n) \rightarrow \inf_{N_\lambda^-}I_\lambda  \geq 0. 
\end{align*}
We claim that $(u_n)$ is bounded in $X$. Indeed, there exists $C>0$ such that $I_\lambda (u_n) \leq C$. Since $u_n \in N_\lambda$, we deduce 
\begin{align}
\label{en-}
\left( \frac{1}{2} - \frac{1}{p} \right) E(u_n) - \lambda \left( \frac{1}{q} - \frac{1}{p} \right) B(u_n) =I_\lambda(u_n)\leq C. 
\end{align}
Assume $\Vert u_n \Vert \to \infty$ and set $v_n = \frac{u_n}{\Vert u_n \Vert}$, so that $\Vert v_n \Vert = 1$. We may assume that $v_n \rightharpoonup v_0$, and $v_n \to v_0$ in $L^p(\Omega)$ and  $L^q(\Omega)$. Then, from
\begin{align*} 
\left( \frac{1}{2} - \frac{1}{p} \right) E(v_n) \leq \lambda \left( \frac{1}{q} - \frac{1}{p} \right) B(v_n)\Vert u_n \Vert^{q-2} + C\| u_n \|^{-2}, 
\end{align*}
we infer that $\limsup E(v_n) \leq 0$. Lemma \ref{lem:wc}(1) yields that $v_0$ is a constant, and $v_n \to v_0$ in $X$, which implies $\Vert v_0 \Vert = 1$. On the other hand, since $u_n \in N_\lambda^-$, we have
$u_n \in A^+$, so that $v_n \in A^+$ and consequently $v_0 \in A_0^+$, which is impossible if $\int_\Omega a<0$. Let us assume now $\int_\Omega b<0$. Since $$E(u_n)>\lambda \frac{p-q}{p-2} B(u_n),$$
from \eqref{en-} we obtain
$$\lambda \frac{(p-q)(q-2)}{2pq} 	 B(u_n) \leq C,$$
so that $$ \lambda \frac{(p-q)(q-2)}{2pq} B(v_n) \leq C\|u_n\|^{-q}.$$
It follows that $v_0 \in B_0^+$ and we get a contradiction. Hence $(u_n)$ is bounded. 
We may then assume that $u_n \rightharpoonup u_0$ in $X$ and $u_n \to u_0$ in $L^p(\Omega)$ and $L^q(\Omega)$. If $u_0 \equiv 0$ then we set $v_n=\frac{u_n}{\|u_n\|}$. From
$$E(u_n)<\frac{p-q}{2-q}A(u_n),$$ we get
$$E(v_n)<\frac{p-q}{2-q}A(v_n)\|u_n\|^{p-2} \rightarrow 0.$$
So we can assume that $v_n \rightarrow v_0$ with $v_0$ constant and non-zero.
Moreover, from $$E(u_n)=\lambda B(u_n)+A(u_n)$$ we deduce that $B(v_n) \rightarrow 0$, i.e. $B(v_0)=0$, so $v_0 \in B_0 \cap A_0^+$, which contradicts our assumption. 
Thus $u_0 \not \equiv 0$. 
Note also that $E(u_0)\leq \frac{p-q}{2-q}A(u_0)$. We claim that $u_0 \in A^+$. Indeed, if $u_0 \in A_0$ then $E(u_0)=0$ i.e. $u_0$ is a non-zero constant. Hence $\int_\Omega a=0$. By \eqref{ab<0} we must have $\int_\Omega b<0$ and consequently $u_0 \in B^-$. But from $E(u_n)=A(u_n)+\lambda B(u_n)$ we obtain $E(u_0)\leq \lambda B(u_0)$, and consequently $B(u_0)\geq 0$, which is a contradiction. Therefore $u_0 \in A^+$. Furthermore, if $E(u_0)=0$ then $u_0$ is a non-zero constant so $u_0 \in B^-$. Summing up, we have either $u_0 \in A^+ \cap B^-$ or $u_0 \in A^+ \cap B_0^+ \cap E^+$. By Proposition \ref{prop:t1t2} we infer the existence of $t_2>0$ such that $t_2 u_0 \in N_\lambda^-$. 
Assume now $u_n\not\to u_0$. Then, since $u_n \in N_\lambda^-$, we get
$$I_\lambda(t_2 u_0)<\liminf I_\lambda (t_2 u_n)\leq \liminf I_\lambda(u_n)=\inf_{N_\lambda^-} I_\lambda,$$
which is a contradiction. Therefore $u_n \rightarrow u_0$. In particular, we get $j_{u_0}'(1)=0$ and $j_{u_0}''(1)\leq 0$. Since $N_\lambda^0$ is empty for $\lambda \in (0,\lambda_0)$, we infer that $u_0 \in N_\lambda^-$ and $I_\lambda (u_0)=\displaystyle \inf_{N_\lambda^-} I_\lambda$. We set $u_{2,\lambda}=u_0$.
\end{proof}

We discuss now the asymptotic profiles of $u_{1,\lambda}, u_{2,\lambda}$ as $\lambda \to 0^+$. 

\begin{lem}
\label{leqil}
 Assume $\int_\Omega b<0$. Then, for $0<\lambda<\lambda_0$, there holds
 \begin{equation}
 \label{eqil}
 I_{\lambda}(u_{1,\lambda})<-D_0\lambda^{\frac{2}{2-q}} -D_1 \lambda^{\frac{p}{2-q}},
\end{equation}  
for some $D_0>0$ and $D_1\geq 0$.
\end{lem}

\begin{proof}
We have $ I_{\lambda}(u_{1,\lambda})\leq I_\lambda(u)$ for any $u \in N_\lambda^+$. Let us take $u \in N_\lambda^+$. Thus $u \in B^+$ and it follows that $u$ is non-constant, i.e. $u \in E^+$.
\begin{itemize}
\item If $u \in A_0^+$ then $$I_\lambda(u) \leq  \tilde{I}_\lambda(u):=\frac{1}{2}E(u)-\frac{\lambda}{q}B(u).$$
Thus $I_\lambda(tu)\leq \tilde{I}_\lambda(tu)$ for every $t>0$.
Note that $\tilde{I}_\lambda(tu)$ has a global minimum point $t_0$ given by
$$t_0=\left(\frac{\lambda B(u)}{E(u)}\right)^{\frac{1}{2-q}}.$$
and $$\tilde{I}_\lambda(t_0u)=-\frac{2-q}{2q}\lambda t_0^q B(u)= -\frac{2-q}{2q}\frac{(\lambda B(u))^{\frac{2}{2-q}}}{E(u)^{\frac{q}{2-q}}}=-D_0\lambda^{\frac{2}{2-q}},$$
where $D_0=\frac{2-q}{2q}\frac{ B(u)^{\frac{2}{2-q}}}{E(u)^{\frac{q}{2-q}}}$.
It follows that $$I_\lambda(u)\leq \tilde{I}_\lambda(t_0u)= -D_0\lambda^{\frac{2}{2-q}}$$
with $D_0>0$. \\

\item If $u \in A^-$ then we consider $t_0$ as in the previous item to obtain $$I_\lambda(u) \leq I_\lambda (t_0u)=-D_0\lambda^{\frac{2}{2-q}}-D_1 \lambda^{\frac{p}{2-q}},$$
where $D_0=\frac{2-q}{2q}\frac{ B(u)^{\frac{2}{2-q}}}{E(u)^{\frac{q}{2-q}}}$ and $D_1=\frac{1}{p} \left( \frac{B(u)}{E(u)}\right)^{\frac{p}{2-q}} A(u)$.
\end{itemize}
Therefore in both cases we have $$I_\lambda(u_{1,\lambda})<-D_0\lambda^{\frac{2}{2-q}}-D_1 \lambda^{\frac{p}{2-q}}$$
for some $D_0>0$ and $D_1 \geq 0$.

\end{proof}

We determine now the asymptotic profile of $u_{1,\lambda}$ as $\lambda \to 0^+$:

\begin{prop} \label{prop:sol:Nlam+:0}
 Assume \eqref{ab<0} and $\Omega_+^b \neq \emptyset$. Then $ u_{1,\lambda} \to 0$ in $C^2(\overline{\Omega})$ as $\lambda \to 0^+$. More precisely:
\begin{enumerate}
\item If $\int_\Omega b <0$ and $\lambda_n \to 0^+$ then, 
 up to a subsequence, $\lambda_n^{-\frac{1}{2-q}}u_{1,\lambda_n} \to w_0$ in $C^2(\overline{\Omega})$, where $w_0$ is a nontrivial non-negative ground state solution of 
\begin{equation*}
\begin{cases}
-\Delta w =  b(x)|w|^{q-2}w & \mbox{in $\Omega$}, \\
\frac{\partial w}{\partial \n} = 0 & \mbox{on $\partial \Omega$}.
\end{cases} 
\end{equation*}
\item If $\int_\Omega a<0\leq \int_\Omega b$ then $\lambda^{-\frac{1}{p-q}}u_{1,\lambda} \to c^*$ in $C^2(\overline{\Omega})$ as $\lambda \to 0^+$. In particular, if $\int_\Omega a < 0 < \int_\Omega b$ then $u_{1,\lambda}$ is positive on $\overline{\Omega}$ for $\lambda > 0$ sufficiently small. 
\end{enumerate}
\end{prop}

\begin{proof} 
First we show that $u_{1,\lambda}$ remains bounded in $X$ as $\lambda \to 0^+$. Indeed,  assume that $\lambda_n \to 0$ and $\|u_n\| \to \infty$, where $u_n=u_{1,\lambda_n}$. We set $v_n = 
\frac{u_n}{\Vert u_n \Vert}$ and assume that for some $v_0 \in X$ we have $v_n \rightharpoonup v_0$ in $X$, 
and $v_n \to v_0$ in $L^p(\Omega)$ and $L^q(\Omega)$. Since $u_n
\in N_{\lambda_n}$, we have
\begin{align*}
E(v_n) \Vert u_n \Vert^{2-p} = A(v_n) + \lambda_n B(v_n)
\Vert u_{n} \Vert^{q-p}.  
\end{align*}
Passing to the limit  we obtain $A(v_0)=0$, i.e. $v_0 \in A_0$.
From $u_n \in N_{\lambda_n}^+$ we have 
\begin{align*}
E(v_n) < \lambda_n \frac{p-q}{p-2} B(v_n) \Vert u_n \Vert^{q-2}, 
\end{align*}
so that $\limsup E(v_n) \leq 0$. By Lemma \ref{lem:wc}(1) we infer that $v_n \to v_0$ in $X$ and
$v_0$ is a non-zero constant. But $v_0 \in A_0 \cap B_0^+$, which contradicts Lemma \ref{lem:wc}(2).  Therefore
$(u_n)$ is bounded in $X$.

Hence we may assume that $u_n \rightharpoonup u_{0}$ in $X$ and $u_n \to u_{0}$ 
in $L^p(\Omega)$ and $L^q(\Omega)$.
From
\begin{equation}
\label{eu1}
E(u_n) < \lambda_n \frac{p-q}{p-2} B(u_n). 
\end{equation}
 we get $\limsup E(u_n)\leq 0$. 
Lemma \ref{lem:wc}(2) provides that $u_{0}$ is a constant and $u_n \to u_{0}$ in $X$. 
Since $u_n \in B^+$, we have $u_0 \in B_0^+$. Finally, from
$$E(u_n)=\lambda_n B(u_n)+A(u_n)$$
we infer that $$0=E(u_0)\leq A(u_0),$$
i.e. $u_0 \in A_0^+$. Thus $u_0 \in A_0^+ \cap B_0^+$, and by Lemma \ref{lem:wc}(2) we deduce that $u_0 \equiv 0$. Thus we have proved that $u_n \to 0$ in $X$. By standard regularity we get $u_n \to 0$ in $C^2(\overline{\Omega})$.

Next we obtain the precise profile of $u_n$. We consider two cases:
\begin{enumerate}
\item Assume $\int_\Omega b <0$. Let $w_n = \lambda_n^{-\frac{1}{2-q}}u_n$. We claim that $(w_n)$ 
is bounded in $X$. Indeed, from \eqref{eu1}
we have $$E(w_n) < \frac{p-q}{p-2} B(w_n).$$ 
Let us assume that $\Vert w_n \Vert \to \infty$ and set
$\psi_n = \frac{w_n}{\Vert w_n \Vert}$. We may assume that
$\psi_n \rightharpoonup \psi_0$ and $\psi_n \to \psi_0$ in $L^p(\Omega)$ 
and  $L^q(\Omega)$. It follows that 
\begin{align*}
E(\psi_n) < \frac{p-q}{p-2}  B(\psi_n) 
\Vert w_\lambda \Vert^{q-2}, 
\end{align*}
so that $\limsup E(\psi_n) \leq 0$. By Lemma \ref{lem:wc}(1) we infer that $\psi_0$ is a constant and $\psi_n \to \psi_0$ in $X$. On the other hand, from $u_n \in B^+$ we have $\psi_n \in B^+$ and consequently $\psi_0 \in B_0^+$. From \eqref{ab<0} we infer that $\psi_0 \equiv 0$, which contradicts $\| \psi_0\|=1$. Hence $(w_n)$ is bounded is $X$ and
we may assume that $w_n \rightharpoonup w_0$ in $X$ and $w_n \to w_0$ in $L^p(\Omega)$ 
and $L^q(\Omega)$. Note that $w_n$ satisfies
\begin{align} \label{vlam:wsol}
\int_\Omega \nabla w_n \nabla w - \lambda_n^{\frac{p-2}{2-q}} 
\int_\Omega a(x) w_n^{p-1} w -  \int_\Omega  
b(x)w_n^{q-1}w = 0, \quad \forall w \in X.
\end{align}
Taking $w=w_n - w_0$  we deduce that 
$w_n \to w_0$ in $X$ and consequently in $C^2(\overline{\Omega})$. Moreover, $w_0$ is a solution of
\eqref{plb}.
We claim that $w_0\not \equiv 0$. Indeed, 
by Lemma \ref{leqil} we have
$$ I_{\lambda_n}(u_n)<-D_0\lambda_n^{\frac{2}{2-q}}-D_1 \lambda_n^{\frac{p}{2-q}},$$
with $D_0>0$ and $D_1 \geq 0$.
Hence $$\frac{\lambda_n^{\frac{2}{2-q}}}{2} E(w_n)- \frac{\lambda_n^{\frac{p}{2-q}}}{p} A(w_n) -\frac{\lambda_n^{\frac{2}{2-q}}}{q} B(w_n) <-D_0\lambda_n^{\frac{2}{2-q}}-D_1 \lambda_n^{\frac{p}{2-q}},$$
so that 
$$\frac{1}{2}E(w_n)-\frac{\lambda_n^{\frac{p-2}{2-q}}}{p} A(w_n)-B(w_n)<-D_0-D_1 \lambda_n^{\frac{p-2}{2-q}}.$$
We obtain then
$$\frac{1}{2}E(w_0)-B(w_0)\leq -D_0,$$
and consequently $w_0 \not \equiv 0$. 

It remains to prove that $w_0$ is a ground state solution of \eqref{plb}, i.e. 
$$I_b(w_0)=\min_{N_b} I_b,$$
where $$I_b(u)=\frac{1}{2}E(u) - \frac{1}{q}B(u)$$ for $u \in X$ and $$N_b=\{u \in X \setminus \{0\};\ \langle I_b'(u),u \rangle =0\}=\{u \in X \setminus \{0\}; \ E(u)=B(u)\}$$
is the Nehari manifold associated to $I_b$. Since $\int_\Omega b<0$ it is easily seen that there exists $w_b\neq 0$ such that $I_b(w_b)=\min_{N_b} I_b$. Note that since $w_0$ is a nontrivial solution of \eqref{plb} we have
$w_0 \in N_b$ and consequently $I_b(w_b)\leq I_b(w_0)$. We prove now the reverse inequality. Since $w_b$ is non-constant, we have $w_b \in B^+ \cap E^+$. We set $u_b=\lambda^{\frac{1}{2-q}} w_b$. Let $\lambda_n \to 0^+$. Since $u_b \in B^+ \cap E^+$ for every $n$ there exists $t_n>0$ such that $t_n u_b \in N_{\lambda_n}^+$. 
Hence $$ t_n^2 E(u_b)<\lambda_n \frac{p-q}{p-2}t_n^q B(u_b),$$ i.e.
$$t_n^{2-q}<\frac{p-q}{p-2}\frac{B(w_b)}{E(w_b)}=\frac{p-q}{p-2}.$$
We may then assume that $t_n \to t_0$. We claim that $t_0=1$. Indeed, note that from $t_n u_b \in N_{\lambda_n}^+$ we infer that 
$$t_n^2 E(u_b) =\lambda_n t_n^qB(u_b)+t_n^pA(u_b)$$
so $$t_n^{2-q} E(w_b) =B(w_b)+t_n^{p-q}\lambda_n^{\frac{p-2}{2-q}}A(w_b).$$
From $E(w_b)=B(w_b)$ we infer that $t_0=1$, as claimed. Now, from
$$I_{\lambda_n}(u_{1,{\lambda_n}})\leq I_{\lambda_n} (t_n u_b)$$
it follows that
$$I_{\lambda_n}(u_{1,\lambda_n}) \leq \left( \frac{1}{2} - \frac{1}{q} \right)t_n^2E(u_b) - \left( \frac{1}{p}-\frac{1}{q}  \right)  t_n^pA(u_b).$$
Hence $$\frac{\lambda_n^{\frac{2}{2-q}}}{2} E(w_n)- \frac{\lambda_n^{\frac{p}{2-q}}}{p} A(w_n) -\frac{\lambda_n^{\frac{2}{2-q}}}{q} B(w_n)  \leq \frac{q-2}{2q} t_n^2\lambda_n^{\frac{2}{2-q}} E(w_b) - \frac{q-p}{pq} \lambda_n^{\frac{p}{2-q}} t_n^pA(w_b),$$
i.e.
$$\frac{1}{2}E(w_n)-\frac{\lambda_n^{\frac{p-2}{2-q}}}{p} A(w_n)-\frac{1}{q}B(w_n)\leq \frac{q-2}{2q}t_n^2E(w_b) - \frac{q-p}{pq} \lambda_n^{\frac{p-2}{2-q}} t_n^pA(w_b).$$
Since $w_n \to w_0$ in $X$ we obtain $$I_b(w_0)\leq \left( \frac{1}{2} - \frac{1}{q} \right)E(w_b) =I_b(w_b).$$
Therefore $I_b(w_0)=I_b(w_b)$, as claimed. 
\\

\item Assume now $\int_\Omega a<0\leq \int_\Omega b$ and set $w_n = \lambda_n^{-\frac{1}{p-q}}u_n$. We claim that $(w_n)$ 
is bounded in $X$. Indeed, since $u_n \in N_{\lambda_n}^+$, 
we have $$E(w_n) < \frac{p-q}{p-2} \lambda_n^{\frac{p-2}{p-q}}B(w_n).$$ 
Let us assume that $\Vert w_n \Vert \to \infty$ and set
$\psi_n = \frac{w_n}{\Vert w_n \Vert}$. We may assume that
$\psi_n \rightharpoonup \psi_0$ and $\psi_n \to \psi_0$ in $L^p(\Omega)$ 
and  $L^q(\Omega)$. It follows that 
\begin{align*}
E(\psi_n) < \frac{p-q}{p-2} \lambda_n^{\frac{p-2}{p-q}} B(\psi_n) 
\Vert w_n \Vert^{q-2}, 
\end{align*}
so that $\limsup E(\psi_n) \leq 0$. By Lemma \ref{lem:wc}(1) we infer that $\psi_0$ is a constant and $\psi_n \to \psi_0$ in $X$. On the other hand, from $u_n \in N_{\lambda_n}$ it follows that $$0\leq A(u_n) + \lambda_n B(u_n),$$ so that $$-B(\psi_n)\Vert w_n \Vert^{q-p} 
\leq A(\psi_n).$$ Taking the limit  we get $0\leq A(\psi_0)$, 
which contradicts Lemma \ref{lem:wc}(2). Hence $(w_n)$ is bounded in $X$  and
we may assume that $w_n \rightharpoonup w_0$ in $X$ and $w_n \to w_0$ in $L^p(\Omega)$ 
and $L^q(\Omega)$. It follows that $\limsup E(w_n) \leq 0$, 
and by Lemma \ref{lem:wc}(1) we get that $w_0$ is a constant and $w_n \to w_0$ in $X$. So $w_n \to w_0$ in $C^2(\overline{\Omega})$.

It remains to show that $w_0 = c^*$. We note that $w_n$ satisfies 
\begin{align} \label{vlam:wsol1}
\int_\Omega \nabla w_n \nabla w - \lambda_n^{\frac{p-2}{p-q}} 
\int_\Omega a w_n^{p-1} w - \lambda_n^{\frac{p-2}{p-q}} \int_\Omega  
b w_n^{q-1}w = 0, \quad \forall w \in X, 
\end{align}
since $u_n$ is a solution of $(P_{\lambda_n})$. We infer that
\begin{align*}
\int_\Omega a w_n^{p-1} + \int_\Omega  b w_n^{q-1} = 0.
\end{align*}
Passing to the limit, we see that either $w_0= 0$ or 
$w_0 = c^*$. However, 
taking now $w = (w_n+\varepsilon)^{1-q}$ with $\varepsilon >0$ in \eqref{vlam:wsol1}, we obtain
\begin{align*}
0 > (1-q) \int_\Omega |\nabla w_n|^2 (w_n+\varepsilon)^{-q} = \lambda^{\frac{p-2}{p-q}} 
\left( \int_\Omega a \frac{w_n^{p-1}}{(w_n+\varepsilon)^{q-1}} + \int_\Omega b \left( \frac{w_n}{w_n + \varepsilon}\right)^{q-1}\right).
\end{align*}
so that $$\int_\Omega b \left( \frac{w_n}{w_n + \varepsilon}\right)^{q-1} < - \int_\Omega a \frac{w_n^{p-1}}{(w_n+\varepsilon)^{q-1}}.$$
Letting $\varepsilon \to 0$ and using the Lebesgue dominated convergence theorem, we get
$$\int_{\text{supp } w_n} b \leq -\int_\Omega a w_n^{p-q}.$$
In particular, since $b \leq 0$ on $\Omega \setminus \text{supp } w_n$, we have
$$\int_{\Omega} b \leq -\int_\Omega a w_n^{p-q}.$$
Letting now $n \to \infty$, we obtain
$$0\leq \int_{\Omega} b \leq -\int_\Omega a w_0^{p-q}.$$
Now, if $\int_{\Omega} b=0$ then $c^*=0$, so $w_0 =0$.
On the other hand, if $\int_{\Omega} b>0$ then we must have $w_0 \neq 0$, i.e.  $w_0 = c^*$. In both cases
we obtain $\lambda_n^{-\frac{1}{p-q}} u_n \to c^*$ in $X$. By elliptic regularity we deduce  that $\lambda^{-\frac{1}{p-q}} u_{1,\lambda} \to c^*$ in $C^2(\overline{\Omega})$. 
Additionally if $\int_\Omega a < 0 < \int_\Omega b$ then $c^* > 0$, so that, by continuity,  $u_{1, \lambda} > 0$ on $\overline{\Omega}$ for $\lambda > 0$ sufficiently small. 
\end{enumerate}
\end{proof}

We consider now the asymptotic behavior of $u_{2,\lambda}$ as $\lambda \to 0^+$:

\begin{lem} \label{lem:u2lam<C} 
Assume \eqref{ab<0} and $\Omega^a_+ \neq \emptyset$. Then there exists a constant $C>0$ such that 
 $\| u_{2,\lambda} \| \leq C$ as $\lambda \to 0^+$. 
\end{lem}

\begin{proof} 
First we show that there exists a constant $C_1 > 0$ such that $I_\lambda (u_{2,\lambda}) \leq C_1$ for every
$\lambda \in (0,\lambda_0)$. 
To this end, we consider the following eigenvalue problem with the Dirichlet boundary condition. 
\begin{align} \label{eigenprob:D}
\begin{cases}
-\Delta \varphi = \lambda a(x)\varphi & \mbox{in $\Omega$}, \\ 
\varphi = 0 & \mbox{on $\partial \Omega$}.  
\end{cases}
\end{align}
We denote by $\lambda_D = \lambda_D(\Omega)$ the positive 
principal eigenvalue of \eqref{eigenprob:D} and by $\varphi_D = \varphi_D(\Omega)$ a positive eigenfunction associated to $\lambda_D$. Taking $\varphi_D^{p-1}$ as test function, we see that $\int_\Omega a \varphi_D^p=\int_\Omega |\nabla \varphi_D|^2>0$, i.e. $\varphi_D \in E^+ \cap A^+$. By Proposition \ref{prop:t1t2} there exists $t_2(\lambda)$ such that 
%$j_{\varphi^*}$ has a global maximum at some $t_2 > 0$, which implies 
$t_2(\lambda) \varphi_D \in N_\lambda^-$. 
Note that $$0<j_{\varphi_D}(t_2(\lambda))=\frac{t_2(\lambda)^2}{2}E(\varphi_D)-\frac{t_2(\lambda)^p}{p}A(\varphi_D)-\lambda \frac{t_2(\lambda)^q}{q}B(\varphi_D).$$
Thus $t_2(\lambda)$ stays bounded as $\lambda \to 0^+$.
Consequently, there holds
$$I_\lambda(t_2(\lambda) \varphi_D)=\frac{q-2}{2q}t_2(\lambda)^2 E(\varphi_D)+\frac{p-q}{pq} t_2(\lambda)^p A(\varphi_D)\leq \frac{p-q}{pq} t_2(\lambda)^p A(\varphi_D)\leq C,$$
for some constant $C>0$.

We assume now that $\lambda_n \to 0^+$ and  $\|u_n\| \to \infty$, where $u_n=u_{2,\lambda_n}$.  We set $v_n=\frac{u_n}{\|u_n\|}$ and assume that $v_n \rightharpoonup v_0$ in $X$ and $v_n \to v_0$ in $L^p(\Omega)$ and $L^q(\Omega)$.
Since $I_{\lambda_n}(u_n) =\displaystyle \min_{N_{\lambda_n}^-} I_{\lambda_n}$, we have 
\begin{align*}
\left( \frac{1}{2} - \frac{1}{p} \right)E(u_n) - \left( \frac{1}{q} - \frac{1}{p} \right) \lambda_n  B(u_n) =I_{\lambda_n} (u_n) \leq C.  
\end{align*}
Hence
\begin{align*}
\left( \frac{1}{2} - \frac{1}{p} \right) E(v_n) \leq \left( \frac{1}{q} - \frac{1}{p} \right) 
\lambda_n B(v_n)\Vert u_{n} \Vert^{q-2} + C_1 \Vert u_n \Vert^{-2}. 
\end{align*}
Letting $\lambda_n \to 0^+$ we obtain $\limsup_\lambda E(v_n)\leq 0$, and by Lemma \ref{lem:wc} we infer that $v_0$ is a constant and $v_n \to v_0$ in $X$. In particular, $\Vert v_0 \Vert = 1$. Moreover, from
$$\int_\Omega \left(\nabla u_n \nabla \phi -\lambda_n bu_n^{q-1} \phi - au_n^{p-1} \phi\right)=0 \quad \forall \phi \in X,$$
we get $$\int_\Omega av_0^{p-1}\phi=\lim \int_\Omega av_n^{p-1}\phi=0\quad \forall \phi \in X,$$ which provides 
$av_0^{p-1} \equiv 0$, so that $v_0 =0$, and we get a contradiction. 
Therefore $(u_{2,\lambda})$ stays bounded in $X$ as $\lambda \to 0^+$. 
\end{proof}

\begin{prop}
\label{asyu2}
Assume  \eqref{ab<0} and $\Omega^a_+ \neq \emptyset$.
\begin{enumerate}
\item If $\int_\Omega a \geq 0>\int_\Omega b$ then $u_{2,\lambda} \to 0$ in $C^2(\overline{\Omega})$ as $\lambda \to 0^+$. If, in addition, $\int_\Omega a>0$ then $\lambda^{-\frac{1}{p-q}}u_{2,\lambda} \to c^*$ in $C^2(\overline{\Omega})$ as $\lambda \to 0^+$. In this case $u_{2,\lambda}$ is a unstable positive solution of $(P_\lambda)$ for $\lambda>0$ sufficiently small.\\
\item If $\int_\Omega a <0$ and $\lambda_n \to 0^+$ then, up to a subsequence, $u_{2,\lambda_n} \to u_{2,0}$ in $C^2(\overline{\Omega})$, where $u_{2,0}$ is a positive ground state solution of \eqref{pla}. In this case $u_{2,\lambda}$ is a unstable positive solution of $(P_\lambda)$ for $\lambda>0$ sufficiently small.
%\begin{align} 
%\begin{cases}
%-\Delta u = a(x)u^{p-1} & \mbox{in $\Omega$}, \\
%\frac{\partial u}{\partial \n} = 0 & \mbox{on $\partial \Omega$}. 
%\end{cases} 
%\end{align} \\ 
\end{enumerate}
\end{prop}

\begin{proof}
Let $\lambda_n \to 0^+$. By Lemma \ref{lem:u2lam<C}, up to a subsequence, we have $u_n=u_{2,\lambda_n} \rightharpoonup u_0$ in $X$ and $u_n \to u_0$ in $L^p(\Omega)$ and $L^q(\Omega)$. Since $u_n$ is a solution of $(P_{\lambda_n})$ it follows that $u_n \rightarrow u_0$ in $C^2(\overline{\Omega})$ and $u_0$ is a non-negative solution of \eqref{pla}.
This problem has a nontrivial non-negative solution if and only if $\int_\Omega a<0$. Hence $u_0 \equiv 0$ if $\int_\Omega a \geq 0$. 
\begin{enumerate}
\item Let us now assume that $\int_\Omega a>0>\int_\Omega b$.
We set $w_n=\lambda_n^{-\frac{1}{p-q}} u_n$. Then $w_n$ is a non-negative solution of
\begin{align}
\begin{cases} \label{eqw}
-\Delta w = \lambda^{\frac{p-2}{p-q}}a(x)w^{p-1} + \lambda^{\frac{p-2}{p-q}}b(x)w^{q-1} & \mbox{in $\Omega$}, \\ 
\frac{\partial w}{\partial \n} = 0 & \mbox{on $\partial \Omega$}.  
\end{cases}
\end{align}
for $\lambda=\lambda_n$. We claim that $(w_n)$ is bounded in $X$. Indeed, assume that $\|w_n\| \to \infty$ and $\psi_n = \frac{w_n}{\|w_n\|} \rightharpoonup \psi_0$ in $X$ with $\psi_n \to \psi_0$ in $L^p(\Omega)$ and $L^q(\Omega)$. Let $c_\lambda=\left(\frac{-\lambda \int_\Omega  b}{\int_\Omega a}\right)^{\frac{1}{p-q}}$.  We use now the fact that  $c_\lambda \in N_\lambda^-$ for any $\lambda >0$. Hence
$$I_{\lambda_n}(u_n)\leq I_{\lambda_n} (c_{\lambda_n})=D \lambda_n^{\frac{p}{p-q}},$$
where
$D=\frac{p-q}{pq} \frac{\left(-\int_\Omega  b \right)^{\frac{p}{p-q}}}{\left(\int_\Omega a\right)^{\frac{q}{p-q}}}$.
Thus $$\frac{p-2}{2p}\lambda_n^{\frac{2}{p-q}} E(w_n)-\frac{p-q}{pq} \lambda_n^{\frac{p}{p-q}} B(w_n) \leq D \lambda_n^{\frac{p}{p-q}},$$
so that $$\frac{p-2}{2p} E(w_n)-\frac{p-q}{pq} \lambda_n^{\frac{p-2}{p-q}} B(w_n) \leq D \lambda_n^{\frac{p-2}{p-q}}.$$
Dividing the latter inequality by $\|w_n\|^2$ we get $E(\psi_n) \to 0$, and consequently $\psi_n \to \psi_0$ in $X$ and $\psi_0$ is a constant. Furthermore, integrating \eqref{eqw} we obtain
\begin{equation}
\label{int1}\int_\Omega aw_n^{p-1} + \int_\Omega  bw_n^{q-1}=0,
\end{equation}
so that $\int_\Omega a\psi_n^{p-1} \to 0$, i.e. $\int_\Omega a\psi_0^{p-1} =0$, and consequently $\int_\Omega a =0$, which is a contradiction. Therefore $(w_n)$ is bounded in $X$. We may assume then that $w_n \rightharpoonup w_0$ in $X$ and $w_n \rightarrow w_0$ in $L^p(\Omega)$ and $L^q(\Omega)$. It follows that
$$\int_\Omega \nabla w_0 \nabla \phi =0,\quad \forall \phi \in X.$$
Hence $w_0$ is a constant and $w_n \to w_0$ in $X$, and consequently in $C^2(\overline{\Omega})$.
It remains to show that $w_0 \neq 0$. If $w_0 =0$ then we set again $\psi_n = \frac{w_n}{\|w_n\|}$. From
$$E(w_n)<\frac{p-2}{p-q} \lambda_n^{\frac{p-2}{p-q}} A(w_n),$$
we infer that $E(\psi_n) \to 0$, so that $\psi_n \to \psi_0$ in $X$ and $\psi_0$ is a constant. Moreover, from 
$$0\leq A(w_n)+B(w_n)$$
we have $$-\|w_n\|^{p-q} A(\psi_n)\leq B(\psi_n),$$
so that $B(\psi_0) \geq 0$. By Lemma \ref{lem:wc}(2) we get a contradiction. Therefore we have proved that $w_0$ is a non-zero constant. Finally, from \eqref{int1} we obtain
$$w_0^{p-1} \int_\Omega a=-w_0^{q-1} \int_\Omega  b,$$ i.e.
$w_0 = c^*$. In particular, we infer that $u_{2,\lambda}$ is positive for $\lambda>0$ sufficiently small. Finally, from Remark \ref{runs} we infer that $u_{2,\lambda}$ is unstable whenever it is positive.\\
\item Let us assume now $\int_\Omega a<0$ and show that $u_{2,0}$ is a positive ground state solution of \eqref{pla}, i.e. 
$$I_a(u_{2,0})=\min_{N_a} I_a,$$
where $$I_a(u)=\frac{1}{2}E(u) - \frac{1}{p}A(u)$$ for $u \in X$ and $$N_a=\{u \in X \setminus \{0\};\ \langle I_a'(u),u \rangle =0\}=\{u \in X \setminus \{0\}; \ E(u)=A(u)\}$$
is the Nehari manifold associated to $I_a$. Since $\int_\Omega a<0$ it is easily seen that there exists $u_a\neq 0$ such that $I_a(u_a)=\min_{N_a} I_a$. Note that since $u_{2,0}$ is a nontrivial solution of \eqref{pla} we have
$u_{2,0} \in N_a$ and consequently $I_a(u_a)\leq I_a(u_{2,0})$. We prove now the reverse inequality. Since $u_a$ is non-constant, we have $u_a \in A^+ \cap E^+$. Thus for any $\lambda<\lambda_0$ there exists $t_{\lambda}>0$ such that $t_{\lambda}u_0 \in N_\lambda^-$. Thus
$$t_{\lambda}^2 \frac{E(u_a)}{2} - \lambda t_{\lambda}^q\frac{B(u_a)}{q}-t_{\lambda}^p \frac{A(u_a)}{p}=I_\lambda(t_{\lambda}u_a)>0,$$
which implies that $t_{\lambda}$ remains bounded as $\lambda \to 0$. 
We may then assume that $t_{\lambda} \to t_0$  as $\lambda \to 0$. We claim that $t_0=1$. Indeed, first note that from $t_{\lambda}u_a \in N_\lambda^-$ we have $$t_{\lambda}^2 E(u_a)<\frac{p-q}{2-q} t_{\lambda}^p A(u_a),$$
and consequently $t_{\lambda}^{p-2}>\frac{2-q}{p-q} \frac{E(u_a)}{A(u_a)}$. Hence $t_0>0$.
In addition, from $t_{\lambda}u_a \in N_\lambda$ we have
$$t_{\lambda}^2 E(u_a) =\lambda t_{\lambda}^qB(u_a)+t_{\lambda}^pA(u_a)$$
so $$t_{\lambda}^{2-q} E(u_a) =t_{\lambda}^{p-q}A(u_a)+o(1)$$
as $\lambda \to 0$. Since $E(u_a)=A(u_a)$ we infer that $t_0=1$, as claimed. Now, from
$$I_{\lambda}(u_{2,\lambda})\leq I_\lambda (t_{\lambda}u_a)$$
it follows that
$$\frac{1}{2}E(u_{2,\lambda}) - \frac{\lambda}{q}B(u_{2,\lambda})- \frac{1}{p}A(u_{2,\lambda})\leq \left( \frac{1}{2} - \frac{1}{p} \right)t_{\lambda}^2E(u_a) - \left( \frac{1}{q} - \frac{1}{p} \right) \lambda t_{\lambda}^qB(u_a).$$
Letting $\lambda \to 0$ and using that $u_{2,\lambda} \to u_{2,0}$ in $X$ we obtain $$I_a(u_{2,0})\leq \left( \frac{1}{2} - \frac{1}{p} \right)E(u_a) =I_a(u_a).$$
Therefore $I_a(u_{2,0})=I_a(u_a)$, as claimed. 

Finally, let us show that $u_{2,\lambda}$ is positive on $\overline{\Omega}$ for $\lambda >0$ sufficiently small. Indeed, assume by contradiction that for every $\lambda>0$ there exists $0<\mu <\lambda$ such that $u_{2,\mu}$ is non-negative but vanishes somewhere on $\overline{\Omega}$. Then we obtain a sequence $\mu_n \to 0^+$ such that $u_n=u_{2,\mu_n}$ are non-negative solutions vanishing somewhere on $\overline{\Omega}$.  But up to a subsequence, $(u_n)$ converges in 
$C^2(\overline{\Omega})$ to a positive function, which is a contradiction.
The proof is now complete.
\end{enumerate}
\end{proof}

\begin{rem} \label{rasyu2}
\strut
{\rm 
\begin{enumerate}

\item If $\Omega_+^a \neq \emptyset$ and $\int_\Omega a<0< \int_\Omega b$ then Propositions \ref{prop:sol:Nlam+:0} and \ref{asyu2} provide us with some $\lambda^*>0$ such that $u_{2,\lambda}>u_{1,\lambda}$ for $0<\lambda<\lambda^*$. Indeed, assume on the contrary that there are two sequences $\lambda_n \to 0^+$ and $x_n \in \overline{\Omega}$ such that $u_{2,\lambda_n} (x_n) \leq u_{1,\lambda_n}(x_n)$ and $u_{2,\lambda_n} \to u_{2,0}$ in $C^2(\overline{\Omega})$. Let $\epsilon = \frac{\min_{\overline{\Omega}} u_{2,0}}{2} > 0$. Since $u_{1,\lambda_n} \to 0$ in $C^2(\overline{\Omega})$, there exists $n_0 \in \mathbb{N}$ such that $\max_{\overline{\Omega}} u_{1,\lambda_n} < \epsilon$ for $n\geq n_0$. It follows that $u_{1,\lambda_n}(x_n) < \epsilon$ for $n\geq n_0$, and thus that $u_{2,\lambda_n}(x_n) < \epsilon$ for such $n$. However, since $u_{2,\lambda_n} \to u_{2,0}$ in $C^2(\overline{\Omega})$, there exists $n_1 \in \mathbb{N}$ such that  $\min_{\overline{\Omega}} u_{2,\lambda_n}>\epsilon$ for $n\geq n_1$, which is a contradiction.  The same result holds if $\Omega_+^b \neq \emptyset$ and $\int_\Omega a>0>\int_\Omega b$. Indeed, one can apply a similar argument to $w_{1,\lambda}=\lambda^{-\frac{1}{p-q}}u_{1,\lambda}$ and $w_{2,\lambda}=\lambda^{-\frac{1}{p-q}}u_{2,\lambda}$. We use now the fact that $w_{2,\lambda} \to c^*$ in $C^2(\overline{\Omega})$ as $\lambda \to 0^+$ and if $\lambda_n \to 0^+$ then, up to a subsequence, $w_{1,\lambda_n} \to 0$ in $C^2(\overline{\Omega})$.\\

\item One can show that ground state solutions of \eqref{pla} converge to $0$ in $C^2(\overline{\Omega})$ as $\int_\Omega a \nearrow 0$. More precisely, let $a_n=a^+ -\delta_n a^-$ where $\delta_n$ is a sequence such that $\delta_n \searrow \delta_0=\frac{\int_\Omega a^+}{\int_\Omega a^-}$. Then $a_n \nearrow a_0=a^+ - \delta_0 a^-$ and $\int_\Omega a_0 =0$. We denote by $u_n$ a ground state solution of \eqref{pla} with $a=a_n$. First we show that $(u_n)$ is bounded in $X$. If not, we set $v_n=\frac{u_n}{\|u_n\|}$ and assume that $v_n \rightharpoonup v_0$ in $X$. We set $A_n(u)=\int_\Omega a_n|u|^p$ and $I_n(u)=\frac{1}{2}E(u)-\frac{1}{p}A_n(u)$ for $u \in X$. In addition, we denote by $N_n$ the Nehari manifold associated to $I_n$. 
Let $\phi \in X$ be such that $a^- \phi \equiv 0$ and $a^+ \phi \not \equiv 0$. Then
$$I_n(t\phi)=\frac{t^2}{2}E(\phi) -\frac{t^p}{p}\int_\Omega a^+|\phi|^p=C(t).$$
Since $\phi$ is non-constant there exists a unique $t_0>0$ such that $t_0\phi \in N_n$ for every $n$. It follows that $$\frac{p-2}{2p}E(u_n)= I_n(u_n)\leq I_n(t_0\phi)=C(t_0).$$
Consequently we have $E(v_n) \to 0$, so that $v_n \rightarrow v_0$ and $v_0$ is a constant. Moreover, since $$\int_\Omega \nabla u_n \nabla w=\int_\Omega a_n u_n^{p-1} w \quad \forall w \in X,$$
we deduce that $$\int_\Omega a_n v_n^{p-1} w \to 0\quad \forall w \in X$$
Thus $$\int_\Omega a_0 v_0^{p-1} w =0\quad \forall w \in X$$ and consequently $a_0 v_0^{p-1} \equiv 0$. Hence $v_0 =0$, which contradicts $\|v_n\|=1$ for every $n$. Therefore $(u_n)$ is bounded and, up to a subsequence, we have $u_n \rightarrow u_0$ in $X$. Moreover $u_0$ is a solution of \eqref{pla} with $a=a_0$. Finally, since $\int_\Omega a_0=0$ we infer that $u_0 \equiv 0$, i.e. $u_n \to 0$ in $X$, and consequently in $C^2(\overline{\Omega})$.

\end{enumerate} }
\end{rem}

\bigskip

\section{Some results via sub-supersolutions}

\bigskip

We use now the asymptotic profile of $u_{1,\lambda}$ as $\lambda \to 0$ to show that for $\lambda>0$ sufficiently small a solution of $(P_\lambda)$ can be obtained by the sub-supersolutions method. In particular, the assumption $p<2^*$ can be dropped.

\begin{prop}
\label{psubsup}
Assume that $\Omega^b_+ \neq \emptyset$ and $\int_\Omega b<0$. Then there exists $\Lambda_0>0$ such that $(P_\lambda)$ has a nontrivial non-negative solution $U_\lambda$ for $0<\lambda<\Lambda_0$. Moreover $U_\lambda \to 0$ in $X$ as $\lambda \to 0^+$.
\end{prop}

\begin{proof}
First we obtain a supersolution of $(P_\lambda)$. To this end, we consider the problem
$$ 
\begin{cases}
-\Delta w = (b(x)+\delta)|w|^{q-2}w & \mbox{in $\Omega$}, \\
\frac{\partial w}{\partial \n} = 0 & \mbox{on $\partial \Omega$}.
\end{cases} 
$$
If $\delta>0$ is such that $\int_\Omega (b+\delta)<0$ then this problem has a nontrivial non-negative solution $w_\delta$.
We set $\overline{u}=\lambda^{\frac{1}{2-q}} w_\delta$.
Then $\overline{u}$ is a weak supersolution of $(P_\lambda)$ if
$$\lambda^{\frac{1}{2-q}}\int_\Omega (b(x)+\delta)w_\delta(x)^{q-1}v\geq \lambda^{\frac{p-1}{2-q}} \int_\Omega a(x) w_\delta(x)^{p-1} v +\lambda^{\frac{1}{2-q}} \int_\Omega b(x)w_\delta(x)^{q-1}v$$
for every non-negative $v \in X$.
It suffices then to have
$$\lambda^{\frac{1}{2-q}}(b(x)+\delta)w_\delta(x)^{q-1}\geq \lambda^{\frac{p-1}{2-q}}  a(x) w_\delta(x)^{p-1}  +\lambda^{\frac{1}{2-q}}b(x)w_\delta(x)^{q-1} $$
for {\it a.e. } $x \in \Omega$.
If $w_\delta(x)=0$ or $a(x) \leq 0$ then the latter inequality is clearly satisfied. Now, if $w_\delta(x)>0$ and $a(x)>0$ then it is equivalent to
$$\delta \geq \lambda^{\frac{p-2}{2-q}} a(x)w_\delta^{p-q},$$
which is satisfied if $$\lambda \leq \Lambda_0:=\left(\delta\|a^+\|_{\infty}^{-1} \|w_\delta\|_\infty^{q-p}\right)^{\frac{2-q}{p-2}}.$$
On the other hand, since $\Omega^b_+ \neq \emptyset$ there exist a subdomain $\Omega' \subset \Omega$ and $\delta'>0$ such that $b \geq \delta'$ in $\Omega'$. Let $\phi_1'$ be a positive eigenfunction associated to $\lambda_1'$, the first eigenvalue of $-\Delta u=\lambda u$ in $\Omega'$, $u=0$ in $\partial \Omega'$. We extend $\phi_1'$ by zero to $\Omega \setminus \Omega'$ and set $\underline{u}=\varepsilon \phi_1'$, where $\varepsilon>0$. Then we have, for a non-negative $v \in X$,
$$\int_\Omega \nabla \underline{u} \nabla v =\varepsilon \int_{\Omega'} \nabla \phi_1' \nabla v=\varepsilon \left(\int_{\partial \Omega'} \frac{\partial \phi_1'}{\partial \n}v+\lambda_1' \int_{\Omega'} \phi_1' v\right)\leq \varepsilon \lambda_1' \int_{\Omega'} \phi_1' v,$$
since $\frac{\partial \phi_1'}{\partial \n}<0$ on $\partial \Omega'$.
Hence $\underline{u}$ is a weak subsolution of $(P_\lambda)$ if, for {\it a.e.} $x \in \Omega'$, we have
$$\varepsilon \lambda_1' \phi_1' \leq a(\varepsilon \phi_1')^{p-1} + \lambda b (\varepsilon \phi_1')^{q-1},$$
i.e.
$$ \lambda_1' (\varepsilon\phi_1')^{2-q} \leq a(\varepsilon \phi_1')^{p-q} + \lambda b.$$ 
This inequality is clearly satisfied for $\varepsilon>0$ sufficiently small since $b \geq \delta'>0$ in $\Omega'$. Finally, taking $\varepsilon>0$ smaller if necessary, we have $\varepsilon \phi_1' \leq \overline{u}$ in $\Omega$. By \cite[Theorem 2]{H} we deduce that $(P_\lambda)$ has a solution $U_\lambda$ which satisfies $\varepsilon \phi_1' \leq U_\lambda \leq \lambda^{\frac{1}{2-q}} w_\delta$ in $\Omega$ for $\lambda<\Lambda_0$. In particular, we have $U_\lambda \to 0$ in $C(\overline{\Omega})$, and consequently in $X$, as $\lambda \to 0^+$. 
\end{proof}

We prove now that bifurcation of nontrivial non-negative solutions from zero can not occur at any $\lambda>0$.

\begin{lem} \label{l:boundbelow}
\label{l1}
Assume that $\Omega^b_+$ is a subdomain of $\Omega$. Let $\overline{\lambda}>0$ and $D$ be a subdomain such that $\overline{D}\subset \Omega^b_+$. Then there exists $C_{\overline{\lambda}}>0$ such that $u \geq C_{\overline{\lambda}}$ in $\overline{D}$ for every $u \in B^+$ which is a non-negative solution of $(P_\lambda)$ for $\lambda \geq \overline{\lambda}$. 
\end{lem}

\begin{proof}
We use a variant of a comparison principle for concave problems due to 
Ambrosetti-Brezis-Cerami \cite[Lemma 3.3]{ABC}. Let $u \in B^+$ be a nontrivial non-negative solution of $(P_\lambda)$ for $\lambda \geq \overline{\lambda}$. First we claim that $u > 0$ in $\Omega^b_+$. Indeed, since $u \in B^+$, we deduce that $u$ is positive somewhere in $\Omega^b_+$. It follows that there exists a constant $M>0$ large such that $(-\Delta + M)u \geq 0$
and $(-\Delta u + M)u \not\equiv 0$ in $\Omega^b_+$. The strong maximum principle provides then the desired conclusion. 

We apply now \cite[Lemma 3.3]{ABC} to the following concave problem 
\begin{align} \label{p:concave}
\begin{cases}
-\Delta v = -a_0 v^{p-1} + \lambda b_0 v^{q-1} & \mbox{in $D$}, \\
v = 0 & \mbox{on $\partial D$}, 
\end{cases}
\end{align}
where $a_0 = \sup_D a^-$ and $b_0 = \inf_D b$. It is clear that $u$ is a supersolution of 
\eqref{p:concave}. Next we construct a subsolution of \eqref{p:concave}. To this end, we use a positive eigenfunction $\phi_1$ associated to the first eigenvalue $\lambda_1 > 0$ of the Dirichlet eigenvalue problem
\begin{align*}
\begin{cases}
-\Delta \phi = \lambda \phi & \mbox{in $D$}, \\
\phi = 0 & \mbox{on $\partial D$}.
\end{cases}
\end{align*}
We normalize $\phi_1$ by $\Vert\phi_1 \Vert_{\mathcal{C}(\overline{D})} = 1$. 
Then %$\delta \phi_1$ is a subsolution of \eqref{p:concave} if
\begin{align*}
-\Delta (\delta \phi_1) -\left\{ -a_0 (\delta \phi_1)^{p-1} + \lambda b_0 (\delta \phi_1)^{q-1} \right\} 
&\leq (\delta \phi_1)^{q-1} \left\{ \lambda_1 \delta^{2-q} + a_0 \delta^{p-q} - \overline{\lambda} b_0 \right\} \\
& \leq (\delta \phi_1)^{q-1} \left\{ 2 \lambda_1 \delta^{2-q} - \overline{\lambda} b_0 \right\}
\end{align*}
if $x\in D$, $\lambda \geq \overline{\lambda}$ and $0<\delta \leq \overline{\delta}$ for some $\overline{\delta}$ sufficiently small. Thus $c_{\overline{\lambda}} \, \phi_1$ is a subsolution of \eqref{p:concave} for $\lambda \geq \overline{\lambda}$ if we set 
\begin{align*}
c_{\overline{\lambda}} = \min \left\{ \left( \frac{\overline{\lambda} \, b_0}{2\lambda_1} \right)^{\frac{1}{2-q}}, \ \overline{\delta} \right\}>0.
\end{align*}
The comparison principle ensures then that $c_{\overline{\lambda}} \phi_1 \leq u$ in $\overline{D}$, from which the desired conclusion follows. 
\end{proof}

\begin{prop}
\label{prop:nobifposi}
Under the assumptions of Lemma \ref{l:boundbelow}, bifurcation from zero never occurs for $(P_\lambda)$ at any $\lambda > 0$. More precisely, it never occurs that there exist $\lambda_n$ and nontrivial non-negative solutions $u_{\lambda_n}$ of $(P_{\lambda_n})$ such that $\lambda_n \to \lambda^*>0$ and $u_n \to 0$ in $C(\overline{\Omega})$. 
\end{prop}

\begin{proof}
Assume by contradiction that $\lambda_n \to \lambda^*>0$ and $u_n$ is a nontrivial non-negative solution of $(P_{\lambda_n})$ with $u_n \to 0$ in 
$C(\overline{\Omega})$. By Lemma \ref{l1} we must have $u_n \in B_0^-$ for $n$ large enough. Moreover, we have $u_n \to 0$ in $X$. We set $v_n= \frac{u_n}{\|u_n\|}$ and assume that $v_n \rightharpoonup v_0$ in $X$ and $v_n \to v_0$ in $L^p(\Omega)$ and $L^q(\Omega)$. Given $\phi \in X$ we have
\begin{equation}
\label{en}
\int_\Omega \left(\nabla u_n \nabla \phi -au_n^{p-1} \phi - \lambda_n bu_n^{q-1}\phi\right)=0.
\end{equation}
Hence $\lambda_n \int_\Omega bv_n^{q-1}\phi \to 0$, and consequently
$\int_\Omega bv_0^{q-1}\phi=0$. Since this holds for any $\phi \in X$, we deduce that $bv_0 \equiv 0$. Taking $\phi=v_0$ in \eqref{en} we obtain
$$\int_\Omega \nabla u_n \nabla v_0= \int_\Omega au_n^{p-1}v_0.$$
It follows that $\int_\Omega \nabla v_n \nabla v_0 \to 0$, so $\int_\Omega |\nabla v_0|^2=0$ i.e. $v_0$ is a constant. Therefore $v_0 \equiv 0$. Now, since $u_n \in B_0^-$ for $n$ large enough and $N_{\lambda_n}^+ \subset B^+$, we have $u_n \in N_{\lambda_n}^- \cup N_{\lambda_n}^0$, i.e. 
$$E(u_n) \leq \frac{p-q}{2-q} A(u_n)$$ for $n$ large enough.
We infer that $\limsup E(v_n) \leq 0$, so that $v_n \to 0$, which contradicts $\|v_n\|=1$. The proof is now complete.
\end{proof}

\bigskip
\section{Nonexistence results}
\bigskip

\begin{prop}
\label{p:nonex}
Assume $\Omega_+^a \cap \Omega_+^b \neq \emptyset$. Then there exists $\overline{\lambda}>0$ such that $(P_\lambda)$ has no positive solution for $\lambda>\overline{\lambda}$.
\end{prop}

\begin{proof}
Assume that $(P_\lambda)$ has a solution $u \geq 0$. By continuity, there exists $\delta>0$ and a ball $D \subset \Omega$ such that $a,b \geq \delta >0$ in $D$.
Let $\varphi>0$ be an eigenfunction associated to $\lambda_1'=\lambda_1(D)$, i.e. $\varphi$ is a solution of $-\Delta \varphi=\lambda_1 \varphi$ in $D$, $\varphi=0$ on $\partial D$. Then $$\int_{D} \nabla \varphi \nabla u = \lambda_1' \int_D  \varphi u +\int_{\partial D} \frac{\partial \varphi}{\partial \n} u.$$
On the other hand, extending $\varphi$ by zero to $\Omega$ and using it as test function in $(P_\lambda)$, we have
$$\int_\Omega \nabla \varphi \nabla u =\int_\Omega \left(au^{p-1} + \lambda bu^{q-1}\right) \varphi.$$
Hence
$$0\geq \int_{\partial D} \frac{\partial \varphi}{\partial \n} u=\int_D \left(au^{p-1} + \lambda bu^{q-1}-\lambda_1'u\right) \varphi \geq \int_D  \left( \delta u^{p-1} + \lambda \delta u^{q-1}-\lambda_1'u\right) \varphi$$
But for $\lambda$ large enough we have $\delta s^{p-1} + \lambda \delta s^{q-1}-\lambda_1's \geq 0$ for every $s \geq 0$. 
Therefore for such $\lambda$ we must have $u \equiv 0$.
\end{proof}

\begin{prop}
Let $\lambda > 0$. Then the following two assertions hold:
\strut
\begin{enumerate}
\item Assume  $b \geq 0$ and $\int_\Omega a \geq 0$. Then $(P_\lambda)$ has no nontrivial non-negative solution.
\item Assume that $b$ changes sign, $\Omega_+^b$ is a subdomain of $\Omega$, and $\Omega^b_- = \Omega \setminus \overline{\Omega^b_+}$. If $a \geq 0$ and $\int_\Omega b \geq 0$ then $(P_\lambda)$ has no non-negative solution taking positive values somewhere in $\Omega_+^b$.
\end{enumerate}
\end{prop}

\begin{proof}
\strut
\begin{enumerate}
\item Let $u \geq 0$ be a nontrivial solution of $(P_\lambda)$. Since $b \geq 0$, by the strong maximum principle, we have $u>0$ on $\overline{\Omega}$. Thus we may take $u^{1-p}$ as test function to get 
$$(1-p)\int_\Omega |\nabla u|^2 (u+\varepsilon)^{-p} -\int_\Omega a - \lambda \int_{\Omega} bu^{q-p}=0.$$
Hence $$\int_\Omega a <- \lambda \int_{\Omega} bu^{q-p} <0.$$\\

\item  Let $u \geq 0$ be a solution of $(P_\lambda)$ such that $u(x_0)>0$ for some $x_0 \in \Omega_+^b$. Since $\Omega_+^b$is a subdomain, by the strong maximum principle we have $u>0$ in $ \Omega_+^b$. Given $\varepsilon >0$, we take $w=(u+\varepsilon)^{1-q}$ to get
$$(1-q)\int_\Omega |\nabla u|^2 (u+\varepsilon)^{-q} -\int_\Omega au^{p-1} (u+\varepsilon)^{1-q} - \lambda \int_{\Omega} b \left(\frac{u}{u+\varepsilon}\right)^{q-1}=0.$$
Since $q>1$ we obtain
$$\lambda \int_{\Gamma_u} b \left(\frac{u}{u+\varepsilon}\right)^{q-1}<-\int_\Omega au^{p-1} (u+\varepsilon)^{1-q},$$
where $\Gamma_u= \text{supp }u$.
Letting $\varepsilon \to 0$ and using the Lebesgue dominated convergence theorem, we get
$$\lambda \int_{\Gamma_u} b \leq -\int_\Omega au^{p-q}.$$
Now, since $b < 0$ in $\Omega \setminus \Gamma_u$
we have $$ \int_{\Omega} b=\int_{\Gamma_u} b + \int_{\Omega \setminus \Gamma_u} b < \int_{\Gamma_u} b\leq -\lambda^{-1}\int_\Omega au^{p-q}.$$
and the conclusion follows.
\end{enumerate}
\end{proof}

\bigskip
\section{Bifurcation for a regularized problem} 
\bigskip

In this section we deal with the following Neumann boundary value problem with $\lambda \in \R$ and $\epsilon >0$:
\begin{align} \label{rp:m}
\begin{cases}
-\Delta u = a(x)|u|^{p-2}u + \lambda m(x) |u+\epsilon|^{q-2}u & \mbox{in} \ \Omega, \\
\frac{\partial u}{\partial \mathbf{n}} = 0 & \mbox{on} \ \partial \Omega.
\end{cases} 
% \leqno{(P_{\lambda,\epsilon})}  
\end{align}
Here $m \in C^\alpha (\overline{\Omega})$, $\alpha \in (0,1)$, satisfies 
\begin{align} \label{c:m}
\Omega^m_+ \ne \emptyset, \quad \mbox{and} \ \ \int_\Omega m < 0. 
\end{align}
Linearizing \eqref{rp:m} at $u=0$ we obtain
\begin{align} \label{lep:m} 
\begin{cases}
-\Delta \varphi = \lambda m \epsilon^{q-2} \varphi & \mbox{in} \ \Omega, \\
\frac{\partial \varphi}{\partial \mathbf{n}} = 0 & \mbox{on} \ 
\partial \Omega. 
\end{cases}
\end{align}
%
%
%\label{a:bposi:intbnega}
%
Under \eqref{c:m} this problem has exactly two principal eigenvalues 
$\lambda=0$ and $\lambda=\lambda_{m, \epsilon} > 0$, 
which are both simple. We denote by $\varphi_{m, \epsilon}$ a positive eigenfunction associated to $\lambda_{m, \epsilon}$ which is normalized as $\Vert \varphi_{m, \epsilon} \Vert_{C(\overline{\Omega})} = 1$. 
Note that $\varphi_{m, \epsilon} > 0$ on $\overline{\Omega}$. \\

We state now the main result of this section for \eqref{rp:m}:

%
%                                 Theorem 
%

\begin{thm} \label{thm:rprob:r}
Let $1<q<2<p$ and $0<\epsilon \leq 1$. Assume \eqref{c:m}. Then \eqref{rp:m} possesses exactly two bifurcation points $(0, 0)$, $(\lambda_{m, \epsilon}, 0)$ on $\{ (\lambda, 0) : \lambda \in \R \}$ from which emanate two subcontinua of positive solutions $\mathcal{C}_0=\mathcal{C}_0(m, \epsilon)$, $\mathcal{C}_1=\mathcal{C}_1 (m, \epsilon)$, respectively. Moreover, the following assertions hold: 
\begin{enumerate}

\item Let $Z$ be any complement of $\langle 1 \rangle$ in $C^{2+\alpha}(\overline{\Omega})$. Then the set $\{ (\lambda, u) \}$ of nontrivial solutions of 
\eqref{rp:m} around $(0, 0)$ is parametrized as
$$
(\lambda, u) = ( \mu (s), \ s(1 + z (s))  ).
$$
with $s \in (-s_0, s_0)$, for some $s_0 >0$. Here $\mu: (-s_0, s_0) \to \R$ and $z: (-s_0, s_0) \to Z$ are continuous and satisfy $\mu (0) =z (0)= 0$. So $\mathcal{C}_0$ is described exactly by $\{ (\mu (s), s(1 + z (s))) : s \in [0, s_0) \}$ around $(0,0)$. Furthermore: 
\begin{enumerate}

  \item $\mu (s)$ satisfies 
\begin{align} \label{asser:lams151011}
\lim_{s\to 0} \frac{\mu (s)}{s^{p-2}} = - \epsilon^{2-q} \frac{\int_\Omega a}{\int_\Omega m}; 
\end{align}

  \item  If, in addition, $p>2$ is an integer and $\int_\Omega a = 0$, then 
$\mu (s)$ is analytic at $s=0$, and its derivatives $\mu^{(k)}$ satisfy 
\begin{align} \label{mu:deri}
\mu^{(k)}(0)=0< \mu^{(2p-4)}(0) \quad \text{for } 1\leq k < 2p-4. 
\end{align}

\end{enumerate}

\item Let $W$ be any complement of $\langle \varphi_{m, \epsilon} \rangle$ in 
$C^{2+\alpha}(\overline{\Omega})$. Then the set $\{ (\lambda, u) \}$ of nontrivial solutions of \eqref{rp:m} around $(\lambda_{m, \epsilon}, 0)$ is parametrized as
$$
(\lambda, u) = ( \gamma (s), \ s(\varphi_{m, \epsilon} + w (s))  ), 
$$
with $s \in (-s_0, s_0)$, for some $s_0 >0$. Here $\gamma: (-s_0, s_0) \to \R$ and $w: (-s_0, s_0) \to W$ are continuous and satisfy $\gamma (0) = \lambda_{m, \epsilon}$ and 
$w (0)=0$. So $\mathcal{C}_1$ is described exactly by $\{ (\gamma (s), 
s(\varphi_{m ,\epsilon} + w (s)) : s\in [0, s_0) \}$ around $(\lambda_{m, \epsilon}, 0)$. 
\\ 

\item Regarding the global nature of $\mathcal{C}_0$ and $\mathcal{C}_1$, we have the following: \\

\begin{enumerate}

  \item $\mathcal{C}_0 \cup \mathcal{C}_1$ does not meet $(\lambda, 0)$ 
except $\lambda = 0$ and $\lambda= \lambda_{m, \epsilon}$.

  \item The following alternative holds: either $\mathcal{C}_0 = \{ (\lambda, u) \}$ and $\mathcal{C}_1 = \{ (\lambda, u) \}$ are both unbounded in $\R \times C(\overline{\Omega})$ or they coincide. 

\end{enumerate}

\end{enumerate}
\end{thm}

%%%%%%%%%%%%%%%%%%%%%%%%%%%%%%%%%%%%%%%%%%%%
\subsubsection*{Proof of Theorem \ref{thm:rprob:r}}  

Under \eqref{c:m} we observe that the principal eigenvalues $\lambda=0$ and $\lambda=\lambda_{m, \epsilon}$ both satisfy the {\it transversality condition} of Crandall and Rabinowitz. Hence the standard local bifurcation theory \cite[Theorem 1.7]{CR71} and the unilateral global bifurcation theory \cite[Theorem 1.27]{Ra71} (see also \cite[Theorem 6.4.3]{LG01}) are applicable at $(0,0)$ and $(\lambda_{m, \epsilon}, 0)$. We obtain then two subcontinua $\mathcal{C}_0$, $\mathcal{C}_1$ of positive solutions of \eqref{rp:m} 
emanating from $(0,0)$ and $(\lambda_{m, \epsilon}, 0)$, respectively. Moreover, assertions (1), (2) and (3) are promptly verified, except \eqref{asser:lams151011} and \eqref{mu:deri}.

Let us show \eqref{asser:lams151011} in assertion (1)(a). 
From assertion (1) we deduce that
\begin{align*}
\int_\Omega \{ a (s+sz)^{p-1} + \mu m (s+sz+\epsilon)^{q-2}(s+sz) \} = 0. 
\end{align*}
It follows that 
\begin{align*} 
\frac{\mu (s)}{s^{p-2}} = -\frac{\int_\Omega a (1+z)^{p-2}}{\int_\Omega m (s+sz+\epsilon)^{q-2}(1+z)} \longrightarrow -\frac{\int_\Omega a}{\int_\Omega m \epsilon^{q-2}}, \quad \text{as } s\to 0^+, 
\end{align*}
as desired. 

Next we verify \eqref{mu:deri} in assertion (1)(b). Following the Lyapunov-Schmidt method, we reduce \eqref{rp:m} to a bifurcation equation around the origin in $\R^2$. Let $w= Q u = u - \frac{1}{\Omega}\int_\Omega u$, where $Q$ is defined as a linear mapping from $L^2(\Omega)$ to $\{ w \in L^2(\Omega) : \int_\Omega w = 0 \}$. We also write $t = \frac{1}{|\Omega|} \int_\Omega u$, so that $u = t + w$. Using $Q$ we decompose \eqref{rp:m} orthogonally in the following way: for $|\lambda|< \lambda^*$ and $u \in U$, a small neighborhood of $0$ in $\mathcal{C}^{2+\alpha}(\overline{\Omega})$, we have
\begin{align}
& Q(-\Delta u) = Q(au^{p-1} + \lambda m (u+\epsilon)^{q-2}u), \label{bifeq1}\\
& (1-Q)(-\Delta u) = (1-Q)(au^{p-1} + \lambda m (u+\epsilon)^{q-2}u). \label{bifeq2}
\end{align}
By applying the implicit function theorem we see that \eqref{bifeq1} is uniquely solvable at $(\lambda, t, w) = (0,0,0)$ by some $w = w(\lambda, t)$ which is analytic at $(0,0)$ and satisfies $w(\lambda, 0) = 0$ for $\lambda$ sufficiently small. We plug $u=t+w(\lambda, t)$ in \eqref{bifeq2}, to obtain the following bifurcation equation around the origin in $\R^2$:
\begin{align} \label{151012Phi}
\Phi (\lambda, t) := 
\int_\Omega a (t+w(\lambda, t))^{p-1} + \lambda \int_\Omega m (t+w(\lambda, t) + \epsilon)^{q-2}(t+w(\lambda, t)) = 0. 
\end{align}
Note that $\Phi$ is also analytic at $(0,0)$. 

We shall now analyse $\Phi$ in \eqref{151012Phi} for $(\lambda, t)$ around $(0,0)$ using its Taylor series expansion. As a preliminary, we compute the partial derivatives of $w(\lambda, t)$ at $(0,0)$:

\begin{lem} \label{lem:w151011}
\strut
\begin{enumerate}
  \item $\frac{\partial^k w}{\partial \lambda^k}(0,0) = 0$ for every 
$k\geq 0$. 
  \item $\frac{\partial w}{\partial t}(0,0)=0$. 
  \item $\frac{\partial^2 w}{\partial t \partial \lambda}(0,0) (= \frac{\partial^2 w}{\partial \lambda \partial t}(0,0))$ is a unique solution of the problem 
\begin{align*}
\begin{cases}
-\Delta w = \epsilon^{q-2}Q[m] & \mbox{in} \ \Omega, \\
\frac{\partial w}{\partial \mathbf{n}} = 0 & \mbox{on} \ \partial \Omega, \\
\int_\Omega w = 0. & 
\end{cases}
\end{align*}
  \item For every integer $k\geq 2$ we have
$$
\frac{\partial^k w}{\partial t^k}(0,0)= \left\{ 
\begin{array}{ll}
w_{k,a}, & p=k+1, \\
0, & p>k+1, 
\end{array} \right.
$$ 
where $w_{k,a}$ is the unique solution of the problem 
\begin{align*}
\begin{cases}
-\Delta w = (k!)a & \mbox{in} \ \Omega, \\
\frac{\partial w}{\partial \mathbf{n}} = 0 & \mbox{on} \ \partial \Omega, \\
\int_\Omega w = 0. & 
\end{cases}
\end{align*}
\end{enumerate}
\end{lem}

\begin{rem}  \label{151014rem:w} {\rm 
%\strut
%\begin{enumerate}
%\item The proof of Lemma \ref{lem:w151011} will be provided in the appendix (see Subsection \ref{subsec:proofLw}). 
%
%\item 
Lemma \ref{lem:w151011}(2)(4) tells us that 
\begin{align*}
\begin{cases}
\frac{\partial^j w}{\partial t^j}(0,0) = 0, & 1\leq j < p-1, \\
\frac{\partial^{p-1} w}{\partial t^{p-1}}(0,0) = w_{p-1, a}. & 
\end{cases}
\end{align*}
%\end{enumerate}
}\end{rem}

% Proof of Lemma \ref{lem:w151011}

\begin{proof}[Proof of Lemma \ref{lem:w151011}]
We denote the partial derivatives of $w$ simply by $w_\lambda, w_t, w_{\lambda \lambda}, w_{tt}, w_{\lambda t}$.
\strut
\begin{enumerate}
  \item By the uniqueness ensured by the implicit function theorem, we see that $w(\lambda, 0)=0$ for $\lambda$ close to $0$. This provides assertion (1). \\
  \item Note that $w = w(\lambda, t)$ satisfies 
\begin{align*}
\begin{cases}
-\Delta w = Q\left[ a(t+w)^{p-1} + \lambda m (t+w+\epsilon)^{q-2}(t+w) \right] 
& \mbox{in $\Omega$}, \\
\frac{\partial w}{\partial \mathbf{n}} = 0 & \mbox{on $\partial \Omega$}. 
\end{cases}
\end{align*}
Differentiating this problem with respect to $t$ we obtain
\begin{align} \label{151012wt}
\begin{cases}
-\Delta w_t = Q[ a(p-1)(t+w)^{p-2}(1+w_t)  & \\ 
\hspace*{2cm} + \lambda m \{ (q-2)(t+w+\epsilon)^{q-3}(1+w_t)(t+w) & \\ 
\hspace*{2.5cm} + (t+w+\epsilon)^{q-2}(1+w_t) \} ] & \mbox{in $\Omega$}, \\
\frac{\partial w_t}{\partial \mathbf{n}} = 0 & \mbox{on $\partial \Omega$}. 
\end{cases}
\end{align}
Putting $(\lambda, t)=(0,0)$ here we deduce 
\begin{align*}
-\Delta w_t(0,0) = 0 \quad\mbox{in $\Omega$}, \qquad 
\int_\Omega w_t = 0,
\end{align*}
which yields Assertion (2).\\
  \item Differentiating \eqref{151012wt} with respect to $\lambda$ we get
\begin{align} \label{151012wtlam} 
\begin{cases}
-\Delta w_{t \lambda} = Q[ a(p-1)\{ (p-2)(t+w)^{p-3}w_\lambda (1+w_t) + 
(t+w)^{p-2}w_{t \lambda} \} & \\ 
+ m \{ (q-2)(t+w+\epsilon)^{q-3}(1+w_t)(t+w) 
+ (t+w+\epsilon)^{q-2}(1+w_t) \} & \\ 
+ \lambda m \{ (q-2)(q-3)(t+w+\epsilon)^{q-4} w_\lambda (1+w_t) (t+w) & \\
\hspace*{1cm} + (q-2) (t+w+\epsilon)^{q-3} w_{t \lambda} (t+w) & \\ 
\hspace*{1.3cm} + (q-2) (t+w+\epsilon)^{q-3} (1+w_t)w_\lambda & \\ 
\hspace*{0.7cm} + (q-2)(t+w+\epsilon)^{q-3} w_\lambda (1+w_t) & \\ 
\hspace*{1cm} + (t+w+\epsilon)^{q-2}w_{t \lambda}\}] & \mbox{in $\Omega$}, \\
\frac{\partial w_t}{\partial \mathbf{n}} = 0 & \mbox{on $\partial \Omega$}. 
\end{cases}
\end{align}
Putting $(\lambda, t)=(0,0)$ here we deduce 
\begin{align*}
-\Delta w_{t \lambda}(0,0) = Q[m \epsilon^{q-2}] \quad\mbox{in $\Omega$}, 
\qquad \int_\Omega w_{t \lambda} = 0, 
\end{align*}
which yields assertion (3).\\
  \item We consider $\frac{\partial^k w}{\partial t^k}(0,0)$ for $k\geq 2$. 
We shall differentiate \eqref{151012wt} with respect to $t$ repeatedly and put $(\lambda, t)=(0,0)$ therein. To this end it is enough to consider the first term on the right-hand side of the equation in \eqref{151012wt}, i.e. the term $\eta := Q[a(p-1)(t+w)^{p-2}(1+w_t)]$. For instance, let us discuss the case $k=2$: we consider the derivative 
\begin{align*}
\eta_t = Q[a(p-1)(p-2)(t+w)^{p-3}(1+w_t) + a(p-1)(t+w)^{p-2}w_{tt}]. 
\end{align*} 
It follows that 
\begin{align*}
\eta_t(0,0) = \left\{ \begin{array}{ll}
Q[(p-1)(p-2)a], & \text{ if }p=3, \\
0, &  \text{ if } p>3. 
\end{array} \right. 
\end{align*}
More generally, putting $\eta = Q[a(p-1)(t+w)^{p-2}(1+w_t)]$ we obtain 
\begin{align*}
\frac{\partial^{k-1} \eta}{\partial t^{k-1}}(0,0) = 
\left\{ \begin{array}{ll}
Q[(p-1)(p-2)\cdots (p-k)a], &  \text{ if } p=k+1, \\
0, &  \text{ if } p>k+1. 
\end{array} \right. 
\end{align*}
Since $\int_\Omega a = 0$, it follows that 
\begin{align*}
-\Delta \frac{\partial^k w}{\partial t^k}(0,0) = 
\left\{ \begin{array}{ll}
(p-1)(p-2)\cdots (p-k)a, &  \text{ if }p=k+1, \\
0, &  \text{ if }p>k+1, 
\end{array} \right. 
\end{align*}
which yields assertion (4). 
\end{enumerate}
The proof of Lemma \ref{lem:w151011} is now complete. 
\end{proof}

Using Lemma \ref{lem:w151011}, we obtain the partial derivatives of $\Phi$ at $(0,0)$:

\begin{lem} \label{lem:Phi151011}
\strut
\begin{enumerate}
  \item $\frac{\partial^k \Phi}{\partial \lambda^k}(0,0) = 0$ for every integer $k\geq 0$. 
  \item $\frac{\partial \Phi}{\partial t}(0,0) = \frac{\partial^2 \Phi}{\partial t^2}(0,0) = 0$. 
  \item $\frac{\partial^2 \Phi}{\partial t \partial \lambda}(0,0) = \frac{\partial^2 \Phi}{\partial \lambda \partial t}(0,0) = \epsilon^{q-2}\int_\Omega m < 0$.  \item For every integer $k \geq 2$ we have
\begin{align} \label{1510122k-1}
\frac{\partial^{2k -1}\Phi}{\partial t^{2k-1}}(0,0) = 
\left\{ \begin{array}{ll}
C_k \int_\Omega \left\vert \nabla \frac{\partial^k w}{\partial t^k}(0,0)\right\vert^2, &  \text{ if } p=k+1, \\
0, &  \text{ if } p>k+1 
\end{array} \right. 
\end{align}
for some constant $C_k > 0$, and 
\begin{align} \label{1510122k}
\frac{\partial^{2k}\Phi}{\partial t^{2k}}(0,0) = 0, \quad  \text{ if } p\geq k + 2.
\end{align}
\end{enumerate}
\end{lem}

\begin{rem} \label{151014remPhi}{\rm 
%\strut
%\begin{enumerate}
%\item The proof of Lemma \ref{lem:Phi151011} will be provided in the appendix (see Subsection \ref{subsec:proofLPhi}).  
%
%\item 
Lemma \ref{lem:Phi151011}(4) tells us that 
\begin{align*}
\begin{cases}
\frac{\partial^j \Phi}{\partial t^j}(0,0) = 0, &  \text{ if } 3\leq j< 2p-3, \\ 
\frac{\partial^{2p-3} \Phi}{\partial t^{2p-3}}(0,0)>0. & 
\end{cases}
\end{align*}
%\end{enumerate}
}\end{rem}

% Proof of Lemma \ref{lem:Phi151011}

\begin{proof}[Proof of Lemma \ref{lem:Phi151011}]
We denote by $\Phi_\lambda, \Phi_t, \Phi_{\lambda \lambda}, \Phi_{\lambda t}$ the derivatives of $\Phi$. 
\strut
\begin{enumerate}
  \item It is straightforward from assertion (1) in Lemma \ref{lem:w151011}. 
  \item Differentiating \eqref{151012Phi} with respect to $t$ we obtain
\begin{align} 
\Phi_t &= \int_\Omega a(p-1)(t+w)^{p-2}(1+w_t) \nonumber \\ 
& + \lambda \int_\Omega m \{ (q-2)(t+w+\epsilon)^{q-3}(1+w_t)(t+w) 
+ (t+w+\epsilon)^{q-2}(1+w_t) \} \label{151012Phit}
\end{align}
It follows that $\Phi_t(0,0) = 0$. Once again we differentiate \eqref{151012Phit} 
with respect to $t$ to obtain
\begin{align} 
\Phi_{tt} &= \int_\Omega a(p-1)\{ (p-2)(t+w)^{p-3}(1+w_t)^2 + (t+w)^{p-2}w_{tt}\} \nonumber \\ 
& + \lambda h(\lambda, t), 
\label{151012Phitt}
\end{align}
where $h$ is bounded. It follows that $\Phi_{tt} (0,0) = 0$ when $p>3$. In addition, this remains true if $p=3$ since $\int_\Omega a = 0$. Assertion (2) is then proved.\\
  \item Differentiating \eqref{151012Phit} with respect to $\lambda$ we get
\begin{align*}
\Phi_{t \lambda} 
& = \int_\Omega a(p-1)\{ (p-2)(t+w)^{p-3}w_\lambda (1+w_t) 
+ (t+w)^{p-2}w_{t \lambda} \} \\ 
&+ \int_\Omega m \{ (q-2)(t+w+\epsilon)^{q-3}(1+w_t)(t+w) 
+ (t+w+\epsilon)^{q-2}(1+w_t) \} \\
& + \lambda \int_\Omega m \{ (q-2)(q-3)(t+w+\epsilon)^{q-4}w_\lambda (1+w_t)(t+w) \\
& \hspace*{1.5cm} + (q-2)(t+w+\epsilon)^{q-3}w_{t \lambda}(t+w) \\
& \hspace*{1.7cm} + (q-2)(t+w+\epsilon)^{q-3}(1+w_t)w_\lambda \\
& \hspace*{1.3cm} + (q-2)(t+w+\epsilon)^{q-3}w_\lambda (1+w_t) \\
& \hspace*{1.5cm} + (t+w+\epsilon)^{q-2}w_{t \lambda} \}. 
\end{align*}
Putting $(\lambda, t)=(0,0)$ here, it follows that 
\begin{align*}
\Phi_{t \lambda}(0,0) = \int_\Omega m \epsilon^{q-2}, 
\end{align*}
which yields Assertion (3).\\
  \item In \eqref{151012Phitt} we put
\begin{align*}
\zeta (\lambda, t) = \int_\Omega a(p-1)\{ (p-2)(t+w)^{p-3}(1+w_t)^2 + (t+w)^{p-2}w_{tt}\} 
%\int_\Omega a(p-1)(t+w)^{p-2}(1+w_t). 
\end{align*}
Then we deduce from \eqref{151012Phitt} that 
%for $k\geq 3$ 
\begin{align}  \label{151014zetadef}
\frac{\partial^j \Phi}{\partial t^j} 
= \frac{\partial^{j-2} \zeta}{\partial t^{j-2}} + \lambda H_j(\lambda, t), 
\end{align}
where $H_j$ is bounded. Let us verify \eqref{1510122k-1} in the case $k=2$: we consider $j=3$ in \eqref{151014zetadef} and observe that
\begin{align} 
\zeta_t &= 
\int_\Omega a(p-1)\left\{ (p-2)(p-3)(t+w)^{p-4}(1+w_t)^3 \right. \nonumber \\
& \left. \hspace*{2.3cm} + 3(p-2)(t+w)^{p-3}(1+w_t)w_{tt} + (t+w)^{p-2}w_{ttt} \right\}. 
\label{151012zetat}
\end{align}
Here it is understood that $(t+w)^\ell = 0$ if $\ell < 0$. Putting $(\lambda, t)=(0,0)$ here, it follows from Lemma \ref{lem:w151011}(4) that 
\begin{align*}
\Phi_{ttt}(0,0) = \zeta_t(0,0) = \left\{ 
\begin{array}{ll}
6\int_\Omega aw_{tt}(0,0) = 3 \int_\Omega |\nabla w_{tt}(0,0)|^2, & \text{ if }
p=3, \\ 
0, & \text{ if } p>3. 
\end{array} \right. 
\end{align*}
Here we have used that $\int_\Omega a = 0$ and $(t+w)^\ell=0$ at $(\lambda, t)=(0,0)$ for some integer $\ell\geq 1$. Thus Assertion \eqref{1510122k-1} with $k=2$ has been verified. 

We prove now \eqref{1510122k} in the case $k=2$: we consider $j=4$ in \eqref{151014zetadef} and differentiate \eqref{151012zetat} with respect to $t$ once more to obtain
\begin{align*}
\zeta_{tt} &= \int_\Omega a(p-1)\left\{ 
(p-2)(p-3)(p-4)(t+w)^{p-5}(1+w_t)^4 \right. \\
& \hspace*{2.3cm} + 6(p-2)(p-3)(t+w)^{p-4}(1+w_t)^2w_{tt} \\ 
& \hspace*{2.3cm} + 3(p-2)(t+w)^{p-3}w_{tt}^2 + 4(p-2)(t+w)^{p-3}(1+w_t)w_{ttt} \\ 
& \left. \hspace*{2.3cm} + (t+w)^{p-2}w_{tttt} \right\}
\end{align*}
When $p\geq 4$ we deduce $\zeta_{tt}(0,0)=0$ in view of Remark \ref{151014rem:w}. Hence Assertion \eqref{1510122k} with $k=2$ has been verified. 
In a similar way,  we can verify assertions \eqref{1510122k-1} and \eqref{1510122k} for the general case $k>2$, using the differential chain rule. 
\end{enumerate}
The proof of Lemma \ref{lem:Phi151011} is now complete. 
\end{proof}

We conclude now the verification of \eqref{mu:deri}: from Lemma \ref{lem:Phi151011}(1)-(3) and the fact that $\Phi$ is analytic at $(0,0)$, we deduce that the Taylor series expansion of $\Phi$ at $(0,0)$ is provided by
\begin{align*}
\Phi (\lambda, t) = t \left( \lambda \frac{\partial^2 \Phi}{\partial t \partial \lambda}(0,0) + \Psi (\lambda, t) \right)
\end{align*}
where $\Psi (\lambda, t)$ is a higher order term. We put 
\begin{align*}
\xi (\lambda, t) := \lambda \frac{\partial^2 \Phi}{\partial t \partial \lambda}(0,0) + \Psi (\lambda, t). 
\end{align*}
Note that $\xi (0,0)=0$ and 
\begin{align*}
\frac{\partial \xi}{\partial \lambda}(0,0) = \frac{\partial^2 \Phi}{\partial t \partial \lambda}(0,0) < 0. 
\end{align*}
Hence the implicit function theorem can be applied to deduce that the solution set of $\xi (\lambda, t) = 0$ near $(0,0)$ is explicitly given by a function $\lambda (t)$ satisfying $\lambda (0) = 0$.

We see that $\lambda'(0) = - \frac{\frac{\partial \xi}{\partial t}(0,0)}{\frac{\partial \xi}{\partial \lambda}(0,0)}=0$, since $\frac{\partial \xi}{\partial t}(0,0) = \frac{\partial^2 \Phi}{\partial t^2}(0,0) = 0$ from Lemma \ref{lem:Phi151011}(2). However, since $\frac{\partial^j \xi}{\partial t^j}(0,0) = \frac{\partial^{j+1} \Phi}{\partial t^{j+1}}(0,0)$, Remark \ref{151014remPhi} provides that 
\begin{align*}
\begin{cases}
\frac{\partial^j \xi}{\partial t^j}(0,0)=0 \text{ if }2\leq j < 2p-4, \\
\frac{\partial^{2p-4} \xi}{\partial t^{2p-4}}(0,0) > 0. & 
\end{cases}
\end{align*}
This implies that 
\begin{align*}
\begin{cases}
\lambda^{(j)}(0)=0  \text{ if } 1\leq j < 2p-4, \\
\lambda^{(2p-4)}(0)= - \frac{\frac{\partial^{2p-4} \xi}{\partial t^{2p-4}}(0,0)}{\frac{\partial \xi}{\partial \lambda}(0,0)} > 0. 
\end{cases}
\end{align*}
Assertion \eqref{mu:deri} is now proved. 

The proof of Theorem \ref{thm:rprob:r} is now complete. \qed \\

\begin{rem}
\label{rbif}
{\rm 
As pointed out in Remark \ref{rem:thm01}(1), assertion (1)(b) with $\int_\Omega b > 0$ and assertion (2)(a) in Theorem \ref{t1} hold true without the restriction $p<\frac{2N}{N-2}$, $N>2$. Indeed, this is verified by a rescaling argument for $(P_\lambda)$ with $v = \lambda^{-\frac{1}{p-q}}u$, the Lyapunov-Schmidt reduction of the rescaled problem 
developed in this section, and an application of the implicit function theorem to the reduced problem. More precisely, by the rescaling $v = \lambda^{-\frac{1}{p-q}}u$ with $\lambda > 0$ and $u\geq 0$, we reduce $(P_\lambda)$ to the problem (see \eqref{eqw}). 
\begin{align} \label{prob:resca}
\begin{cases}
-\Delta v = \mu (a v^{p-1} + b v^{q-1}) & \mbox{in $\Omega$}, \\
\frac{\partial v}{\partial \mathbf{n}} = 0 & \mbox{on $\partial \Omega$}
\end{cases}
\end{align}
with $\mu = \lambda^{\frac{p-2}{p-q}}$. Employing the Lyapunov-Schmidt method 
with the linear mapping $w = Qv = v - \frac{1}{|\Omega|} \int_\Omega v$ as in \eqref{bifeq1} and \eqref{bifeq2}, we have the following bifurcation equation in $\R^2$:
\begin{align} \label{eq:Phi:mut} 
\begin{cases}
& \!\!\!\! \Phi (\mu, t) = \int_\Omega \left\{ a (t + w(\mu, t))^{p-1} + b (t + w(\mu, t))^{q-1} 
\right\} = 0, \quad (\mu, t) \simeq (0, c^*),  \\
& \!\!\!\! \Phi (0, c^*) = 0.
\end{cases}
\end{align}
Here $t = \frac{1}{|\Omega|} \int_\Omega v$ and $w = w(\mu, t)$ is the unique solution in $C^{2+\alpha}(\overline{\Omega})$ of the following boundary value problem defined in a neighborhood of $(\mu, t, w) = (0, c^*, 0)$:
\begin{align*}
\begin{cases}
-\Delta w = \mu Q\left( a(t+w)^{p-1} + b(t+w)^{q-1} \right) & \mbox{in $\Omega$}, \\
\frac{\partial w}{\partial \mathbf{n}} = 0 & \mbox{on $\partial \Omega$}. 
\end{cases}
\end{align*} 
The existence and uniqueness of $w(\mu, t)$ is ensured by the implicit function theorem. Then we can prove that if $\left( \int_\Omega a \right) \left( \int_\Omega b \right) < 0$ then 
\begin{align*}
\Phi_t(0, c^*) = (p-1) (c^*)^{p-2} \left( \int_\Omega a \right) + (q-1) (c^*)^{q-2} 
\left( \int_\Omega b \right) = (q-p) (c^*)^{q-2} \left( \int_\Omega b \right) \not= 0.  
\end{align*}
Still by the implicit function theorem, there exists a unique solution $t(\mu)$ of \eqref{eq:Phi:mut}, i.e. 
\begin{align*}
\mbox{\eqref{eq:Phi:mut}}  
\ \Longleftrightarrow \ \left\{ \begin{array}{ll} 
& \! \!\!\!\! t = t(\mu) \ \ \mbox{for $\mu \simeq 0$}, \\ 
& \! \!\!\!\! t(0)=c^*.  
\end{array} \right. 
\end{align*}
Hence \eqref{prob:resca}  has a positive solution $\hat{v}_\lambda = v_\mu = t(\mu) + w(\mu, t(\mu))$ with $\mu = \lambda^{\frac{p-2}{p-q}}$ bifurcating to the region $\lambda > 0$ from $\{ (0,c) : \mbox{$c$ is a constant} \}$ at $(0,c^*)$, i.e. satisfying $\hat{v}_0 = c^*$. Moreover, this solution is unique and $\hat{v}_\lambda \to c^*$ in $C^{2+\alpha}(\overline{\Omega})$ as $\lambda \to 0^+$. Finally, going back to $(P_\lambda)$ by the rescaling $u = \lambda^{\frac{1}{p-q}} v$, we have a positive solution $u_\lambda = \lambda^{\frac{1}{p-q}} \hat{v}_\lambda$ of $(P_\lambda)$ and $\lambda^{-\frac{1}{p-q}} u_\lambda \to c^*$ in $C^{2+\alpha}(\overline{\Omega})$ as $\lambda \to 0^+$, as desired. 
}\end{rem}

As a byproduct relying on the uniqueness result for $\hat{v}_\lambda$ in Remark \ref{rbif}, we show that the variational positive solution $u_{1,\lambda}$ of $(P_\lambda)$ given by Proposition \ref{p1} is asymptotically stable for $\lambda > 0$ sufficiently small when $\int_\Omega a < 0 < \int_\Omega b$.

\begin{prop}  \label{p:u1as}
Under the assumptions of Proposition \ref{p1}, if $\int_\Omega a < 0 < \int_\Omega b$ then $u_{1,\lambda}$ is asymptotically stable for $\lambda > 0$ sufficiently small.  
\end{prop}

\begin{proof}
With the aid of the argument in Remark \ref{rbif}, it suffices to show that the positive solution $v_\mu = t(\mu) + w(\mu, t(\mu))$ of \eqref{prob:resca} is asymptotically stable for $\mu > 0$ sufficiently small. To this end, recalling \eqref{eigenvp:ulam}, we investigate the sign of the first eigenvalue $\hat{\gamma} = \hat{\gamma}_{1,\mu}$ of the following eigenvalue problem to discuss the linearized stability of $v_\mu$. 
\begin{align*}
\begin{cases}
-\Delta \phi = \mu \left( (p-1) a v_\mu^{p-2} + (q-1) b v_\mu^{q-2} \right) \phi + \gamma \phi & \mbox{in $\Omega$}, \\
\frac{\partial \phi}{\partial \mathbf{n}} = 0 & \mbox{on $\partial \Omega$}. 
\end{cases}
\end{align*}
By $\hat{\phi} = \hat{\phi}_{1, \mu}$ we denote a positive eigenfunction associated with $\hat{\gamma}_{1,\mu}$ and normalized as $\int_\Omega \hat{\phi}_{1,\mu}^2 = 1$. Note that $\hat{\gamma}_{1, 0} = 0$ and $\hat{\phi}_{1, 0} = |\Omega|^{-\frac{1}{2}}$, and moreover that the mapping $\mu \mapsto \left( \hat{\gamma}_{1,\mu}, \hat{\phi}_{1,\mu} \right)$ is continuous in $\R \times C^{2+\alpha}(\overline{\Omega})$ for $\mu$ close to $0$ by the implicit function theorem. We consider $\int_\Omega (- \Delta v_\mu) \frac{\hat{\phi}^2}{v_\mu}$. By the divergence theorem we have
\begin{align} \label{hatgammbb}
\hat{\gamma} = \left( \int_\Omega \left\vert \frac{\hat{\phi}}{v_\mu} \nabla v_\mu - \nabla \hat{\phi} \right\vert^2 \right) + \mu M_\mu \geq \mu M_\mu, 
\end{align}
where $M_\mu = \int_\Omega \left( -(p-2) a v_\mu^{p-2} + (2-q) b v_\mu^{q-2} \right) \hat{\phi}^2$. Since $v_\mu \to c^* =\left( \frac{\int_\Omega b}{-\int_\Omega a}\right)^{\frac{1}{p-q}}$ in $C(\overline{\Omega})$ as $\mu \to 0$, we deduce that
$$
M_\mu \longrightarrow \frac{(p-q)}{|\Omega|} \left( \int_\Omega b \right)^{\frac{p-2}{p-q}} \left( -\int_\Omega a \right)^{\frac{2-q}{p-q}} > 0 \quad \text{as } \mu \to 0.  
$$
Hence it follows from \eqref{hatgammbb} that $\hat{\gamma} > 0$ for $\mu > 0$ sufficiently small, as desired. 
\end{proof}

%%%%%%%%%%%%%%%%%%%%%%%%%%%%%%%%%%%%%%%%%%%%%%%%%%%%%%
\bigskip
\section{Existence of loop type subcontinua: results and expectations} 
\bigskip

%\subsection{The case $\Omega^b_+ \not= \emptyset$ and $\int_\Omega b < 0$}

Let $\epsilon > 0$, and let $b_\epsilon = b - \epsilon$. 
Then we consider the following regularized problem for $(P_\lambda)$. 
$$ 
\begin{cases}
-\Delta u = \lambda b_\epsilon (x)|u + \epsilon|^{q-2}u+a(x)|u|^{p-2}u & \mbox{in $\Omega$}, \\
\frac{\partial u}{\partial \n} = 0 & \mbox{on $\partial \Omega$}. 
\end{cases} \leqno{(P_{\lambda, \epsilon})}   
$$

First we establish an {\it a priori} bound on $|\lambda|$ for nontrivial non-negative  solutions of $(P_{\lambda, \epsilon})$.

%                              Proposition 

\begin{prop}  \label{p:boundslam}
Assume that there exists a ball $B$ such that $\overline{B} \subset \Omega$ with the condition 
$$ 
a\geq 0,  \ \ a\not\equiv 0 \quad\mbox{and} \ \ b>0 \quad \mbox{on} \ \ \overline{B}. 
$$
Then there exists constants $\overline{\lambda} > 0$ and $\epsilon_0 > 0$ such that 
$(P_{\lambda, \epsilon})$ has no nontrivial non-negative solutions for any $\lambda \geq \overline{\lambda}$ and $\epsilon \in (0, \epsilon_0]$.
\end{prop}

\begin{proof}
The proof is carried out as for Proposition \ref{p:nonex}, with small modifications. Let $\varphi  = \varphi_D(B) \in C^2(\overline{B})$ be a positive eigenfunction associated with the first eigenvalue $\lambda_1 =\lambda_D(B) > 0$ of \eqref{eigenprob:D} with $\Omega$ replaced by $B$.  We extend $\varphi$ to the whole $\Omega$ by setting $\varphi=0$ in $\Omega\setminus \overline{B}$. Then $\varphi \in H^1(\Omega)$.

Let $u \in C^2(\overline{\Omega})$ be a nontrivial non-negative solution of $(P_{\lambda, \epsilon})$. 
Note that $u>0$ in $\overline{\Omega}$. Let $\epsilon_0 > 0$ be such that $b_{\epsilon_0} > 0$ in $\overline{B}$. 
By the divergence theorem, we deduce that $\int_B \nabla \cdot u\nabla \varphi= \int_{\partial B} u \frac{\partial \varphi}{\partial \mathbf{n}} < 0$. It follows that 
\begin{align*}
\int_B \nabla u \nabla \varphi - \lambda_1 \int_B au\varphi < 0. 
\end{align*}
On the other hand, the function $u$ satisfies 
\begin{align*}
\int_\Omega \nabla u \nabla w - \int_\Omega au^{p-1}w - \lambda \int_\Omega 
b_\epsilon (u+\epsilon)^{q-2}uw = 0, \quad \forall w \in H^1(\Omega). 
\end{align*}
Taking $w = \varphi$, we have 
\begin{align*}
\int_B \nabla u \nabla \varphi = \int_B au^{p-1}\varphi + \lambda \int_B b_\epsilon (u+\epsilon)^{q-2}u\varphi. 
\end{align*}
We deduce then that
\begin{align*}
\int_B u^{q-1}\varphi  \left( au^{p-q} + \lambda b_\epsilon \left( \frac{u}{u+\epsilon} \right)^{2-q} - \lambda_1 au^{2-q} \right)< 0. 
\end{align*}
Let us set
\begin{align*}
h(x,s) =  a(x)s^{p-q} + \lambda b_\epsilon (x)\left( \frac{s}{s+\epsilon} \right)^{2-q} - \lambda_1 a(x)s^{2-q}, \quad \text{for } (x,s)\in \overline{B}\times [0, \infty). 
\end{align*} 
For our purpose it suffices to show that there exists $\overline{\lambda}$ 
such that $h\geq 0$ if $\lambda \geq \overline{\lambda}$,  $(x,s)\in \overline{B}\times [0, \infty)$, and $\epsilon \in (0, \epsilon_0]$. Indeed, note that
\begin{align*}
h(x,s) \geq as^{p-q} - \lambda_1 a s^{2-q} = as^{2-q}(s^{p-2} - \lambda_1) \geq 0 
\end{align*}
if $s\geq s_0 := \lambda_1^{\frac{1}{p-2}}$ and $x\in \overline{B}$. Next we observe that
\begin{align*}
\frac{\left( \frac{s}{s+\epsilon}\right)^{2-q}}{s^{2-q}} = \left( \frac{1}{s+\epsilon}\right)^{2-q} \geq \left( \frac{1}{s_0 + \epsilon_0} \right)^{2-q}
\end{align*}
if $0\leq s\leq s_0$ and $\epsilon \in (0, \epsilon_0]$. Hence, if $0\leq s\leq s_0$, $x \in \overline{B}$, and $\epsilon \in (0, \epsilon_0]$, then 
\begin{align*}
h \geq \lambda b_\epsilon \left( \frac{s}{s+\epsilon}\right)^{2-q} - \lambda_1 a s^{2-q} 
\geq \left( \lambda \min_{\overline{B}} b_{\epsilon_0} \left( \frac{1}{s_0 + \epsilon_0} \right)^{2-q} - \lambda_1 \Vert a \Vert_{C(\overline{B})} \right) s^{2-q}, 
\end{align*}
So,  if in addition,
\begin{align*}
\lambda \geq \overline{\lambda} := \frac{\lambda_1 \Vert a \Vert_{C(\overline{B})}(s_0 + \epsilon_0)^{2-q}}{\min_{\overline{B}} b_{\epsilon_0}} 
\end{align*}
then $h\geq 0$, as desired. The proof of Proposition \ref{p:boundslam} is now complete. 
\end{proof}

The following result is a direct consequence of Proposition \ref{p:boundslam}:

\begin{cor}  \label{cor:boundslam}
Assume $(H_0)$. Then there exists constant $\overline{\lambda} > 0$ and 
$\epsilon_0 > 0 $ such that $(P_{\lambda, \epsilon})$ has no nontrivial non-negative  solutions for any $|\lambda| \geq \overline{\lambda}$ and $\epsilon \in (0, \epsilon_0]$. 
\end{cor}

\begin{proof}
It suffices to note that if $u$ is a nontrivial non-negative solution of $(P_{\lambda, \epsilon})$ for some $\lambda < 0$, then 
$-\Delta u = a u^{p-1} + (-\lambda)(-b_\epsilon) (u+\epsilon)^{q-2}u$ in $\Omega$.  
\end{proof}

Next we obtain {\it a priori} bounds in $C(\overline{\Omega})$ for nontrivial non-negative solutions of $(P_{\lambda, \epsilon})$. We recall that 
$$
\Omega^a_{\pm} = \{ x \in \Omega : a\gtrless 0 \}, \quad \Omega^a_0 = \{ x \in \Omega : a = 0 \}, \quad \Omega^b_{\pm} = \{ x \in \Omega : b\gtrless 0 \}. 
$$
%%
%\quad \Omega^a_0 = \{ x \in \Omega : a = 0 \}$$ and 
%$$,  \quad \Omega^b_0 = \{ x \in \Omega : b = 0 \}.$$ 
We assume that $\Omega^a_{\pm}$ are both subdomains of $\Omega$ with smooth boundary and satisfy $(H_1)$, i.e.
\begin{align*}
% \label{A+A-:assump} 
\overline{\Omega^a_+} \subset \Omega, \quad \overline{\Omega^a_+} \cup \Omega^a_- = \Omega. 
\end{align*}
%and also assume 
%
%\begin{align} \label{assump:bequiv0}
%\Omega^a_- \subset \Omega^b_0. 
%\end{align}

\begin{prop}  \label{lem:bound:norm}
Assume $(H_1)$ and let $\Lambda > 0$. Suppose there exists a constant $C_1 > 0$ such that $\Vert u \Vert_{\mathcal{C}(\overline{\Omega^a_+})} \leq C_1$ for all nontrivial non-negative solutions $u$ of $(P_{\lambda, \epsilon})$ with $\lambda \in [0, \Lambda]$ and $\epsilon \in (0, 1]$. Then there exists $C_2>0$ such that $\Vert u \Vert_{C(\overline{\Omega})} \leq C_2$ for all nontrivial non-negative solutions $u$ of $(P_{\lambda, \epsilon})$ with $\lambda \in [0, \Lambda]$ and $\epsilon \in (0, 1]$. 
\end{prop}

\begin{proof}
The argument relies on the use of the comparison principle for a concave problem: consider the problem
\begin{align} \label{p:a-}
\begin{cases}
-\Delta v = -a^- v^{p-1} + \lambda b^+ (v+\epsilon)^{q-2}v & \mbox{in $\Omega^a_-$}, \\
v =  C_1 & \mbox{on $\partial \Omega^a_+$}, \\
\frac{\partial v}{\partial \mathbf{n}} = 0 & \mbox{on $\partial \Omega$}. 
\end{cases}
\end{align}
Let $u$ be a nontrivial non-negative solution of $(P_{\lambda, \epsilon})$ with $\lambda\in [0, \Lambda]$ and $\epsilon \in (0, 1]$. The strong maximum principle and boundary point lemma ensure $u>0$ in $\overline{\Omega^a_-}$.   Since $u\leq C_1$ on $\partial \Omega^a_-$ from the assumption and $b_\epsilon < b \leq b^+$, $u$ is a subsolution of \eqref{p:a-}, that is,  
\begin{align*}
\begin{cases} 
-\Delta u \leq -a^- u^{p-1} + \lambda b^+ (u+\epsilon)^{q-2}u & \mbox{in $\Omega^a_-$}, \\
u \leq C_1 & \mbox{on $\partial \Omega^a_+$}, \\
\frac{\partial u}{\partial \mathbf{n}} = 0 & \mbox{on $\partial \Omega$}.
\end{cases}
\end{align*}
We next construct a supersolution of  \eqref{p:a-}. Consider the unique positive solution $w_0$ of the problem 
$$
\begin{cases}
-\Delta w = 1 & \mbox{in $\Omega^a_-$}, \\
w=0 & \mbox{on $\partial \Omega^a_+$}, \\
\frac{\partial w}{\partial \mathbf{n}} = 0 & \mbox{on $\partial \Omega$}.
\end{cases}
$$
Set $\overline{w} = C(w_0+1)$, $C>0$. Then $\overline{w} > C_1$ on $\partial \Omega^a_+$ and $\frac{\partial \overline{w}}{\partial \mathbf{n}} = 0$ on $\partial \Omega$ if $C$ is large. Moreover 
\begin{align*}
-\Delta \overline{w} - \left\{ -a^- \overline{w}^{p-1} + \lambda b^+ (\overline{w} + 
\epsilon)^{q-2}\overline{w} \right\} 
\geq C \left\{ 1 - \lambda b^+ \overline{w}^{q-2} (w_0 + 1) \right\} >0 \quad \mbox{in  $\Omega^a_-$}.  
\end{align*}
Hence $\overline{w}$ is a supersoltuion of \eqref{p:a-}, where $C$ can be chosen independently of $\lambda\in [0,\Lambda]$ and $\epsilon \in (0, 1]$. By using a variant of \cite[Lemma 3.3]{ABC}, we deduce $u\leq \overline{w}$ in $\overline{\Omega^a_-}$, so that $u\leq C_2 := C_1 + \max_{\overline{\Omega^a_-}} \overline{w}$ in $\overline{\Omega}$, as desired. 
The proof of Proposition \ref{lem:bound:norm} is complete. \\ 
\end{proof}

Proposition \ref{lem:bound:norm} can be extended to $\lambda <0$ as follows:

\begin{cor} \label{cor:bound:norm}
Assume $(H_1)$ and let $\Lambda > 0$. Suppose that there exists a constant $C_1 > 0$ such that $\Vert u \Vert_{\mathcal{C}(\overline{\Omega^a_+})} \leq C_1$ for all nontrivial non-negative solutions $u$ of $(P_{\lambda., \epsilon})$ with $|\lambda|\leq \Lambda$ and $\epsilon \in (0, 1]$. Then there exists $C_2>0$ such that $\Vert u \Vert_{C(\overline{\Omega})} \leq C_2$ for all nontrivial non-negative solutions $u$ of $(P_{\lambda, \epsilon})$ with $|\lambda|\leq \Lambda$ and $\epsilon \in (0, 1]$. 
\end{cor}

\begin{proof} 
We discuss the case $-\Lambda \leq \lambda \leq 0$. Note that any nontrivial non-negative solution $u$ of $(P_{\lambda, \epsilon})$ with $\lambda \in [-\Lambda, 0]$ and $\epsilon \in (0, 1]$ satisfies 
\begin{align*}
-\Delta u = a u^{p-1} + (-\lambda) (-b_\epsilon) (u+\epsilon)^{q-2}u \quad \mbox{in $\Omega^a_-$}. 
\end{align*}
with $- \lambda \in [0, \Lambda]$. Instead of \eqref{p:a-} we consider the following concave problem
\begin{align*}
\begin{cases}
-\Delta v = -a^- v^{p-1} + (-\lambda) (b^- + 1) (v+\epsilon)^{q-2}v & \mbox{in $\Omega^a_-$}, \\
v =  C_1 & \mbox{on $\partial \Omega^a_+$}, \\
\frac{\partial v}{\partial \mathbf{n}} = 0 & \mbox{on $\partial \Omega$}. 
\end{cases}
\end{align*}
Then $u$ is a subsolution of this problem. 
The rest of the argument is the same as in the proof of Proposition \ref{lem:bound:norm}. 
\end{proof}

Based on Corollary \ref{cor:bound:norm}, we use an argument from Amann and Lopez-Gomez \cite{ALG} (see also Section 6 of  L\'opez-G\'omez, Molina-Meyer and Tellini \cite{LGMMT13}) to obtain {\it a priori} bounds in 
$C(\overline{\Omega})$ for positive solutions of $(P_{\lambda, \epsilon})$:

\begin{prop} \label{prop:bounds:norm2}
Assume $(H_1)$ and $(H_2)$. Then,
%Then the following two assertions hold:
%\begin{enumerate}
%
%
%  \item 
%If we additionally assume \eqref{assump:bequiv0} then, 
%For any $\Lambda > 0$, there exists $C_\Lambda >0$ such that $\Vert u \Vert_{C(\overline{\Omega})} \leq C_\Lambda$ for all nontrivial non-negative solutions of $(P_{\lambda, \epsilon})$ with $\lambda \in [0, \Lambda]$ and $\epsilon \in [0,1]$. 
%
% \item 
%If we additionally assume \eqref{assump:bequiv00} then, 
for any $\Lambda > 0$ there exists $C_\Lambda >0$ such that $\Vert u \Vert_{C(\overline{\Omega})} \leq C_\Lambda$ for all nontrivial non-negative solutions of $(P_{\lambda, \epsilon})$ with $\lambda \in [-\Lambda, \Lambda]$ and $\epsilon \in (0, 1]$. 
%
%\end{enumerate}
\end{prop}

Now we assume \eqref{a:b}. Note that $b_\epsilon$ satisfies \eqref{c:m} with $m = b_\epsilon$. Then, by applying Theorem \ref{thm:rprob:r} with $m = b_\epsilon$, $(P_{\lambda, \epsilon}) $ possesses exactly two bifurcation points $(0,0)$ and $(\lambda_\epsilon, 0)$, where $\lambda_\epsilon = \lambda_{b_\epsilon, \epsilon}$, from which there bifurcate subcontinua $\mathcal{C}_0 (\epsilon) = \mathcal{C}_0(b_\epsilon, \epsilon)$ and $\mathcal{C}_1 (\epsilon) = \mathcal{C}_1(b_\epsilon, \epsilon)$ of positive solutions, respectively, and $\mathcal{C}_0(\epsilon)$ and $\mathcal{C}_1(\epsilon)$ satisfy assertions (1)-(3) in Theorem \ref{thm:rprob:r} with $m=b_\epsilon$. Moreover, the bifurcation point $(\lambda_\epsilon, 0)$ tends to $(0,0)$:

\begin{lem} \label{lem:lamepto0cc}
$\displaystyle{\lim_{\epsilon \to 0^+} \lambda_\epsilon = 0}$. 
\end{lem}

\begin{proof}
We consider the eigenvalue problem
\begin{align}  \label{epr:ep:phi}
\begin{cases}
-\Delta \phi = \lambda b_\epsilon \epsilon^{q-2} \phi + \mu(\lambda) \phi & \mbox{in $\Omega$}, \\
\frac{\partial \phi}{\partial \mathbf{n}} = 0 & \mbox{on $\partial \Omega$}. 
\end{cases}
\end{align}
This problem has the smallest eigenvalue  $\mu_{1, \epsilon}(\lambda)$ which satisfies 
\begin{align*}
\begin{cases}
\mu_{1, \epsilon}(\lambda) = 0, & \text{ for } \lambda = 0, \lambda_\epsilon, \\
\mu_{1,\epsilon}(\lambda) > 0, & \text{ for } 0<\lambda< \lambda_\epsilon, \\
\mu_{1, \epsilon}(\lambda) < 0, & \text{ for } \lambda > \lambda_\epsilon. 
\end{cases}
\end{align*}

First we consider the case $\int_\Omega b < 0$. Since $\Omega^b_+ \ne \emptyset$, the eigenvalue problem 
\begin{align*}
\begin{cases}
-\Delta \varphi = \lambda b \varphi & \mbox{in $\Omega$}, \\
\frac{\partial \varphi}{\partial \mathbf{n}} = 0 & \mbox{on $\partial \Omega$}, 
\end{cases}
\end{align*}
has a unique positive prinicipal eigenvalue $\lambda_1(b)$ with a positive eigenfunction $\varphi_1(b)$. It follows that $\int_\Omega b \varphi_1(b)^2 = \frac{1}{\lambda_1(b)} \int_\Omega |\nabla \varphi_1(b)|^2 > 0$. Then there exist $\epsilon_1 > 0$ and $c_1 > 0$ such that $\int_\Omega b_\epsilon \varphi_1(b)^2 > c_1$ if $0<\epsilon \leq \epsilon_1$. 
Let $\lambda > 0$. Then we deduce 
for such $\epsilon$
\begin{align*}
\int_\Omega |\nabla \varphi_1(b)|^2 - \lambda \epsilon^{2-q} \int_\Omega b_\epsilon  \varphi_1(b)^2 < \int_\Omega |\nabla \varphi_1(b)|^2 - \lambda \epsilon^{2-q} c_1. 
\end{align*}
Let $\epsilon_2 = \left( \frac{\int_\Omega |\nabla \varphi_1(b)|^2}{\lambda c_1}\right)^{\frac{1}{2-q}}$. If $0<\epsilon \leq \min (\epsilon_1, \epsilon_2)$ then 
\begin{align*}
\int_\Omega |\nabla \varphi_1(b)|^2 - \lambda \epsilon^{2-q} \int_\Omega b_\epsilon \varphi_1(b)^2 < 0. 
\end{align*}
This implies $\mu_{1, \epsilon}(\lambda)< 0$, and hence, $\lambda_\epsilon < \lambda$, as desired.

Next we consider the case $\int_\Omega b = 0$. In this case, the eigenvalue problem 
\begin{align*}
\begin{cases}
-\Delta \phi = \lambda b \phi + \mu (\lambda) \phi & \mbox{in $\Omega$}, \\
\frac{\partial \phi}{\partial \mathbf{n}} = 0 & \mbox{on $\partial \Omega$} 
\end{cases}
\end{align*}
has the smallest eigenvalue $\mu_1(\lambda)$, which is negative for every $\lambda > 0$. 
Let $\lambda > 0$. Since $\mu_1(\lambda) < 0$, we can choose $\phi$ such that 
$\int_\Omega |\nabla \phi|^2 - \lambda \int_\Omega b \phi^2 < 0$. 
Note that $\phi$ is not a constant, so that $\int_\Omega b \phi^2 > \frac{1}{\lambda} \int_\Omega |\nabla \phi|^2 > 0$. 
Then there exist $\epsilon_1 > 0$ and $c_1 > 0$ such that if $0<\epsilon \leq \epsilon_1$ then $\int_\Omega b_\epsilon \phi^2 > c_1$. The rest of the proof in this case is the same as in the previous case. The proof of Lemma \ref{lem:lamepto0cc} is complete. 
\end{proof}

Now Corollary \ref{cor:boundslam} and Proposition \ref{prop:bounds:norm2} provide sufficient conditions under which $\mathcal{C}_0(\epsilon)$ and $\mathcal{C}_1(\epsilon)$ are bounded in $\R \times C(\overline{\Omega})$, and consequently coincide:

\begin{thm} \label{t:ep:global}
Let $1<q<2<p$ and $\epsilon > 0$. Assume \eqref{a:b}, $(H_0)$, $(H_1)$ and $(H_2)$, and in addition, $0<\epsilon \leq \epsilon_0$, where $\epsilon_0$ is given by Corollary \ref{cor:boundslam}. 
Then the subcontinua $\mathcal{C}_0 (\epsilon)$ and $\mathcal{C}_1 (\epsilon)$ obtained in Theorem \ref{thm:rprob:r} with $m = b_\epsilon$ are bounded in $\R\times C(\overline{\Omega})$, uniformly in $\epsilon > 0$ small. 
Consequently, $\mathcal{C}_0 (\epsilon) = \mathcal{C}_1 (\epsilon)$ (see Figure  \ref{fig15_111402}). 
\end{thm}

\begin{proof} 
By Corollary \ref{cor:boundslam} and Proposition \ref{prop:bounds:norm2} we know that if $u$ is a nontrivial non-negative solution of $(P_{\lambda, \epsilon})$ then $|\lambda| \leq \overline{\lambda}$  and $\|u\|_{C(\overline{\Omega})} \leq C_{\overline{\lambda}}$ for some $\overline{\lambda}$ and $C_{\overline{\lambda}}$. Hence the conclusion follows from Theorem \ref{thm:rprob:r}(3). 
\end{proof}

From now on we write $\mathcal{C}_*(\epsilon)=\mathcal{C}_0 (\epsilon) = \mathcal{C}_1 (\epsilon)$. As a by-product of Theorem \ref{t:ep:global}, we determine the direction of the bifurcation $\mathcal{C}_0(\epsilon)$ at $(0,0)$ if $\int_\Omega a \geq 0$, which partially complements \eqref{asser:lams151011} and \eqref{mu:deri}. To this end we use the following lemma:

\begin{lem} \label{lem:lameq0} 
The following two assertions hold:
\begin{enumerate}

  \item If $\int_\Omega a \geq 0$ then there is no positive solution of \eqref{pla}. \\

  \item Assume $2<p<\frac{2N}{N-2}$ if $N>2$. If $\int_\Omega a < 0$ then there exists $C>0$ such that $\Vert u \Vert_{C(\overline{\Omega})} \geq C$ for all positive solutions of \eqref{pla}.

\end{enumerate}
\end{lem}

\begin{proof}
\strut
\begin{enumerate}
\item If $u$ is a positive solution of \eqref{pla} then 
\begin{align*}
\int_\Omega a = \int_\Omega \frac{-\Delta u}{u^{p-1}} = \int_\Omega |\nabla u|^2 (1-p) u^{-p} < 0. 
\end{align*}\\

\item  Assume by contradiction that $(u_n)$ is a sequence of positive solutions of \eqref{pla} such that $u_n \to 0$ in $C(\overline{\Omega})$ . It follows that $u_n \to 0$ in $X$, since $u_n$ is a positive solution of \eqref{pla}. We set $v_n = \frac{u_n}{\Vert u_n \Vert}$ and assume that $v_n \rightharpoonup v_0$ in $X$, $v_n \to v_0$ in $L^p(\Omega)$, and $v_n \to v_0$ a.e. in $\Omega$, for some $v_0 \in H^1(\Omega)$. We deduce that
\begin{align*}
 \int_\Omega |\nabla v_0|^2 \leq \liminf_{n\to \infty} \int_\Omega |\nabla v_n|^2  \leq \limsup_{n\to \infty} \int_\Omega |\nabla v_n|^2 = \lim_{n\to \infty} \| u_n \|^{p-2} \int_\Omega av_n^p = 0. 
\end{align*}
It follows that $v_n \to v_0$ in $X$ and $v_0$ is a non-negative constant. Since $\Vert v_n \Vert = 1$, we deduce that $v_0 \not= 0$. However, $$\int_\Omega a v_n^p = \Vert u_n \Vert^{2-p} \int_\Omega |\nabla v_n|^p \geq 0.$$ Passing to the limit, we get $v_0^p \int_\Omega a \geq 0$. Hence $\int_\Omega a \geq 0$, which contradicts our assumption.

\end{enumerate} 
\end{proof}

\begin{cor} \label{cor:intaposi:dir}
In addition to the conditions of Theorem \ref{t:ep:global},  assume that $\int_\Omega a \geq 0$. Then $\mathcal{C}_0(\epsilon)$ bifurcates to the region $\lambda > 0$ at $(0,0)$, that is, $\mu (s) > 0$ for $s>0$ small in Theorem \ref{thm:rprob:r}(1) with $m=b_\epsilon$. 
%meaning that if there exist $\lambda_n \to 0$ and positive solutions $u_n$ of $(P_{\lambda_n, \epsilon})$ such that $\Vert u_n \Vert_{C(\overline{\Omega})} \to 0$ then $\lambda_n > 0$ for $n$ large enough. 
\end{cor}

\begin{proof}
We argue by contradiction:  in view of Lemma \ref{lem:lameq0}(1) we may assume $\lambda_n < 0$ and a positive solution $u_n$ of $(P_{\lambda_n, \epsilon})$ such that 
$\lambda_n \to 0^-$ and $\Vert u_n \Vert_{C(\overline{\Omega})} \to 0$. 
Theorem \ref{thm:rprob:r}(1) shows that $(\lambda_n, u_n) \in \mathcal{C}_*(\epsilon) (= \mathcal{C}_0(\epsilon))$. Hence Theorem \ref{t:ep:global} ensures that $\mathcal{C}_*(\epsilon)$ should contain $(0, u)$ for some $u\not= 0$. However, this contradicts Lemma \ref{lem:lameq0}(1). 
\end{proof}

%We turn now the limit status of $\mathcal{C}_*(\epsilon)$ as $\epsilon \to 0^+$, where {\it a priori} bounds from below for positive solutions of $(P_{\lambda, \epsilon})$ with $\lambda = 0$ play an important role.

Let us show now that under $(H_3)$ bifurcation from zero can not occur for $(P_\lambda)$ at any $\lambda\not= 0$, i.e. it never occurs that there are $\lambda_n \to \lambda^* \not= 0$ and nontrivial non-negative solutions $u_n$ of $(P_{\lambda_n})$ such that $ u_n \to 0$ in $C(\overline{\Omega})$.  This result is deduced from Proposition \ref{prop:nobifposi}:

\begin{prop} \label{prop:nobiforip}
Assume $(H_3)$. Then bifurcation from zero never occurs for $(P_\lambda)$ at any $\lambda \not= 0$.  
\end{prop}

\begin{proof}
%	Assume by contradiction that $\lambda_n \to \lambda^*>0$ and $u_n$ is a nontrivial non-negative solution of $(P_{\lambda_n})$ with $u_n \to 0$ in $C(\overline{\Omega})$. By Lemma \ref{l1} we must have $u_n \in B_0^-$ for $n$ large enough. Moreover, we have $u_n \to 0$ in $X$. We set $v_n= \frac{u_n}{\|u_n\|}$ and assume that $v_n \rightharpoonup v_0$ in $X$ and $v_n \to v_0$ in $L^p(\Omega)$ and $L^q(\Omega)$. For $\phi \in X$ we have
%	\begin{equation}
%	\label{en}
%	\int_\Omega \left(\nabla u_n \nabla \phi -au_n^{p-1} \phi - \lambda_n bu_n^{q-1}\phi\right)=0.
%	\end{equation}
%	Hence $\lambda_n \int_\Omega bv_n^{q-1}\phi \to 0$, and consequently
%	$\int_\Omega bv_0^{q-1}\phi=0$. Since this holds for any $\phi \in X$, we deduce that $bv_0 \equiv 0$. Taking $\phi=v_0$ in \eqref{en} we obtain
%	$$\int_\Omega \nabla u_n \nabla v_0= \int_\Omega au_n^{p-1}v_0$$
%	for any $\phi \in X$. It follows that $\int_\Omega \nabla v_n \nabla v_0 \to 0$, so $v_0$ is a constant. Therefore $v_0 \equiv 0$. Now, since $u_n \in B_0^-$ for $n$ large enough and $N_{\lambda_n}^+ \subset B^+$, we have $u_n \in N_{\lambda_n}^- \cup N_{\lambda_n}^0$, i.e. 
%	$$E(u_n) \leq \frac{p-q}{2-q} A(u_n)$$ for $n$ large enough.
%	We infer that $\limsup E(v_n) \leq 0$, so that $v_n \to 0$, which contradicts $\|v_n\|=1$. 
%
%Next we verify the latter assertion. 
%
The assertion for $\lambda > 0$ has been already verified in Proposition \ref{prop:nobifposi}. Assume that $\lambda_n \to \lambda^* < 0$ and  $u_n$ are nontrivial non-negative solutions of $(P_{\lambda_n})$ with  $u_n \to 0$ in $C(\overline{\Omega})$.  Then $(\lambda_n, u_n)$ satisfies 
\begin{align*}
-\Delta u_n = a u_n^{p-1} + (-\lambda_n) (-b) u_n^{q-1} \quad\mbox{in $\Omega$}.\end{align*}
Since $-\lambda_n \to -\lambda^* > 0$ and $\Omega^{-b}_+$ is a subdomain of $\Omega$, Proposition \ref{prop:nobifposi} provides us with a contradiction. The proof of Proposition \ref{prop:nobiforip} is now complete.
\end{proof}
\medskip

%\subsubsection{End of proof of Theorem  \ref{thm:ep0:exist}}
We are now in position to prove Theorem \ref{thm:ep0:exist}.

\begin{proof}[Proof of Theorem \ref{thm:ep0:exist}] 
We use Whyburn's topological method \cite{Wh64}. Let us recall from \cite{Wh64} that if $E_n \subset X$ for a complete metric space $X$, then
\begin{align*}
& \liminf E_n = \{ x \in X : \lim_{n\to \infty}{\rm dist}\, (x, E_n) = 0 \}, 
\\
& \limsup E_n = \{ x \in X : \liminf_{n\to \infty}{\rm dist}\, (x, E_n) = 0 \}. 
\end{align*} 
Note that the values $\overline{\lambda}$ from Corollary \ref{cor:boundslam} and $C_\Lambda$ from Proposition \ref{prop:bounds:norm2} do not depend on the value $\epsilon > 0$ determined in Theorem \ref{t:ep:global}. Hence, we see that for such $\epsilon$ 
\begin{align*}
\mathcal{C}_*(\epsilon) \subset 
\{ (\lambda, u) \in \R \times C(\overline{\Omega}) : |\lambda| \leq \overline{\lambda}+1, \ 0\leq u < C_\Lambda + 1 
\ \mbox{on $\overline{\Omega}$} \}. 
\end{align*}
It's then clear that $\mathcal{C}_*(\epsilon)$ is non-empty and connected, and $(0,0) \in \liminf C_*(\epsilon)$. Moreover, by elliptic regularity,  
$\bigcup_{\epsilon} \mathcal{C}_*(\epsilon)$ is precompact. Indeed, for any $\{ (\lambda_n, u_n) \} \subset \bigcup_{\epsilon} \mathcal{C}_*(\epsilon)$ it follows that $(\lambda_n, u_n) \in \mathcal{C}_*(\epsilon_n)$ for some $\epsilon_n \in (0,1]$, so that $u_n \in C^2(\overline{\Omega})$, and 
\begin{align*}
\begin{cases}
-\Delta u_n = a u_n^{p-1} + \lambda_n b_{\epsilon_n} (u_n + \epsilon_n)^{q-2}u_n & \mbox{in $\Omega$}, \\
\frac{\partial u_n}{\partial \mathbf{n}} = 0 & \mbox{on $\partial \Omega$}.  
\end{cases}
\end{align*}
Since $|\lambda_n|\leq \overline{\lambda} + 1$ and $\Vert u_n \Vert_{C(\overline{\Omega})} \leq C_\Lambda + 1$, by elliptic regularity we find a constant $C > 0$ such that $\Vert u_n \Vert_{\mathcal{C}^{1+\theta}(\overline{\Omega})} \leq C$ for all $n$. By the compact embedding $\mathcal{C}^{1+\theta}(\overline{\Omega}) \subset C(\overline{\Omega})$, a subsequence of $(u_n)$ converges to some $u_0$ in $C(\overline{\Omega})$, as desired.

Now, by \cite[Theorem 9.1]{Wh64}, we deduce that $\mathcal{C} := \limsup_{\epsilon \to 0^+} \mathcal{C}_*(\epsilon)$ is non-empty, closed and connected. In addition, $\mathcal{C}$ is contained in the non-negative solutions set of $(P_\lambda)$. Indeed, let $(\lambda, u) \in \mathcal{C}$. Then we deduce that there exists $(\lambda_n, u_n) \in \mathcal{C}_*(\epsilon_n)$ such that $(\lambda_n, u_n) \to (\lambda, u)$ in $\R \times C(\overline{\Omega})$. Similarly as in the above argument,  elliptic regularity and Schauder estimates yield a constant $C>0$ such that $\Vert u_n \Vert_{\mathcal{C}^{2+\theta}(\overline{\Omega})} \leq C$ for all $n$. By a compactness argument, we deduce that a subsequence of $(u_n)$ converges to $u$ in $C^2(\overline{\Omega})$, so that $u$ is a non-negative  solution of $(P_\lambda)$, as desired. Furthermore, from Lemma \ref{lem:lamepto0cc} we infer that $\mathcal{C}$ is a loop type subcontinuum in the sense that it bifurcates at $(0,0)$ and goes back to $(0,0)$.

It remains to show that $\mathcal{C}$ is nontrivial. Since $\int_\Omega a<0$, assertion \eqref{asser:lams151011} with $m = b_\epsilon$ enables us to deduce that for any $\epsilon > 0$ small  there exists $u_\epsilon \not= 0$ such that $(0, u_\epsilon) \in \mathcal{C}_*(\epsilon)$. Then $u_\epsilon$ is a positive solution of \eqref{pla}. By a standard compactness argument, there exist some sequences $\epsilon_n$ and $u_n := u_{\epsilon_n}$ such that $\epsilon_n \searrow 0$ and $u_n$ converges to a non-negative solution $u_0$ of \eqref{pla} in 
$C(\overline{\Omega})$. By definition, we have $(0, u_0) \in \mathcal{C}$. In addition, we deduce that $u_0$ is positive on $\overline{\Omega}$ thanks to Lemma \ref{lem:lameq0}(2). This implies that $\mathcal{C}$ is nontrivial, i.e. it never shrinks to the $\lambda$ axis. Finally,  by Proposition \ref{prop:nobiforip} we infer that $\mathcal{C}\setminus \{ (0,0)\}$ does not include trivial solutions $(\lambda,0)$ with $\lambda \neq 0$, and by Lemma \ref{lem:lameq0}(2) that there exists $\delta > 0$ such that $\mathcal{C}$ never meets any positive solution $u$ of \eqref{pla} satisfying $\Vert u \Vert_{C(\overline{\Omega})} \leq \delta$. 

The proof of Theorem \ref{thm:ep0:exist} is now complete. 
\end{proof}

\subsection{Remarks and expectations} 

\strut\begin{enumerate}

\item Theorem \ref{thm:ep0:exist} remains true if we replace the condition $\int_\Omega a<0$ by the condition that there exists $\epsilon_n \searrow 0$, $u_n \not=0$ such that $\mathcal{C}_*(\epsilon_n)$ include $(0, u_n)$, respectively. \\

%(2) The direction of the subcontinua at $(0,0)$ appearing in Figure \ref{fig15_1106} is due to  suggestions provided by \eqref{asser:lams151011}, \eqref{asy:2-3} and \eqref{asy:geq3} with $\epsilon\simeq 0$. 

\item  If instead of $\int_\Omega a<0$ we assume now $\int_\Omega a \geq 0$ in Theorem \ref{thm:ep0:exist} then we expect the existence of a loop type subcontinuum of nontrivial non-negative solutions of $(P_\lambda)$ bifurcating at $(0,0)$ as in Figure \ref{fig15_1107a}. Note that, by Corollary \ref{cor:intaposi:dir}, $\mathcal{C}_*(\epsilon)$ bifurcates to the region $\lambda > 0$ at $(0,0)$. In addition, Lemma \ref{lem:lameq0}(1) tells us that $\mathcal{C}_*(\epsilon)$ never meets the vertical axis $\lambda = 0$.  The nontrivial non-negative solutions $u_{1,\lambda}, u_{2,\lambda}$ of $(P_\lambda)$ provided by Theorem \ref{t1} via a variational approach and also the non-existence result in Remark \ref{rem:t2} would strongly support this suggestion. See also Remark \ref{rasyu2} (2) for the case $\int_\Omega a=0$.  However, we couldn't  exclude the possibility that $\mathcal{C}_*(\epsilon)$ shrinks to the origin $(0,0)$ as $\epsilon \to 0^+$. \\

\item Changing $\lambda_\epsilon$ to $-\lambda_\epsilon$ we see that Theorems \ref{t:ep:global} and \ref{thm:ep0:exist} hold true as well under the condition $\Omega^b_- \not= \emptyset$ and $\int_\Omega b > 0$.

\end{enumerate}

%%%%%%%%%%%%  one figure 

%	\begin{figure}[H]
	\begin{figure}[!htb]
		%WinTpicVersion4.28b
{\unitlength 0.1in
\begin{picture}( 19.6000, 21.4000)( 21.3000,-24.9000)
% VECTOR 2 0 3 0 Black White
% 2 2130 2320 4090 2320
% 
\special{pn 8}%
\special{pa 2130 2320}%
\special{pa 4090 2320}%
\special{fp}%
\special{sh 1}%
\special{pa 4090 2320}%
\special{pa 4024 2300}%
\special{pa 4038 2320}%
\special{pa 4024 2340}%
\special{pa 4090 2320}%
\special{fp}%
% VECTOR 2 0 3 0 Black White
% 2 2920 2490 2920 660
% 
\special{pn 8}%
\special{pa 2920 2490}%
\special{pa 2920 660}%
\special{fp}%
\special{sh 1}%
\special{pa 2920 660}%
\special{pa 2900 728}%
\special{pa 2920 714}%
\special{pa 2940 728}%
\special{pa 2920 660}%
\special{fp}%
% STR 2 0 3 0 Black White
% 4 2730 2380 2730 2480 5 0 0 0
% O
\put(27.3000,-24.8000){\makebox(0,0){O}}%
% STR 2 0 3 0 Black White
% 4 4250 2030 4250 2130 5 0 0 0
% $\lambda$
\put(42.5000,-21.3000){\makebox(0,0){$\lambda$}}%
% STR 2 0 3 0 Black White
% 4 2910 330 2910 430 5 0 0 0
% $u$
\put(29.1000,-4.3000){\makebox(0,0){$u$}}%
% SPLINE 0 0 3 0 Black White
% 9 2920 2320 2990 2320 3380 2190 3780 1900 3710 1620 3350 1420 3050 1790 2940 2200 2920 2320
% 
\special{pn 20}%
\special{pa 2920 2320}%
\special{pa 2952 2322}%
\special{pa 3016 2318}%
\special{pa 3046 2314}%
\special{pa 3078 2308}%
\special{pa 3106 2302}%
\special{pa 3196 2272}%
\special{pa 3226 2260}%
\special{pa 3286 2234}%
\special{pa 3348 2206}%
\special{pa 3380 2190}%
\special{pa 3444 2162}%
\special{pa 3510 2132}%
\special{pa 3544 2116}%
\special{pa 3576 2100}%
\special{pa 3606 2084}%
\special{pa 3636 2066}%
\special{pa 3688 2030}%
\special{pa 3712 2008}%
\special{pa 3732 1986}%
\special{pa 3750 1964}%
\special{pa 3766 1940}%
\special{pa 3776 1914}%
\special{pa 3784 1886}%
\special{pa 3788 1856}%
\special{pa 3788 1826}%
\special{pa 3784 1794}%
\special{pa 3778 1762}%
\special{pa 3768 1730}%
\special{pa 3754 1698}%
\special{pa 3738 1666}%
\special{pa 3720 1636}%
\special{pa 3700 1606}%
\special{pa 3678 1576}%
\special{pa 3626 1524}%
\special{pa 3570 1480}%
\special{pa 3540 1462}%
\special{pa 3480 1434}%
\special{pa 3450 1426}%
\special{pa 3420 1420}%
\special{pa 3390 1418}%
\special{pa 3360 1418}%
\special{pa 3332 1424}%
\special{pa 3306 1434}%
\special{pa 3280 1448}%
\special{pa 3256 1464}%
\special{pa 3232 1484}%
\special{pa 3188 1532}%
\special{pa 3168 1560}%
\special{pa 3148 1590}%
\special{pa 3130 1620}%
\special{pa 3096 1684}%
\special{pa 3082 1716}%
\special{pa 3054 1782}%
\special{pa 3018 1878}%
\special{pa 3008 1910}%
\special{pa 2990 1972}%
\special{pa 2974 2032}%
\special{pa 2962 2092}%
\special{pa 2954 2124}%
\special{pa 2950 2154}%
\special{pa 2938 2216}%
\special{pa 2932 2248}%
\special{pa 2928 2278}%
\special{pa 2922 2310}%
\special{pa 2920 2320}%
\special{fp}%
% DOT 0 0 3 0 Black White
% 2 2920 2320 2920 2320
% 
\special{pn 4}%
\special{sh 1}%
\special{ar 2920 2320 16 16 0  6.28318530717959E+0000}%
\special{sh 1}%
\special{ar 2920 2320 16 16 0  6.28318530717959E+0000}%
\end{picture}}%
	  \caption{An expected loop type subcontinuum of $(P_\lambda)$ when $\int_\Omega a \geq 0$, $\Omega^b_+ \not= \emptyset$ and $\int_\Omega b \leq 0$.} 
	\label{fig15_1107a} 
	    \end{figure}
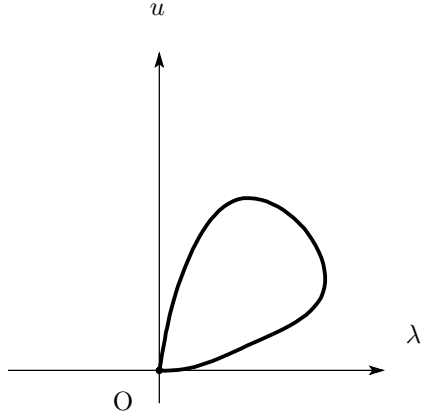 

%%%%%%%%%%%%%%%

\end{document}